%% file: dominant-arxiv-v2.tex
\documentclass[12pt,letterpaper]{article}

	\usepackage{euscript}
	\usepackage{amsthm}
	\usepackage{amssymb,amscd}
	\usepackage{amsmath}
	\usepackage{graphicx}
	\usepackage{color}
	\usepackage{mathtools}

	\usepackage{microtype}

	\usepackage{hyperref}

	\usepackage[T1]{fontenc}
	\usepackage{textcomp}
	\usepackage{lmodern}

	\usepackage[margin=1.2in]{geometry}

	\usepackage[style=alphabetic,  backend=bibtex, doi=false, url=false, firstinits=true, isbn=false]{biblatex}
	\usepackage{todonotes}

	\def\bdef{\begin{definition}}
	\def\endef{\end{definition}}
	\def\bthm{\begin{theorem}}
	\def\ethm{\end{theorem}}
	\def\blm{\begin{lemma}}
	\def\elm{\end{lemma}}
	\def\brm{\begin{remark}}
	\def\erm{\end{remark}}
	\def\bprop{\begin{proposition}}
	\def\eprop{\end{proposition}}
	\def\bcor{\begin{corollary}}
	\def\ecor{\end{corollary}}
	\def\be{\begin{eqnarray}}
	\def\ee{\end{eqnarray}}
	\def\beal{\begin{aligned}}
	\def\enal{\end{aligned}}

	\newtheorem*{remark}{Remark}
	\newtheorem{theorem}{Theorem}[section]
	\newtheorem*{thmstar}{Theorem}
	\newtheorem{proposition}[theorem]{Proposition}
	
	\newtheorem{corollary}[theorem]{Corollary}
	\newtheorem{lemma}[theorem]{Lemma}

	\newtheorem*{definition}{Definition}
	
	\newtheorem*{main}{Main Result}
	\newtheorem*{conjecture}{Conjecture}

	\newcommand{\Z}{\mathbb{Z}}
	\newcommand{\N}{\mathbb{N}}
	\newcommand{\R}{\mathbb{R}}

	\newcommand{\T}{\mathbb{T}}

	\newcommand{\al}{\alpha}
	
	\newcommand{\gm}{\gamma}
	\newcommand{\Gm}{\Gamma}
	\newcommand{\lb}{\lambda}
	
	\newcommand{\eps}{\epsilon}
	\newcommand{\cC}{\mathcal{C}}

	\newcommand{\cA}{\mathcal{A}}
	\newcommand{\cV}{\mathcal{V}}
	\newcommand{\cG}{\mathcal{G}}
	
	\newcommand{\cM}{\mathcal{M}}
	\newcommand{\cL}{\mathcal{L}}
	\newcommand{\cF}{\mathcal{F}}
	\newcommand{\cN}{\mathcal{N}}
	\newcommand{\cB}{\mathcal{B}}
	\newcommand{\cH}{\mathcal{H}}
	\newcommand{\cU}{\mathcal{U}}
	\newcommand{\cI}{\mathcal{I}}
	
	\newcommand{\fC}{\mathfrak{C}}
	\newcommand{\cX}{\mathcal{X}}
	
	\newcommand{\cY}{\mathcal{Y}}
	\newcommand{\cZ}{\mathcal{Z}}

	\newcommand{\Span}{\mathrm{span}}

	\newcommand{\st}{\mathrm{st}}
	\newcommand{\wk}{\mathrm{wk}}
	\newcommand{\lM}{\mu}
	\newcommand{\bcG}{\overline{\mathcal{G}}}

\usepackage{multirow}

\usepackage{mathabx}

\newcommand{\myheading}[1]{%
    \par\medskip\textit{#1}%
}
\newcommand{\breakheading}[1]{%
    \par\medskip\noindent\textit{#1}%
}

\numberwithin{equation}{section}

\newcommand{\dist}{\mathrm{dist}}

\newcommand{\tilphi}{\tilde{\varphi}}
\newcommand{\tilI}{\tilde{I}}

\newcommand{\bphi}{\bar{\phi}}

\newcommand{\diag}{\mathrm{diag}}

\newcommand{\bmat}[1]{\begin{bmatrix}#1\end{bmatrix}}

\newcommand{\Id}{\mathrm{Id}}
\newcommand{\bark}{\bar{k}}
\newcommand{\Sym}{\mathrm{Sym}}

\newcommand{\supp}{\mathrm{supp}}

\newcommand{\Ess}{\mathcal{ES}}
\newcommand{\Path}{\mathcal{P}}

\newcommand{\Lbi}[1]{\Lambda^{(#1)}}
\newcommand{\cBi}[1]{\mathcal{B}^{(#1)}}
\newcommand{\Gai}[1]{\Gamma^{(#1)}}
\newcommand{\cKi}[1]{\mathcal{K}^{(#1)}}
\newcommand{\cPi}[1]{\mathcal{P}^{(#1)}}
\newcommand{\cEi}[1]{\mathcal{E}^{(#1)}}
\newcommand{\Sgi}[1]{\Sigma^{(#1)}}
\newcommand{\cBe}[1]{\mathcal{B}^{(#1)}_e}
\newcommand{\cLi}[1]{\mathcal{L}^{(#1)}}
\newcommand{\graph}{\mathrm{Graph}}
\newcommand{\inv}{\mathrm{inv}}

\usepackage{mathtools}
\DeclarePairedDelimiter{\flr}{\lfloor}{\rfloor}
\usepackage{mathabx}

\author{
 V. Kaloshin\footnote{University of Maryland at College Park
 (\texttt{vadim.kaloshin\@ gmail.com})},\ \ \
 K. Zhang\footnote{University of Toronto (\texttt{kzhang\@ math.utoronto.edu})}}

\title{Dynamics of the dominant Hamiltonian,
with
 applications to Arnold diffusion}
\begin{document}
\maketitle

\begin{abstract}

It is well known that instabilities of nearly integrable Hamiltonian 
systems occur around resonances. Dynamics near resonances 
of these systems is well approximated by the associated 
averaged system, called {\it slow system}. Each resonance is 
defined by a basis (a collection of integer vectors). We introduce 
a class of resonances whose basis can be divided into two well 
separated groups and call them {\it dominant}. We prove that 
the associated slow system can be well approximated by 
a subsystem given by one of the groups, both in the sense 
of the vector field and weak KAM theory. 

One of crucial ingredients of proving Arnold diffusion is 
understanding the structure of invariant (Aubry) sets of nearly 
integrable systems.  As an important application we construct 
a diffusion path for a generic nearly integrable system such 
that invariant (Aubry) sets along this path have a "simple" 
structure similar to the structure of Aubry-Mather sets of 
twist maps. This is a crucial ingredient in proving Arnold 
diffusion for convex Hamiltonians in any number of degrees 
of freedom.
\end{abstract}


\tableofcontents

\section{Introduction}
\label{sec:intro}

Consider a nearly integrable system with $n\frac12$ degrees of freedom
\begin{equation}
 \label{pert-hamiltonian}
	H_\varepsilon(\theta, p, t) = H_0(p) + \varepsilon H_1(\theta, p, t), \quad \theta \in \T^n,\ p \in \R^n,\ t\in \T. 
\end{equation}
We will restrict to the case where the integrable part $H_0$ is strictly convex, more precisely, we assume that there is $D>1$ such that 
\[
	D^{-1}\, \Id \le \partial^2_{pp}H_0(p) \le D\, \Id
\]
as quadratic forms, where $\Id$ denotes the identity matrix. 

The main motivation behind this work is the question of Arnold diffusion, that is, topological instability for the system $H_\varepsilon$. Arnold  provided the first example in \cite{Arn64}, and asks (\cite{Ar3,Ar2,Ar5}) whether topological instability is ``typical'' in nearly integrable systems with $n\ge 2$ (the system is stable when $n=1$, due to low dimensionality). 

It is well known that the instabilities of nearly integrable systems occurs along resonances. Given an integer vector $k = (\bark, k^0) \in \Z^n \times \Z$ with $\bark \ne 0$, we define the resonant submanifold to be	$\Gamma_k = \{p \in \R^n: \,k \cdot (\omega(p), 1) =0\}$, 
where $\omega(p) = \partial_p H_0(p)$. More generally, we consider a subgroup $\Lambda$ of $\Z^{n+1}$ which does not contain vectors of the type $(0, \cdots, 0, k^0)$, called a \emph{resonance lattice}. The \emph{rank} of $\Lambda$ is the dimension of the real subspace containing it. Then for a rank $d$ resonance lattice $\Lambda$, we define
\[
	\Gamma_\Lambda = \bigcap \{\Gamma_k: \ \  k \in \Lambda\} = \bigcap_{i=1}^d \Gamma_{k_i},
\]
where $\{k_1, \cdots, k_d\}$ is any linear independent set in $\Lambda$. We call such $\Gamma_\Lambda$ a {\it $d-$resonance submanifold} ($d-$resonance for short), which is a co-dimension $d$ submanifold of $\R^n$, and in particular, an $n-$resonant submanifold is a single point. We say that $\Lambda$ is \emph{irreducible} if it is not contained in any lattices of the same rank, or equivalently, $\Span_\R \Lambda \cap \Z^{n+1} = \Lambda$.

We now focus on the diffusion that occurs along a connected net of  $(n-1)-$resonances, with each $(n-1)-$resonance being  a curve in $\R^n$. Let us first consider diffusion along a single $(n-1)-$resonance $\Gamma$. It is shown in \cite{BKZ11} that generically, diffusion indeed occur  along $\Gamma$, except for a finite subset of $n-$resonances (called 
{\it the strong resonances}) which divides $\Gamma$ into disconnected
 components. A strong resonance can be viewed as the intersection of 
$\Gamma$ with a transveral $1-$resonance manifold $\Gamma_{k'}$
(see Figure \ref{fig:diffusion path}). 

\begin{figure}[t]
\centering
\def\svgwidth{2.5in}
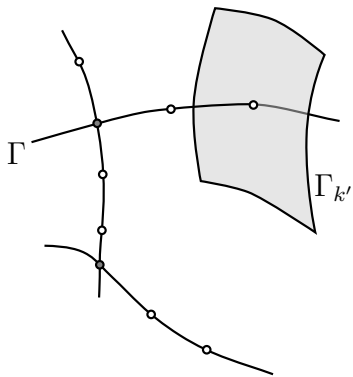
\caption{Diffusion path and essential resonances in $n=3$.  The hollow dots requires crossing, while the grey dots requires switching}
\label{fig:diffusion path}
\end{figure}

The main obstacle to proving diffusion along $\Gamma$ reduces to whether the diffusion can ``cross'' the strong resonances. In a more general diffusion path that contains two intersecting $(n-1)-$resonances $\Gamma_1$ and $\Gamma_2$, the intersection is an $n-$resonance which by definition is 
strong. The question is then whether one can travel along $\Gamma_1$
 and then ``switch'' to $\Gamma_2$ at the intersection. Solution to either 
problem requires an understanding of the system near an $n-$resonance. 

For an $n-$resonance $\{p_0\} = \Gamma_\Lambda$, we assume that $\Lambda$ is irreducible, and $\cB = [k_1, \cdots, k_n]$ is an ordered basis over $\Z$.  The study of diffusion near $p_0$ reduces to the study of a particular \emph{slow system} defined on $\T^n \times \R^n$, denoted $H^s_{p_0, \cB}$. More precisely, in an $O(\sqrt{\varepsilon})-$neighborhood of $p_0$, the system $H_\varepsilon$ admits the normal form (see \cite{KZ13}, Appendix B)
\[
	H^s_{p_0,\cB}(\varphi, I) + \sqrt{\varepsilon} P(\varphi, I, \tau), \quad \varphi \in \T^n, I\in \T^n, \tau \in \sqrt{\varepsilon}\T, 
\]
where
\[
	\varphi_i = k_i \cdot (\theta,t), \, 1\le i \le n, \quad (p-p_0)/\sqrt{\epsilon} = \bark_1 I_1 + \cdots \bark_n I_n. 
\]
Therefore, $H_\epsilon$ is conjugate to  a fast periodic perturbation to $H^s_{p_0, \cB}$.  While it is possible to give a basis free definition of the slow system, we opt to choose a particular basis $\cB$, and rendering our setup \emph{basis dependent}. Such averaged systems were studied in \cite{Mat08}.

When $n=2$, the slow system is a $2$ degrees of freedom mechanical system, the structure of
its (minimal) orbits is well understood. This fact underlies the results on Arnold diffusion in 
two and half degrees of freedom 
(see \cite{Mather03}, \cite{Mat08}, \cite{Mather11}, \cite{Ch13}, \cite{KZ13},\cite{GK14}, \cite{KMV04}, \cite{Mar1}, \cite{Mar2}). 
This is no longer the case when $n>2$, which is a serious obstacle to proving Arnold diffusion in 
higher degrees of freedom. In \cite{KZ14} it is proposed that we can sidestep this difficulty by using 
\emph{dimension reduction}: using existence of normally hyperbolic invariant cylinders (NHICs) to 
restrict the system to a lower dimensional manifold. This approach only works when the slow system 
has a particular \emph{dominant structure}, which is the topic of this paper. 

In order to make this idea specific it is convenient 
to define the slow system for any $p_0$ and any $d-$resonance 
$d\le n$. For $p_0 \in \R^n$, an irreducible rank $d$ resonance 
lattice $\Lambda$, and its basis 
$\cB = [k_1, \cdots, k_d]$, the slow system is
\begin{equation}
\label{eq:slow-system}
H_{p_0, \cB}^s(\varphi, I) = K_{p_0, \cB}(I) - U_{p_0, \cB}(\varphi), \quad \varphi\in \T^d, I \in \T^d.
\end{equation}
Suppose the fourier expansion of $H_1$ is  $\sum_{k \in \Z^{n+1}} h_k(p) e^{2\pi i k \cdot (\theta, t)}$, then  
\begin{equation}
\label{eq:Kn}
	K_{p_0, \cB} (I) = \frac12 \partial^2_{pp}H_0(p_0)(I_1 \bark_1 + \cdots + I_d \bark_d) \cdot (I_1 \bark_1 + \cdots + I_d \bark_d), 
\end{equation}
\begin{equation}
\label{eq:Un}
	  U_{p_0, \cB}(\varphi_1, \cdots, \varphi_d) = - \sum_{l \in \Z^d} h_{l_1 k_1 + \cdots l_d k_d}(p) e^{2\pi i (l_1 \varphi_1 + \cdots + l_d \varphi_d)}. 
\end{equation}
The system $H_{p_0, \cB}^s$ is only dynamically meaningful when $p_0 \in \Gamma_\Lambda$. However, the more general set up allows us to embed the meaningful slow systems into a nice space.

We say that the resonance lattice $\Lambda$ admits a dominant structure if it contains an irreducible lattice $\Lambda^\st$ of rank $m <d$, such that 
\begin{equation}
 \label{eq:relative-lattice-norm}
	M(\Lambda|\Lambda^\st) : = 
\min_{k \in \Lambda \setminus \Lambda^\st} |k| \gg \max_{k \in \cB^\st} |k|, 
\end{equation}
where $|k| = \sup_i|k_i|$ is the sup-norm. Given the relation $\Lambda^\st \subset \Lambda$, one can choose an \emph{adapted basis} $\cB = [k_1, \cdots, k_d]$ of $\Lambda$, meaning that $\cB^\st = [k_1, \cdots, k_m]$ 
is a properly ordered  basis of $\Lambda^\st$. 

In this case we have two 
systems $H^s_{p_0, \cB^\st}$ and $H^s_{p_0, \cB}$, which we will call the \emph{strong system} and \emph{slow system} respectively. When the lattices have a dominant structure (see \eqref{eq:relative-lattice-norm}), the slow system $H^s_{p_0, \cB}$ inherits considerable amount of information from the strong system. Indeed, let us denote 
\[
	  H^s_{p_0, \cB^\st} = K^\st(I_1, \cdots, I_m) - U^\st(\varphi_1, \cdots, \varphi_m), 
\]
\[
	  H^s_{p_0, \cB} = K^s(I_1, \cdots, I_d) - U^s(\varphi_1, \cdots, \varphi_d), 
\]
under \eqref{eq:relative-lattice-norm} and after choosing an appropriate adapted basis, we will show that 
\begin{equation}\label{eq:comp-slow-st}
  K^\st(I_1, \cdots, I_m) = K^s(I_1, \cdots, I_m, 0, \cdots, 0), \quad \|U^s - U^\st\|_{C^2} \ll \|U^\st\|_{C^2},
\end{equation}
which indicates $H^s_{p_0, \cB}$ can be approximated by $H^s_{p_0, \cB^\st}$. The variables $\varphi_i, I_i$, $1 \le i \le m$ are called the \emph{strong variables}, while $\varphi_i, I_i$, $m+1 \le i \le d$ are called the \emph{weak variables}.

Recall that for each convex Hamiltonian $H$, we can associate a Lagrangian $L = L_H$, and the Euler-Lagrange flow is conjugate to the Hamiltonian flow. In particular, when $H = K(I) - U(\varphi)$, and $K(I)$ is a quadratic form,
\[
   \dot{\varphi} = v, \quad \dot{v} = \frac{d}{dt}(\partial_I K(I))  = (\partial^2_{II} K) \, \partial_\varphi U(\varphi).
\]
whose vector field is denoted the Euler-Lagrange vector field.  Denote by $X_{Lag}^\st$ and $X_{Lag}^s$ the Euler-Lagrange vector fields associated to the Hamiltonians $H^s_{p_0, \cB^\st}$ and $H^s_{p_0, \cB}$. Since the system for $X^\st_{Lag}$ is only defined for the strong variables $(\varphi_i, v_i)$, $1\le i\le m$, we define a \emph{trivial extension} of $X^\st_{Lag}$ by setting $\dot{\varphi}_i = \dot{v}_i =0 $, $m+1 \le i \le d$.



We show that after 
choosing a proper adapted basis, and a suitable rescaling 
transformation in the weak variables, the transformed vector field $X^s_{Lag}$ converges to that of $X^\st_{Lag}$ in some sense. In particular, if $X^\st_{Lag}$ admits a normally hyperbolic invariant cylinder (NHIC), so does $X^s_{Lag}$.  In a separate direction, 
we also obtain a limit theorem on the weak KAM solutions by 
variational arguments. We now formulate our main results in 
loose language, leaving the precise version for the next section. 
\begin{main}
Assume that $r>n+4(d-m)+4$. Given a fixed lattice
$\Lambda^\st$ of rank $m$ with a fixed basis $\cB^\st$, for each 
rank $d,\ m\le d\le n$ irreducible lattice $\Lambda\supset \Lambda^\st$, there exists an adapted basis $\cB$ such that:
\begin{enumerate}
\item (Geometrical) As $M(\Lambda|\Lambda^\st) \to \infty$, the projection of $X_{Lag}^s$ to the strong variables $(\varphi_i, v_i)$, $1 \le i \le m$ converges to $X_{Lag}^\st$ uniformly. Moreover, by introducing a coordinate change and rescaling affecting only the weak variables $(\varphi_i, v_i)$, $m+1 \le i \le d$, the transformed vector field of $X_{Lag}^s$ converges to a trivial extension of the vector field of $X_{Lag}^\st$. As a corollary, we obtain that if $X_{Lag}^\st$ admits an NHIC, then so does $X_{Lag}^s$ for sufficiently large $M(\Lambda|\Lambda^\st)$. 
\item (Variational) As $M(\Lambda|\Lambda^\st)\to \infty$, the weak KAM solution of  $H^s_{p_0, \cB}$ (of properly chosen cohomology classes)  converges uniformly to a trivial extension of a weak KAM solution of $H^s_{p_0, \cB^\st}$, considered as functions on $\R^d$. We also obtain corollaries concerning the limits of Ma\~ne, Aubry sets, rotation number of minimal measure, and Peierl's barrier function. The precise definitions of these objects will be given later. 
\end{enumerate}
\end{main}

The combination of the persistence of NHIC, and limits of the weak KAM solutions allows us to localize and restrict Aubry sets. 
 As a demonstration of our theory,  in Appendix~\ref{sec:diff-path} 
we show that one can construct a connected net of $(n-1)-$resonances, 
such that all strong resonances have the dominant structure. 
We then show that most\footnote{See definition of the AM property in Appendix~\ref{sec:intro-path}} Aubry sets with a chosen rational homology class are contained 
in three dimensional NHICs, and the topology of such Aubry sets resembles the Aubry-Mather sets 
for twist maps. This is closely related to the AM property we introduce in Appendix~\ref{sec:diff-path}. 
The relation between the property of the Aubry sets and diffusion mechanism is given in Appendix~\ref{sec:diffusion-amproperty}.  In \cite{KZ13} and \cite{KZ14} using these structures we prove 
existence of Arnold diffusion. This net of diffusion paths also can be chosen to 
be $\gamma-$dense for any pre-determined $\gamma>0$.  
We show that  for a ``typical'' $H_\varepsilon$ (for any $n > 2$), such 
a net exists and expect to prove Arnold diffusion along this net in 
a future publication.

The statement that $H^s_{p_0, \cB^\st}$ approximates $H^s_{p_0, \cB}$ is related to the classic result of partial averaging (see for example \cite{AKN}). The statement $\min_{k \in \Lambda \setminus \Lambda^\st} |k| \gg \max_{k \in \Lambda^\st} |k|$ says that the resonances in $\Lambda^\st$ is much \emph{stronger} than the rest of the resonances in $\Lambda$.   Partial averaging says that the weaker resonances contributes to smaller terms in a normal form. 

However, our treatment of the partial averaging theory is quite different from the classical theory. By looking at the rescaling limit, we study the property of the averaging independent of the small parameter $\varepsilon$. 
The theory is far from a simple corollary of \eqref{eq:comp-slow-st}, with the main difficulty coming from the fact that 
as $M(\Lambda|\Lambda^\st) \to \infty$, the quadratic part of the system $H^s_{p_0, \cB}$ becomes unbounded.

In \cite{Mat08}, John Mather developed a theory of (partial) averaging for a nearly integrable Lagrangian  system 
$$
L_\varepsilon(\theta, v, t) = L_0(v)  + \varepsilon L_1(\theta, v, t),\  
\textup{ where }\ (\theta,v,t)\in\T^n \times \R^n \times \T.
$$ 
In particular, it is shown that the slow system relative to a resonant lattice can be defined on the tangent bundle of 
a sub-torus $\T^d \times \R^d,\ d<n$. Quantitative estimates on the action of  minimizing orbits of the original system 
versus the slow system are obtained. Our variational result is related to \cite{Mat08}, but different in many ways. 
We work with the scaling limit system, and the small parameter $\varepsilon$ does not show up in our analysis. 
We also avoid quantitative estimates (in the statement of the theorem) and obtain a limit theorem in weak KAM 
solutions. This allows us to take consecutive limits, which is very useful for our construction of diffusion path in 
higher dimensions (see Appendix~\ref{sec:diff-path}).

The formulation of the limit theorem in weak KAM solution requires special care. 
A natural candidate is Tonelli convergence (convergence of Lagrangian within the Tonelli family, see \cite{Ber10}). In our setup, $H^s_{p_0, \cB^\st}$ and $H^s_{p_0, \cB}$ are defined on different spaces, we need to consider the trivial extension of $H^s_{p_0, \cB^\st}$ to a higher dimensional space. The extended Lagrangian is then \emph{degenerate} and obviously not Tonelli. Moreover, the standard $C^2$ norm of the Lagrangian becomes unbounded in the limit process.  We nevertheless obtain the convergence of weak KAM solutions. 

While this paper is mainly motivated by Arnold diffusion, we hope our treatment of partial averaging is of independent interest. 

The plan of the paper is as follows. The rigorous formulation of the results will be presented in section~\ref{sec:results}. The choice of the basis is handled in section~\ref{sec:basis}, and the estimates of the vector fields, including the geometrical result is in  section~\ref{sec:strong-slow}. The variational aspect is more involved, and occupies sections 
\ref{sec:con-weak-kam}  and  \ref{sec:cor-aub}, with some technical estimates deferred to section~\ref{sec:tech-est} .
As we already mentioned Appendices A and B are devoted to application of dominant systems to Arnold diffusion.
In Appendix C we prove Theorem \ref{thm:nhic-persist} about existence of normally hyperbolic invariant 
cylinders stated in section \ref{sec:nhic-persistence}.

\section{Formulation of results}
\label{sec:results}

\subsection{The slow system and the choice of basis}

Recall that a slow system is 
\[
	H_{p_0, \cB}^s(\varphi, I)  = K_{p_0, \cB}(I) - U_{p_0, \cB}(\varphi), \quad \varphi \in \T^d, I \in \R^d,
\]
defined for a rank-$d$ irreducible lattice $\Lambda$ with ordered basis $\cB = [k_1, \cdots, k_d]$, and a point $p_0 \in \R^n$. Let $Q_0(p) = \partial^2_{pp}H_0(p) \in \Sym(n)$, where $\Sym(n)$ denote the space of $n\times n$ symmetric matrices. Define 
\be \label{quadratic-form}
	Q(p) = \bmat{ Q_0(p) & 0 \\ 0 & 0} \in \Sym(n+1), 
\ee
then \eqref{eq:Kn} becomes
\begin{equation}\label{eq:Kd}
  	K_{p_0, \cB}(I_1, \cdots, I_d) = \frac12 Q(p_0)(I_1 k_1 + \cdots + I_d k_d) \cdot (I_1 k_1 + \cdots + I_d k_d). 
\end{equation}
Let us also denote 
\begin{equation}
\label{eq:ZB}
	Z_\cB(\varphi_1, \cdots, \varphi_d,p) = \sum_{l \in \Z^d} h_{l_1 k_1 + \cdots l_d k_d}(p) e^{2\pi i (l_1 \varphi_1 + \cdots + l_d \varphi_d)}, 
\end{equation}
where  $H_1(\theta, p, t)  = \sum_{k \in \Z^{d+1}} h_k(p) e^{2\pi i k \cdot(\theta, t)}$, then \eqref{eq:Un} becomes 
\[
U_{p_0, \cB}(\varphi) = - Z_{\cB}(\varphi, p_0). 
\]

We now fix  a rank-$m$ irreducible resonant lattice $\Lambda^\st$, called the strong lattice,  and study all irreducible lattices $\Lambda\supset \Lambda^\st$ of rank $d$. Fix a basis $\cB^\st = [k_1, \cdots, k_m]$ of the strong lattice $\Lambda^\st$, we extend it to an adapted basis $\cB = [k_1, \cdots, k_d]$ of $\Lambda$. The extended basis is of course non-unique, and our first theorem concerns the choice of a ``nice'' basis.  When $\cB$ is an adapted basis, the corresponding slow system is $H^s_{p_0, \cB}$ is denoted $H^s_{p_0, \cB^\st, \cB^\wk}$, where
\[
	\cB^\st = [k_1, \cdots, k_m] = [k_1^\st, \cdots, k_m^\st],
\]
\[
  \cB^\wk = [k_{m+1}, \cdots, k_d] = [k_1^\wk, \cdots, k_{d-m}^\wk].
\]
Denote
\[
	\varphi^\st = (\varphi_1, \cdots, \varphi_m), \quad \varphi^\wk = (\varphi_{m+1}, \cdots, \varphi_d),
\]
\[
	I^\st = (I_1, \cdots, I_m), \quad I^\wk = (I_{m+1}, \cdots, I_d), 
\]
we write
\begin{multline*}
	H^s_{p_0, \cB}(\varphi, I) = H^s_{p_0, \cB^\st, \cB^\wk}(\varphi^\st, \varphi^\wk, I^\st, I^\wk)  \\
	 =  K_{p_0,\cB^\st,\cB^\wk}(I^\st, I^\wk) - U_{p_0, \cB^\st}(\varphi^\st) - U^\wk_{p_0, \cB^\st, \cB^\wk}(\varphi^\st, \varphi^\wk), 
\end{multline*}
where 
\[
	U_{p_0, \cB^\st}(\varphi^\st) = - Z_{\cB^\st}(\varphi^\st, p_0),  \quad U^\wk_{p_0, \cB^\st, \cB^\wk}(\varphi^\st, \varphi^\wk)= - (Z_\cB - Z_{\cB^\st})(\varphi^\st, \varphi^\wk,p_0).
\]
Note that the slow system  to $\Lambda^\st$ is
\begin{align*}
	H^s_{p_0, \cB^\st}(\varphi^\st, I^\st) &= K_{p_0, \cB^\st}(I^\st) - U_{p_0, \cB^\st}(\varphi^\st) \\
	&= K_{p_0, \cB^\st, \cB^\wk}(I^\st, 0) - U_{p_0, \cB^\st}(\varphi^\st),
\end{align*}
where the second line follows directly from \eqref{eq:Kd}.
In other words, $H^s_{p_0, \cB^\st}(\varphi^\st, I^\st)$ can be obtained from $H^s_{p_0, \cB^\st, \cB^\wk}$ by setting $I^\wk =0$ and $U^\wk_{p_0, \cB^\st, \cB^\wk}=0$. 

We will show that by choosing a good basis, the term $U^\wk_{p_0, \cB^\st, \cB^\wk}$ can be made arbitrarily small. Moreover, we can further decompose $U^\wk_{p_0, \cB^\st, \cB^\wk}$ to obtain precise estimates on its dependence on each  of the weak angles $\varphi_j^\wk$.  To do this 
we define, for each $1 \le i \le d$, $\cB_i = [k_1, \cdots, k_i]$ an ordered 
basis and $\Lambda_i$  the corresponding rank $i$ lattice. Then 
$\cB_m = \cB^\st \subset \cdots \subset \cB_d = \cB$. For 
$m < i  \le d$, we write
\begin{align*}
	 U^\wk_{p_0,\cB_{i-1}, \cB_i} &= -  (Z_{\cB_i}- Z_{\cB_{i-1}})(\varphi_1, \cdots, \varphi_i,p_0) \\
	 &= -  (Z_{\cB_i}- Z_{\cB_{i-1}})(\varphi^\st, \varphi_1^\wk, \cdots, \varphi_{i-m}^\wk,p_0).   
\end{align*}
Then the slow system takes the form
\begin{equation}
\label{eq:Hs}
H^s_{p_0, \cB^\st, \cB^\wk} = K_{p_0, \cB} - U_{p_0, \cB^\st} - U^\wk_{p_0, \cB_m, \cB_{m+1}} - \cdots - U^\wk_{\cB_{d-1}, \cB_d}. 
\end{equation}

%

\begin{theorem}\label{thm:basis-norm}
Let $\Lambda^\st \subset \Lambda$ be irreducible resonance lattices of 
rank $m$ and $d$ resp., where $m < d$. Fix an ordered basis 
$\cB^\st =[k_1^\st, \cdots, k_m^\st]$ of $\Lambda^\st$. Suppose $H_1$ 
is $C^r$ with $r >n+ 2(d-m)+4$, and $\|H_1\|_{C^r} =1$. Then there exists  
a constant $\kappa = \kappa(H_0, \cB^\st, n)>1$, integer vectors $\cB^\wk = [k_1^\wk, \cdots, k_{d-m}^\wk]$ with $[\cB^\st, \cB^\wk]$ forming an adapted basis,  such that the following hold. 
  \begin{enumerate}
   \item For any $1 \le i < j \le d-m$, $|k_i^\wk| < \kappa (1+|k_j^\wk|)$. 
   \item For $1 \le j \le d-m$, we have
   \[
   \|U^\wk_{p_0,\cB_{j+m-1}, \cB_{j+m}}\|_{C^2} \le \|Z_{\cB_{j+m}} - Z_{\cB_{j+m-1}}\|_{C^2} \le \kappa |k_j^\wk|^{-r + n + 2(d-m) +4}. 
   \]
   \end{enumerate} 
\end{theorem}
\begin{remark}
Item 2 implies that as $M(\Lambda|\Lambda^\st) \to \infty$, for 
the specifically chosen basis, we have 
$\|U^\wk_{p_0, \cB^\st,\cB^\wk}\|_{C^2} \to 0$. 

Item 1 says that vectors in $\cB^\wk$ are approximately in an increasing order. This, when combined with item 2, implies that the norm of the weak potentials $U_{p_0, \cB_{j+m-1}, \cB_{j+m}}$ are approximately in an decreasing order.
\end{remark}
This theorem is proven in section~\ref{sec:basis}. 

We will call any Hamiltonian that satisfy the conclusions of Theorem~\ref{thm:basis-norm} a \emph{dominant Hamiltonian}. In the next section, we define an abstract space of dominant Hamiltonians. 

\subsection{Abstract space of dominant Hamiltonians}
\label{sec:intro-dom}

We start with the following data:
\begin{enumerate}
\item A $C^2$ function $Q_0: \R^n \to Sym(n)$, where $Sym(n)$ denote the space of $n\times n$ symmetric matrices. Define as before $Q(p) = \bmat{Q_0(p) & 0 \\ 0 & 0}$, and assume $D^{-1}\Id \le Q_0 \le D \Id$. 
\item An irreducible resonance lattice $\Lambda^\st \subset \Z^{n+1}$ of rank $1 \le m < n$, and a basis $\cB^\st = [k_1^\st, \cdots, k_m^\st]$. 
\item Constants $\kappa >1$ and $q>1$.  
\end{enumerate}
We continue to use the notation $(k_1, \cdots, k_m) = (k_1^\st, \cdots, k_m^\st)$ and $(k_{m+1}, \cdots, k_d) = (k_1^\wk, \cdots, k_{d-m}^\wk)$, and apply the same convention to the variables $\varphi$ and $I$. Define
\[
	\Omega^{m,d}:= (\Z^{n+1})^{d} \times \R^n \times C^2(\T^m)  \times C^2(\T^{m+1}) \times \cdots \times C^2(\T^d),
\quad \cH^s: \Omega^{m,d} \to C^2(\T^d \times \R^d), 
\]
with
\begin{multline*}
	(\cB^\st = [k_1^\st, \cdots, k_m^\st], 
\cB^\wk = [k_1^\wk, \cdots, k_{d-m}^\wk], 
p_0, U^\st, \cU^\wk= \{U^\wk_1, \cdots, U^\wk_{d-m}\}) \mapsto  \\
	 \cH^s(\cB^\st,\cB^\wk, p_0, U^\st, \cU^\wk) = K_{p_0, \cB^\st, \cB^\wk}(I) - U^\st(\varphi_1, \cdots, \varphi_m) - \sum_{j=1}^{d-m} U_j^\wk(\varphi_1, \cdots, \varphi_{j+m}), 
\end{multline*}
where 
\[
	K_{p_0, \cB^\st, \cB^\wk}(I) = \frac12  
 Q(p_0)(k_1 I_1 + \cdots k_d I_d) \cdot (k_1 I_1 + \cdots + k_d I_d) .
\]
We equip $\Omega^{m,d}$ with the product topology, with discrete topology on $k_j^\wk$ and the standard norms on other components.  $\cH^s$ is smooth in $p_0, U^\st, U_1^\wk, \cdots, U_{d-m}^\wk$. 
Let $\Omega^{m,d}(\cB^\st)$ be the subset of $\Omega^{m,d}$ 
with fixed $\cB^\st$.

We define  $\Omega^{m,d}_{\kappa, q} (\cB^\st)
\subset \Omega^{m,d}(\cB^\st)$ to be the tuple 
$(\cB^\wk, p_0, U^\st, \cU^\wk)$ satisfying 
the following conditions:
\begin{enumerate}
   \item For any $1 \le i < j \le d-m$, $|k_i^\wk| \le \kappa$ $(1+|k_j^\wk|)$. 
   \item For each $1 \le j \le d-m$,   $\|U^\wk_j\|_{C^2} \le \kappa |k_j^\wk|^{-q}$.
\end{enumerate}
Each element in $\cH^s(\Omega^{m,d}_{\kappa,q})$ is called an $(m,d)-$dominant Hamiltonian with constants $(\kappa, q)$.  Define $$\lM(\cB^\wk) = \min_{1 \le j \le d-m}|k_j^\wk|,
$$ 
then in $\Omega^{m,d}_{\kappa, q}(\cB^\st)$, we have 
$\|U_j^\wk\| \le \kappa \lM(\cB^\wk)^{-q}$, i.e. the weak potential 
$U^\wk := \sum_{j=1}^{d-m} U_j^\wk\to 0$ as $\lM(\cB^\wk) \to \infty$. 

We now restate Theorem~\ref{thm:basis-norm} using the formal definition. 
\begin{thmstar}[Theorem~\ref{thm:basis-norm} restated]
Under the assumptions of Theorem~\ref{thm:basis-norm},  there exists  a constant $\kappa = \kappa(H_0, \cB^\st, n)>1$, integer vectors $\cB^\wk = [k_1^\wk, \cdots, k_{d-m}^\wk]$ with $\cB^\st, \cB^\wk$ forming an adapted basis, such that
\[
	(\cB^\wk, p, U_{p_0, \cB^\st}, (U_{p_0, \cB_m, \cB_{m+1}}, \cdots, U_{p_0, \cB_{d-1}, \cB_d})) \in \Omega^{m,d}_{\kappa, r-n- 2(d-m)-4}(\cB^\st). 
\]
\end{thmstar}

The strong Hamiltonian is defined by the mapping
\begin{equation*}
	\cH^\st: \R^n \times C^2(\T^m) \to C^2(\T^m \times \R^m), \quad  \cH^\st(p_0, U^\st)
	= K_{p_0, \cB^\st, \cB^\wk}(I^\st, 0) - U^\st(\varphi^\st).
\end{equation*}
We extend the definition to $\Omega^{m,d}$ by writing $\cH^\st(\cB^\st,\cB^\wk, p_0, U^\st, \cU^\wk) = \cH^\st(p_0, U^\st)$. We will prove all our limit theorems in the space $\Omega^{m,d}_{\kappa,q}(\cB^\st)$. 

\subsection{The rescaling limit}
\label{sec:intro-rescale}
 We fix	$\cB^\st, \kappa>1$ and 
$(\cB^\wk, p, U^\st, \cU^\wk) \in \Omega^{m,d}_{\kappa, q}(\cB^\st)$. 
Denote
\[
	H^s  = \cH^s(\cB^\st,\cB^\wk, p, U^\st, \cU^\wk),   \quad H^\st = \cH^\st(p, U^\st).
\]
Then
\[
	H^s(\varphi, I) = K(I) - U^\st(\varphi^\st) - U^\wk(\varphi^\st, \varphi^\wk), \quad  H^\st(\varphi^\st, I^\st) = K(I^\st, 0) - U^\st(\varphi^\st), 
\]
where $U^\wk = \sum_{j=1}^{d-m} U_j^\wk$. As $\lM(\cB^\wk) \to \infty$, we have $\|U^\wk\|_{C^2} \to 0$. However, $K(I^\st, I^\wk)$ is not a small perturbation of $K(I^\st, 0)$, in fact, as $\lM(\cB^\wk) \to \infty$,  $K(I^\st, I^\wk)$ becomes unbounded (since each $|k_j^\wk| \to \infty$, see also \eqref{eq:Kd}) .

We write
\begin{equation}
\label{eq:block-ABC}
	\partial^2_{II}K = 
	\begin{bmatrix}
	A & B \\ B^T & C
	\end{bmatrix}, \quad A = \partial^2_{I^\st I^\st}K, B = \partial^2_{I^\st I^\wk}K, C = \partial^2_{I^\wk I^\wk}K, 
\end{equation}
then 
\begin{equation}\label{eq:ABC}
	(A)_{ij} = (k_i^\st)^TQk_j^\st, \quad (B)_{ij} = (k_i^\st)^T Q k_j^\wk, \quad 
	(C)_{ij} = (k_i^\wk)^T Q k_j^\wk. 
\end{equation}
Note in particular that $A =\partial^2_{I^\st I^\st} H^\st$. The Hamiltonian equation for $H^s$ reads
\begin{equation*}
	 \begin{cases}
	\dot\varphi^\st = A I^\st + BI^\wk, \quad & \dot{I}^\st = \partial_{\varphi^\st} U,   \ \ \\
	\dot\varphi^\wk = B^T I^\st + C I^\wk,	\quad &\dot{I}^\wk  = \partial_{\varphi^\wk} U,\qquad \qquad
	 \end{cases}
\end{equation*}
where $U = U^\st + U^\wk$. Then the Lagrangian vector field is 
\begin{equation}\label{eq:Hs-Lag}
  \begin{cases}
  \dot\varphi^\st = v^\st, \quad & \dot{v}^\st = A\partial_{\varphi^\st} U + B \partial_{\varphi^\wk} U,   \ \ \\
	\dot\varphi^\wk = v^\wk,	\quad &\dot{v}^\wk  = B^T \partial_{\varphi^\st} U  + C\partial_{\varphi^\wk} U,\qquad \qquad
  \end{cases}
\end{equation}
which will be compared to the Lagrangian vector field of $H^\st$
\begin{equation}\label{eq:Xst}
	\dot{\varphi}^\st = v^\st, \quad \dot{v}^\st = A \partial_{\varphi^\st} U^\st,
\end{equation}
denoted $X^\st$. 
To show that the projection of \eqref{eq:Hs-Lag} converges to \eqref{eq:Xst}, we only need to show $\|B \partial_{\varphi^\wk}U\| \to 0$. 

For convergence of weak variables, we will need a rescaling. It turns out that it is better to rescale the $I^\wk$ variable. Introduce the coordinate change
\begin{equation}
\label{eq:half-lag}
(\varphi^\st, v^\st, \varphi^\wk, v^\wk) \mapsto (\varphi^\st, v^\st, \varphi^\wk,   I^\wk). 
\end{equation}
 This is a ``half Lagrangian'' setting in the sense that $(\varphi^\st, v^\st)$ is remain the Lagrangian setup, while $(\varphi^\wk, I^\wk)$ is in the Hamiltonian format. Using 
 \[
 v^\st = A I^\st + B I^\wk, \quad v^\wk = B^T I^\wk + C I^\wk,
 \]
 we get 
 \[
 v^\wk = B^T A^{-1} v^\st - \tilde{C} I^\wk,
 \]
 where $\tilde{C} = C - B^T A^{-1} B$ is an invertible symmetric matrix. Then the half-Lagrangian equation writes
\begin{equation}
\label{eq:Xs}
	 \begin{cases}
	\dot\varphi^\st = v^\st, \quad &\dot{v}^\st = A \partial_{\varphi^\st} U + B \partial_{\varphi^\wk} U,\\
	 \dot\varphi^\wk = B^T A^{-1}v^\st - \tilde{C} I^\wk, \quad	&\dot{I}^\wk = \partial_{\varphi^\wk} U.
	\end{cases}
\end{equation}
 We denote by $X^s(\varphi^\st, v^\st, \varphi^\wk, I^\wk)$ the vector field of \eqref{eq:Xs}, defined on the universal cover $\R^m \times \R^m\times \R^{d-m} \times \R^{d-m}$.

Consider the trivial lift of the strong Lagrangian vector field $X^\st$, defined on the universal cover
\begin{equation}
\label{eq:XstL}
	 \begin{cases}
	\dot\varphi^\st = v^\st, \quad \qquad &\dot{v}^\st = A \partial_{\varphi^\st}U, \\
	\dot\varphi^\wk = 0, \qquad	\quad &\dot{I}^\wk = 0,
	\end{cases}
\end{equation}
whose vector field we denote by $X^\st_L(\varphi^\st, v^\st, \varphi^\wk, I^\wk)$. We show that $X^\st_L$ is a rescaling limit of $X^s$. 

Given $1\ge  \sigma_1 \ge \cdots \ge \sigma_{d-m}>0$, let $\Sigma = \diag\{\sigma_1, \cdots, \sigma_{d-m}\}$. We define a rescaling coordinate change $\Phi_\Sigma: \R^{2d}\to \R^{2d}$ by
\be \label{rescale}
	\Phi_\Sigma: (\varphi^\st,v^\st,  \varphi^\wk, I^\wk) \mapsto (\varphi^\st,  v^\st, \tilde{\varphi}^\wk,\tilI^\wk): = (\varphi^\st,  v^\st, \Sigma^{-1}\varphi^\wk,\Sigma I^\wk).
\ee
The rescaled vector field for $X^s$ is
\begin{equation}\label{eq:rescaling}
 \tilde{X}^s := (\Phi_\Sigma)^{-1} X^s \circ \Phi_\Sigma^{-1},
\quad \tilde{X}^s(\varphi^\st, v^\st, \varphi^\wk,  I^\wk)=
(\Phi_\Sigma)^{-1} X^s
(\varphi^\st,  v^\st, \Sigma^{-1}\varphi^\wk, \Sigma I^\wk), 
\end{equation}
while $X^\st_L$ is unchanged under the rescaling. 

\begin{theorem}\label{thm:resc-est}
Fix $\cB^\st$ and $\kappa>1$.
Assume that $q>2$. Then there exists a constant $M = M(\cB^\st, Q, \kappa, q,d-m) >1$, such that for $(\cB^\wk, p, U^\st, \cU^\wk) \in \Omega^{m,d}_{\kappa, q}(\cB^\st)$ and $H^s = \cH^s(\cB^\st,\cB^\wk, p, U^\st, \cU^\wk)$, $H^\st = \cH^\st(p, U^\st)$, such that the following hold. 

For the rescaling parameter $\sigma_j = |k_j^\wk|^{-\frac{q+1}{3}}$, uniformly on $\R^{m} \times \R^{d-m} \times \R^{m} \times \R^{d-m}$ we have 
\[
	\|\Pi_{(\varphi^\st, v^\st)}(\tilde{X}^s - X^\st_L) \|_{C^0} \le M \lM(\cB^\wk)^{-(q-1)},
\]
\[
	\|D\tilde{X}^s - DX_L^\st\|_{C^0} \le M \lM(\cB^\wk)^{-\frac{q-2}{3}}. 
\]
\end{theorem}

In particular, Theorem~\ref{thm:resc-est} implies that as $\lM(\cB^\wk) \to \infty$, the vector field $\tilde{X}^s$ 
converges to $X^\st_L$ in the $C^1$ topology over compact sets. Theorem~\ref{thm:resc-est} is proven in 
section~\ref{sec:vector-field}. 

\subsection{Persistence of normally hyperbolic invariant cylinders}
\label{sec:nhic-persistence}

Our main application for Theorem~\ref{thm:resc-est} is to prove persistence of normally hyperbolic invariant cylinders (NHICs). 

Let $W$ be a manifold. For $R>0$, let $B^l_R\subset \R^l$ denote the ball 
of radius $R$ at the origin. A $2l-$cylinder $\Lambda$ is defined by 
$\Lambda = \chi(\T^l \times B^l_R)$, where 
$\chi: \T^l \times B^l_R \to W$ is an embedding. 

Let $\phi_t$ be a $C^2$ flow on $W$, and  $\Lambda \subset W$ be a cylinder.  
We say that $\Lambda$ is \emph{normally hyperbolic (weakly) invariant cylinder} (NHWIC) if there exists 
$t_0 >0$ such that the following hold. 
\begin{itemize}
\item The vector field of $\phi_t$ is tangent to $\Lambda$ at every $z \in \Lambda$. 	
\item For each $z \in \Lambda_a$, there exists a splitting
\[
	T_z M = E^c(z) \oplus E^s(z) \oplus E^u(z), \quad \text{ where } E^c(z) = T_z \Lambda, 
\]
weakly invariant in the sense that 
\[
	D\phi_{t_0}(z) E^\sigma(z) = E^\sigma(\phi_{t_0}z), \quad \text{ if }  z,\, \phi_{t_0}z \in \Lambda \quad \text{and} 
\quad \sigma = c, s, u.  
\]
\item There exists $0 < \alpha < \beta <1$ and a $C^1$ Riemannian metric $g$ called the adapted metric on a neighborhood of $\Lambda$ such that whenever $z, \phi_{t_0}z \in \Lambda$,
\[
	\|D\phi_{t_0}(z)|E^s\|, \quad \|(D \phi(z)|E^u)^{-1}\| < \alpha,
\]
\[ 
\|(D\phi_{t_0}(z)|E^c)^{-1}\|,  \quad  \|D\phi_{t_0}(z)|E^c\| > \beta, 
\]
where the norms taken is with respect to the metric $g$. 
\end{itemize}
The cylinder is called \emph{normally hyperbolic (fully) invariant} if it satisfies the above conditions, and both $\Lambda$ and $\partial \Lambda$ are invariant under $\phi_{t_0}$.  A more common definition of normally hyperbolic (fully) invariant cylinders assumes a spectral radius condition, but our definition is equivalent, see  e.g. \cite{BrSt} Prop.5.2.2. 

Moreover:
\begin{itemize}
 \item  If the parameters $\alpha, \beta$ satisfies the bunching condition $\alpha < \beta^2$, then the bundles $E^s, E^u$ are $C^1$ smooth. 
 \item  When $E^s, E^u$ are smooth, we can always choose the adapted metric $g$ such that $E^s$, $E^u$ and $E^c$ are orthogonal. 
\end{itemize}


Recall that $X_{Lag}^\st$, $X_{Lag}^s$ denotes the Lagrangian vector fields. Suppose $X_{Lag}^\st$ admits a normally hyperbolic (fully) invariant cylinder $\Lambda^\st$, we claim that $X_{Lag}^s$  admits an weakly invariant cylinder diffeomorphic to $\Lambda^\st \times (\T^{d-m} \times \R^{d-m})$. 

\begin{theorem}\label{thm:nhic-persist}
Consider a strong lattice $\cB^\st$, a strong potential $U^\st \in C^2(\T^m),$ $\kappa>0,\ a>0,$ and $q >2$. 

 Assume that the Euler-Lagrange vector field $X^\st_{Lag}$ of $H^\st = \cH^\st(\cB^\st, p_0, U^\st)$ admits a $2l-$dimensional cylinder $\Lambda^\st_a=\chi^\st(\T^l \times B^l_{1+a})$ that 
is normally hyperbolic (fully) invariant, with the parameters $0 < \alpha < \beta^2 < 1$. 

Then there exists 
an open set $V \supset \Lambda^\st_0$ such that for any 
$\delta>0$, there exists $M>0$, such that for any 
$(\cB^\st, \cB^\wk, p_0, U^\st) \in \Omega_{m,d}^{\kappa,q}(\cB^\st)$, $H^s = \cH^s(\cB^\st, \cB^\wk, p_0, U^\st) $, the following hold.

There exists a $C^1$ embedding
 \[
 \eta^s = (\eta^\st, \eta^\wk) : (\T^l \times B^l) \times (\T^{d-m} \times \R^{d-m})  \to (\T^m \times B^m) \times (\T^{d-m} \times \R^{d-m}),
 \]
 such that $\Lambda^s = 
\eta^s((\T^l \times B^l) \times (\T^{d-m} \times \R^{d-m}))$ 
is a $2(l+d-m)-$dimensional NHWIC under $X_{Lag}^s$. Moreover, 
we have 
 \[
 \|\eta^\st - \chi^\st\|_{C^0}<\delta,
 \]
 and any $X_{Lag}^s-$invariant set contained in $V \times (\T^{d-m}\times \R^{d-m})$ is contained in $\Lambda^s$.  
 \end{theorem} 
The  assumption $\alpha < \beta^2$ is not necessary, and is assumed for simplicity of the proof. Nevertheless, the assumption is  satisfied in our intended application and in most perturbative settings. The proof is presented in Appendix~\ref{sec:nhic-proof}. 

\subsection{The variational aspect of dominant Hamiltonians}
\label{sec:intro-var}

We will develop a similar perturbation theory for the weak KAM solutions of the dominant Hamiltonian. The weak KAM solution is closely related to some important invariant sets of the Hamiltonian system, known as the Mather, Aubry and Ma\~ne sets.

\begin{itemize}
\item {\bf Preliminaries in weak KAM solutions}

In this section we give only enough concepts to formulate our theorem.  A more detailed exposition will be given in Section~\ref{sec:intro-weak-kam}. 
Let
\[
	H: \T^d \times \R^d \to \R
\]
be a $C^3$ Hamiltonian satisfying the condition $D^{-1}\Id \le  \partial^2_{II}H(\varphi, I) \le D\Id$.  The associated Lagrangian $L: \T^d \times \R^d \to \R$ is given by 
\[
	L_H(\varphi, v) = \sup_{I\in \R^n}\{ I \cdot v - H(\varphi, I)\} .
\]

Let $c \in \R^d \simeq H^1(\T^d, \R)$, we define Mather's alpha function to be
\[
	\alpha_H(c)  = -  \inf_\nu \left\{ \int (L_{H} - c\cdot v) d\nu \right\},
\]
where the infimum is taken over all Borel probability measures on $\T^d \times \R^d$ that is invariant under the Euler-Lagrange flow of $L_{H}$.

A continuous function $u : \T^d \to \R$ is called a (backward) weak KAM solution to $L_{H} - c\cdot v$ if for any $t>0$, we have 
\[
	u(x) = \inf_{y \in \T^d, \gamma(0) =y, \gamma(t) =x} \left( u(y) + \int_0^t (L_{H}(\gamma(t), \dot{\gamma}(t))- c \cdot \dot{\gamma}(t) + \alpha_H(c) )dt  \right), 
\]
where $\gamma:[0,t] \to \T^d$  is absolutely continuous. Weak KAM solutions exist and are Lipschitz  (see \cite{Fat08}, \cite{Ber10}
). 

\item {\bf The relation between Lagrangians}\ 
We now turn to the weak KAM solutions of dominant Hamiltonians. Fix 
$\cB^\st$ and consider
\[
	(\cB^\wk, p, U^\st, \cU^\wk) \in \Omega^{m,d}(\cB^\st)
\]
and write $H^s = \cH^s(\cB^\st,\cB^\wk, p, U^\st, \cU^\wk)$, $H^\st = \cH^\st(p, \cB^\st,U^\st)$. Note 
\[
	H^s(\varphi^\st, \varphi^\wk, I^\st, I^\wk) = K(I^\st, I^\wk) - U^\st(\varphi^\st) - U^\wk(\varphi^\st,\varphi^\wk), 
\]
where $U^\wk = \sum_{j=1}^{d-m} U_j^\wk$, and 
\[
	H^\st(\varphi^\st, I^\st) = K(I^\st, 0) - U^\st(\varphi^\st). 
\]
Denote $L^s = L_{H^s}$ and $L^\st = L_{H^\st}$, we have 
\[
	L^s(\varphi^\st, \varphi^\wk, v^\st, v^\wk) = L_0^s(v^\st, v^\wk) +  U^\st(\varphi^\st) + U^\wk(\varphi^\st,\varphi^\wk), 
\]
\[
	L^\st(\varphi^\st, v^\st) = L_0^\st(v^\st) + U^\st(\varphi^\st), 
\]
where $L_0^s$, $L_0^\st$ are quadratic functions with $(\partial^2_{vv}L_0^s) = (\partial^2_{II} K)^{-1}$ and $(\partial^2_{v^\st v^\st}L_0^\st) = (\partial^2_{I^\st I^\st}K)^{-1}$ as matrices.

 Given $c = (c^\st, c^\wk) \in \R^m \times \R^{d-m} = \R^d$, we show that the weak KAM solution of $L^s - c\cdot v$ is related to the weak KAM solution of $L^\st - \bar{c}\cdot v^\st$, where $\bar{c}$ is computed using an explicit formula. 
More precisely, we define 
\[	
	\bar{c} = c^\st + A^{-1}B c^\wk,
\]
where $\partial^2_{II}K = \bmat{A & B \\ B^T & C}$  as in \eqref{eq:ABC}. Then (we refer to section~\ref{sec:slow-lag} for details)
\begin{align*}
	&L^s(\varphi^\st, \varphi^\wk, v^\st, v^\wk) - (c^\st, c^\wk) \cdot (v^\st, v^\wk)  = L^\st(\varphi^\st, v^\st) - \bar{c} \cdot v^\st  \\
	&+ \frac12 (v^\wk - B^T A^{-1}v^\st - \tilde{C} c^\wk) \cdot \tilde{C}^{-1} (v^\wk - B^T A^{-1}v^\st - \tilde{C} c^\wk) \\
	&+ \frac12 c^\wk \cdot \tilde{C} c^\wk + U^\wk(\varphi^\st, \varphi^\wk),
\end{align*}
where $\tilde C=C-B^TA^{-1}B$. 
The above computation suggests a connection between the Lagrangian $L^s - c \cdot v$ and $L^\st - \bar{c} \cdot v^\st$. Indeed, in Proposition~\ref{prop:Ls-alpha} we show
\[
	\alpha_{H^s}(c) - \|U^\wk\|_{C^0} \le   \alpha_{H^\st}(\bar{c}) + \frac12 c^\wk \cdot \tilde{C}c^\wk \le \alpha_{H^s}(c)+ \|U^\wk\|_{C^0}. 
\]

\item {\bf Semi-continuity of weak KAM solutions}

We now state our main variational results. We consider a sequence  of dominant Hamiltonians with  $\lM(\cB^\wk)\to \infty$, and cohomology classes $c_i$ such that the corresponding $\bar{c}_i$ converge. Then the weak KAM solutions has a converging subsequence, and the limit point is the weak KAM solution of the strong Hamiltonian. This is sometimes referred to as upper semi-continuity. 

\begin{theorem}\label{thm:semi-cont}
Fix $\cB^\st$ and $\kappa>1$. Assume that $q> 2(d-m)$.

For $p_0 \in \R^n$, $U^\st_0 \in C^2(\T^m)$ and $\bar{c}\in \R^m$, we consider a sequence 
\[
	(\cB^\wk_i, p_i, U^\st_i, \cU^\wk_i) \in 
\Omega^{m,d}_{\kappa, q}(\cB^\st), \quad 
c_i = (c_i^\st, c_i^\wk) \in \R^m\times \R^{d-m}, 
\]
and let $u_i$ be a weak KAM solution of
\[
L_{\cH^s(\cB^\st,\cB^\wk_i, p_i, U^\st_i, \cU^\wk_i)} - c_i \cdot v.
\]
Denote $K_i = K_{p_i, \cB^\st, \cB_i^\wk}$, and 
\[
	A_i = \partial^2_{I^\st I^\st}K_i, B_i = \partial^2_{I^\st I^\wk}K_i, C_i = \partial^2_{I^\wk I^\wk}K_i. 
\]
Assume:
\begin{itemize}
\item $\lM(B_i^\wk) \to \infty$, $p_i \to p_0$, $U^\st_i \to U^\st_0$.
\item $c_i^\st + A_i^{-1}B_i c_i^\wk \to \bar{c}$.
\end{itemize}
Then: 
\begin{enumerate}
\item The sequence $\{u_i\}$ is equi-continuous. In particular, the sequence $\{u_i(\cdot) - u_i(0)\}$ is pre-compact in the $C^0$ topology.
\item Let $u$ be any accumulation point of the sequence $u_i(\cdot) - u_i(0)$. Then there exists $u^\st: \T^m \to \R$ such that $u(\varphi^\st, \varphi^\wk) = u^\st(\varphi^\st)$, i.e, $u$ is independent of $\varphi^\wk$.
\item $u^\st$ is a weak KAM solution of 
\[
	L_{\cH^\st(p_0, U_0^\st)} - \bar{c} \cdot v^\st.
\]
\end{enumerate}
\end{theorem}
The proof of Theorem~\ref{thm:semi-cont} occupies sections \ref{sec:strong-slow} and \ref{sec:con-weak-kam}, with some 
technical statements deferred to section~\ref{sec:tech-est}.
\end{itemize}
\begin{remark}
Theorem~\ref{thm:basis-norm} implies that by choosing a good basis, we can express a slow system as a dominant system with parameters $\kappa, q$, where $q = r- n - 2(d-m) -4$. For Theorem~\ref{thm:resc-est} we need $q>2$, and for  Theorem~\ref{thm:semi-cont} we need $q > 2(d-m)$. Therefore for application to nearly integrable systems, we need $r > n + 4(d-m) + 4$ as stated in our main result. 
\end{remark}

Using the point of view in \cite{Ber10}, the semi-continuity of the weak KAM solution is closely related to the semi-continuity of the Aubry and Ma\~ne sets. These properties have important applications to Arnold diffusion.  In section~\ref{sec:cor-aub} we develop an analog of these results for the dominant Hamiltonians. 

\section{The choice of basis and averaging}\label{sec:basis}

In this section we prove Theorem~\ref{thm:basis-norm}. The proof consists of two parts: the choice of the basis and estimates on the norms.

\subsection{The choice of the basis}

Recall that we have a fixed irreducible lattice $\Lambda^\st \subset \Z^{n+1}$ of rank $m < n$, and a fixed basis $\cB^\st = \{k_1, \cdots, k_m\}$ for $\Lambda^\st$. The following proposition describes the choice of the adapted basis for any irreducible $\Lambda \supset \Lambda^\st$.

\begin{proposition}\label{prop:basis}
Let $\Lambda^\st \subset \Z^{n+1}$ be an irreducble  lattice of rank $m < n$, and fix a basis $k_1, \cdots, k_m$. Let $\Lambda \supset \Lambda^\st$ be an irreducible lattice of rank $m<d\le n$, then there exists $k_{m+1}, \cdots, k_d \in \Z^{n+1}$ such that $k_1, \cdots, k_d$ form a basis of $\Lambda$, and the following hold. 
\begin{enumerate}
\item For each $m < j \le d$, 
\[
	|k_j| \le \bar{M} + (d-m)M_j,
\]
where 
\[
	\bar{M} = |k_1| + \cdots + |k_m|, \quad \Lambda_j = \Span_\Z\{ k_1, \cdots, k_j\}, \quad  M_j = M(\Lambda_j | \Lambda_{j-1}). 
\]
\item For each $m < i < j \le d$, 
\[
	|k_i| \le \bar{M} + (d-m) |k_j|. 
\]
\end{enumerate}
\end{proposition}

We now describe the choice of the vectors $k_{m+1}, \cdots, k_d$. We define
$k_i' = k_i$ for $1\le i \le m$, and define $k_i'$ with $i >m$ inductively using the following procedure. Suppose $k_1', \cdots, k_i'$ are defined, let 
\[
	\Lambda_i = \Span_\R\{k_1', \cdots, k_i'\} \cap \Lambda, \quad
	M_{i+1} = \min\{ |k|: \quad k \in \Lambda \setminus \Lambda_i\}.
\]
We define $k_{i+1}'$ to be a vector reaching the minimum in the definition of $M_{i+1}$, i.e $|k_{i+1}'| = M_{i+1}$. We have 
\[
	|k_i'| = M_i, \, m < i \le d, \quad |k_j'| \le |k_i'|, \, m < j < i \le d,
\]
but $k_1', \cdots, k_d'$ may not form a basis. We turn them into a basis using the following procedure (see \cite{Sie}). 

 For each $j=1, \cdots, m$, define 
\begin{equation}\label{eq:ci}
	c_j = \min\{s_j: \quad  s_{j,1} k_1' + \cdots + s_{j , j-1} k_{j-1}' + s_j k_j' \in \Lambda, \, s_j \in \R^+,  \,  s_{j,i} \in \R^+ \cup \{0\}\}.
\end{equation}
We define $c_{j,j-1}$ using a similar minimization given the value $c_j$:
\[
	c_{j,j-1} =  \min\{s_{j,j-1}: \quad 
    s_{j,1} k_1' + \cdots + s_{j,j-1} k_{j-1}' + c_j k_j', \, s_{j,i} \in \R^+\cup\{0\}
	\}.
\]
We now define $c_{j,i}$ for $1\le i\le j-2$ inductively as follows. Assume that $c_{j,i}, \cdots, c_{j,j-1}$ are all defined, then
\begin{multline*}
	c_{j,i-1} = \min\{s_{j,i-1}: \\
    s_{j,1} k_1' + \cdots + s_{j,i-1} k_{i-1}' +  c_{j,i} k_i' + \cdots + 
	c_{j,j-1} k_{j-1}' +  c_j k_j'\in \Lambda, \\
     s_{j,1}, \cdots, s_{j,i-1} \in \R^+\cup\{0\}
	\}.
\end{multline*}
Finally, 
\[
	k_j = c_{j,1} k_1' + \cdots + c_{j, j-1} k_{j-1}' + c_j k_j'. 
\]
We have the following lemma from the geometry of numbers.

\begin{lemma}[see \cite{Sie}]\label{lem:cons-min}
Let $\Lambda\subset \Z^{n+1}$ be a lattice of rank $d\le n$ and let $k_1', \cdots, k_d'$ be any linearly independent set in $\Lambda$. Let 
\[
	k_j = c_{j,1} k_1' + \cdots + c_{j, j-1} k_{j-1}' + c_j k_j'. 
\]
be defined using the procedure above. Then
\begin{enumerate}
\item For each $1 \le j \le d$, $k_1, \cdots, k_j$ form a basis of $\Span_\R\{k_1', \cdots, k_j'\}\cap \Lambda$ over $\Z$. In particular, 
$k_1, \cdots, k_d$ form a basis of $\Lambda$. 
\item For $1 \le j < d$ and $1 \le i \le j-1$, we have 
\[
	0 \le c_{j,i} <1, \quad 0 < c_j \le 1. 
\]
\item If for some $1\le m \le d$, $k_1', \cdots, k_m'$ already form a basis 
of $\Span_\R\{k_1', \cdots, k_m'\}\cap \Lambda$ over $\Z$, then 
$k_1 = k_1', \cdots, k_m = k_m'$. 
\end{enumerate}
\end{lemma}

\begin{proof}



For proof of item 1, we refer to \cite{Sie}, Theorem 18. Item 2 and 3 follow from definition and item 1 as we explain below.

For item 2, note that for any 
$k_j = c_{j,1} k_1' + \cdots + c_{j, j-1} k_{j-1}' + c_j k_j' \in \Lambda$, 
we can always subtract an integer from any $c_{j,i}$ or $c_j$ and remain 
in $\Lambda$. If the estimates do not hold, we can get a contradiction by reducing $c_{j,i}$ or $c_j$. 

For item 3, if $k_1', \cdots, k_m'$ is a basis (over $\Z$) of $\Span_\R\{k_1', \cdots, k_m'\}\cap \Lambda$, then all coefficients of $k_j = c_{j,1} k_1' + \cdots + c_{j, j-1} k_{j-1}' + c_j k_j' \in \Lambda$ for $j \le m$ must be integers. Then the constraints of item 2 implies $c_{j,i}=0$ and $c_j =1$, namely $k_j = k_j'$. 
\end{proof}

\begin{proof}[Proof of Proposition~\ref{prop:basis}]
We choose the basis $k_1, \cdots, k_d$ as described.  Lemma~\ref{lem:cons-min} implies $k_j = k_j'$ for $1\le j \le m$. Using  
\[
	0 < c_{j+1} \le 1, \quad 0 \le c_{j+1, i} < 1,
\]
we get
\[
	|k_j| \le |k_1'| + \cdots + |k_j'| = |k_1| + \cdots+ |k_m| + M_{m+1} + \cdots + M_j. 
\]
Since $M_{m+1} \le \cdots \le M_d$, and $\bar{M} = |k_1| + \cdots+ |k_m|$, we get 
\[
	|k_j| \le \bar{M} + (j-m)M_j \le \bar{M} + (d-m) M_j. 
\]
Moreover, for $i < j$, we have 
\[
|k_i| \le \bar{M} + (d-m)M_i < \bar{M} + (d-m)M_j \le \bar{M} + (d-m)|k_j|.
\]
\end{proof}
We note that the basis, as chosen  in Proposition~\ref{prop:basis}, 
satisfies item 1 of Theorem~\ref{thm:basis-norm} for 
$\kappa \ge \max\{\bar{M}, d-m\}$. 

\subsection{Estimating the weak potential}

In this section we prove the second item in Theorem~\ref{thm:basis-norm} and conclude its proof. Assume that $H_1 \in C^r(\T^n\times \R^n \times \T)$ with $r > n + 2d - 2m + 4$. Let the basis $k_1, \cdots, k_n$ be chosen as in Proposition~\ref{prop:basis}. We show that there exists $\kappa = \kappa(\cB^\st, Q, n)>1$ such that for $m < i \le d$,
   \[
   \|U^\wk_{p_0,\cB_{i-1}, \cB_i}\|_{C^2} \le \|Z_{\cB_i} - Z_{\cB_{i-1}}\|_{C^2} \le \kappa |k_i|^{-r + 3n -2m +6}. 
   \]
For a lattice $\Lambda$ let 
\[
	[H]_\Lambda(\theta, p, t) = \sum_{k \in \Lambda} h_k(p) e^{2\pi i k \cdot (\theta, t)}, 
\]
then we have 
\[
	(Z_{\cB_i} - Z_{\cB_{i-1}})(k_1 \cdot (\theta, t), \cdots, k_i \cdot (\theta, t),p) = ([H_1]_{\Lambda_i}- [H_1]_{\Lambda_{i-1}} )(\theta,p, t),
\]
and the norm of $[H_1]_{\Lambda_i}- [H_1]_{\Lambda_{i-1}} $ can be estimated using a standard estimates of the Fourier series.
\begin{lemma}[c.f. \cite{BKZ11}, Lemma 2.1, item 3]\label{lem:fourier-bd}
Let $H_1(\theta, p, t) = \sum_{k\in \Z^{n+1}} h_k(p) e^{2\pi i k \cdot (\theta, t)}$  satisfy $\|H_1\|_{C^r} =1$, with $r\ge n+4$. There exists a constant $C_n$ depending only on $n$, such that for any subset  $\tilde{\Lambda}\subset \Z^{n+1}$  with
$\min_{k \in \tilde{\Lambda}} |k| = M >0$, we have 
\[
	\|\sum_{k\in \tilde{\Lambda}} h_k(p) e^{2\pi i k \cdot (\theta, t)}\|_{C^2} \le C_n M^{-r+n+4}.
\]
\end{lemma}

Since $\min_{k\in \Lambda_i \setminus \Lambda_{i-1}}|k| = M_i$, we apply Lemma~\ref{lem:fourier-bd} to $\Lambda_i \setminus \Lambda_{i-1}$ to get 
\begin{equation}
\label{eq:hl-norm}
	\| [H_1]_{\Lambda_i}- [H_1]_{\Lambda_{i-1}} \|_{C^2} \le C_n M_i^{-r+n+3}.
\end{equation}
To estimate $Z_{\cB_i} - Z_{\cB_{i-1}}$, we apply a linear coordinate change. Given $k_1, \cdots, k_i$, we choose $\hat{k}_{i+1}, \cdots, \hat{k}_{n+1}\in \Z^{n+1}$ to be coordinate vectors (unit integer vectors) so that 
\[P_i := \bmat{k_1& \cdots& k_i& \hat{k}_{i+1}& \cdots& \hat{k}_{n+1}}\]
is invertible. We extend $(Z_{\cB_i} - Z_{\cB_{i-1}})(\varphi_1, \cdots, \varphi_i)$ trivially to a function of $(\varphi_1, \cdots, \varphi_{n+1})$, then 
\[
	(Z_{\cB_i} - Z_{\cB_{i-1}})\left(P_i^T \bmat{\theta \\ t} \right) = ([H_1]_{\Lambda_i}- [H_1]_{\Lambda_{i-1}})(\theta,t). 
\]
We get 
\[
	\|Z_{\cB_i} - Z_{\cB_{i-1}}\|_{C^2} \le (1+\|P_i^{-1}\| ) (1+\|(P_i^T)^{-1}\|)  \| [H_1]_{\Lambda_i}- [H_1]_{\Lambda_{i-1}} \|_{C^2}. 
\]
We apply the following lemma in linear algebra:
\begin{lemma}\label{lem:int-norm}
Given $1 \le s \le n+1$, let $P = \bmat{k_1 & \cdots & k_s}$ be an integer matrix with linearly independent columns. Then there exists $c_n>1$ depending only on $n$ such that 
\[
	\min_{\|v\| = 1} \|Pv\| = \min_{\|v\| =1}(v^T P^T Pv)^{\frac12} = \|(P^TP)^{-1}\|^{-\frac12} \ge c_n^{-1} |k_1|^{-1} \cdots |k_m|^{-1}. 
\]
In particular, if $s = n+1$, then $\|P^{-1}\| = \|(P^T)^{-1}\| \le c_n |k_1| \cdots |k_{n+1}|$. 
\end{lemma}
\begin{proof}
We only estimate $\|(P^TP)^{-1}\|$. Let $a_{ij} = (P^TP)_{ij}$ and $b_{ij} = (P^TP)^{-1}_{ij}$, then using Cramer's rule and the definition of the cofactor, we have 
\[
	|b_{ij}| \le  \frac{1}{\det(P^TP)}\sum_{\sigma} \prod_{s \ne i} a_{s \sigma(s)},
\]
where $\sigma$ ranges over all one-to-one mappings from $\{1, \cdots, m\} \setminus \{i\}$ to $\{1, \cdots, m\} \setminus \{j\}$. Since $P$ is a nonsingular integer matrix, we have $\det(P^TP)\ge 1$. Moreover, $a_{ij} = k_i^T k_j \le n |k_i| |k_j|$. Therefore
\[
	|b{ij}| \le \sum_{\sigma}\prod_{s \ne i} |k_s| |k_{\sigma(s)}| \le c_n (\prod_{s \ne i} |k_s|)(\prod_{s \ne j} |k_s|), 
\]
where $c_n$ is a constant depending only on $n$. Using the fact that the norm of a matrix is bounded by its largest entry, up to a factor depending only on dimension, by changing to a different $c_n$, we have 
\[
	\|(P^TP)^{-1}\| \le c_n \sup_{i,j} |B_{ij}| \le c_n \sup_{i,j} (\prod_{s \ne i} |k_s|)(\prod_{s \ne j} |k_s|) \le c_n (\prod_{s=1}^m |k_s|)^2. 
\]
If $s=n+1$, then $\|P^{-1}\| = \|(P^TP)^{-1}\|^{\frac12} = \|(P P^T)^{-1}\|^\frac12 = \|(P^T)^{-1}\|$. 
\end{proof}

Using Lemma~\ref{lem:int-norm}, there exists a constant $c_n>0$  depending only on $n$ such that 
\[
	\|P_i^{-1}\| = \|(P_i^T)^{-1}\| \le c_n |k_1| \cdots |k_i| |\hat{k}_{i+1}| \cdots |\hat{k}_{n+1}| .
\]
We have $|k_1|, \cdots, |k_m| \le \bar{M}$, $|\hat{k}_{i+1}|=\cdots = |\hat{k}_{n+1}| =1$, and from Lemma~\ref{lem:cons-min}, $|k_{m+1}|, \cdots, |k_i| \le \bar{M} + (d-m)M_i$. Hence there exists a constant $c_{n, \bar{M}}>0$  such that 
\[
	\|P_i^{-1}\| = \|(P_i^T)^{-1}\| \le c_{n,\bar{M}} M_i^{i-m}. 
\]
Combine with \eqref{eq:hl-norm}, we get for $\kappa = \kappa(n, \bar{M})$,
\[
	\|Z_{\cB_i} - Z_{\cB_{i-1}}\|_{C^2} \le \kappa M_i^{-r+n+4 + 2(i-m)} \le \kappa M_i^{-r+n+4 + 2(d-m)} \le \kappa |k_i|^{-r+n + 2d - 2m +4}. 
\]
This implies item 2 of Theorem~\ref{thm:basis-norm}. The proof is complete. 

\section{Strong and slow systems of dominant 
Hamiltonians}\label{sec:strong-slow}

In this section we study the relation between Hamiltonians and 
the corresponding Lagrangians for dominant systems. 
We start by comparing the Hamitonian vector fields 
and then compare their Lagrangians. 
 
\subsection{Vector fields of dominant Hamiltonians}\label{sec:vector-field}

In this section we expand on section~\ref{sec:intro-rescale} and prove Theorem~\ref{thm:resc-est}. 
Fix	$\cB^\st, \kappa>1$ and let $(\cB^\wk, p, U^\st, \cU^\wk) \in \Omega^{m,d}_{\kappa, q}(\cB^\st)$, we recall the notations
\[
	H^s  = \cH^s(\cB^\st,\cB^\wk, p, U^\st, \cU^\wk),   \quad H^\st = \cH^\st(p, U^\st).
\]
Then
\[
	H^s(\varphi, I) = K(I) - U^\st(\varphi^\st) - U^\wk(\varphi^\st, \varphi^\wk), \quad  H^\st(\varphi^\st, I^\st) = K(I^\st, 0) - U^\st(\varphi^\st). 
\]
Recall from \eqref{eq:block-ABC} that $\partial^2_{II}K = \bmat{A & B \\ C & D}$, then 
\[
	(A)_{ij} = (k_i^\st)^TQk_j^\st, \quad (B)_{ij} = (k_i^\st)^T Q k_j^\wk, \quad 
	(C)_{ij} = (k_i^\wk)^T Q k_j^\wk. 
\]

The vector field $X^s(\varphi^\st, v^\st,  \varphi^\wk, I^\wk) $ defined on the universal cover $\R^m  \times \R^m \times \R^{d-m}\times \R^{d-m}$ is obtained from the Lagrangian vector field via  the coordinate change $\tilde{C}I^\wk = B^T A^{-1}v^\st - v^\wk$ (see \eqref{eq:half-lag}). The vector field $X^\st_L(\varphi^\st, v^\st, \varphi^\wk,  I^\wk) $ is defined as a trivial extension of the Lagrangian vector field of $H^\st$, also defined on the universal cover. More explicitly (see \eqref{eq:Xs}, \eqref{eq:XstL})
\begin{equation}
\label{eq:XsXstl}
	X^s= 
	\bmat{ v^\st \\ 
	A\partial_{\varphi^\st}U + B \partial_{\varphi^\wk} U \\ 
	B^T A^{-1} v^\st - \tilde{C} I^\wk \\
	\partial_{\varphi^\wk} U}, 
	 \quad 
	X^\st_L = 
	\bmat{v^\st \\ A \partial_{\varphi^\st}U^\st \\0 \\  0 }. 
\end{equation}

Given $1\ge  \sigma_1 \ge \cdots \ge \sigma_{d-m}>0$, let $\Sigma = \diag\{\sigma_1, \cdots, \sigma_{d-m}\}$. The rescaling is  $\Phi_\Sigma: \R^{2d}\to \R^{2d}$, given by (\ref{rescale}).
 We denote by $\tilde{X}^s(\varphi^\st, v^\st, \tilphi^\wk, \tilI^\wk)$ the rescaled $X^s$. Using  \eqref{eq:XsXstl}, we have
\begin{equation}\label{eq:xs-diff}
	\tilde{X}^s - X^\st_L = (\Phi_\Sigma)^{-1} X^s \circ \Phi_\Sigma^{-1} - X^\st =
	\begin{bmatrix}
	0 \\
	(A \partial_{\varphi^\st} U^\wk  + B \partial_{\varphi^\wk} U^\wk)(\varphi^\st, \Sigma^{-1} \tilde{\varphi}^\wk) \\	
	\Sigma B^T A^{-1} v^\st - \Sigma \tilde{C} \Sigma \tilI^\wk \\
	\Sigma^{-1}\partial_{\varphi^\wk} U^\wk (\varphi^\st, \Sigma^{-1} \tilde{\varphi}^\wk)
	\end{bmatrix}  
\end{equation}
noting that $U^\st$ is independent of $\varphi^\wk$, so $\partial_{\varphi^\wk}U = \partial_{\varphi^\wk} U^\wk$. 
Furthermore
\begin{multline}\label{eq:dxs-diff}
	D(\tilde{X}^s - X^\st_L) = (\Phi_\Sigma)^{-1} D X^s  \circ (\Phi_\Sigma)^{-1} - X^\st_L = \\
	\begin{bmatrix}
	0 & 0 & 0 & 0 \\
	A \partial^2_{\varphi^\st \varphi^\st}U^\wk  + B  \partial^2_{\varphi^\st \varphi^\wk} U^\wk & 0 & B \partial^2_{\varphi^\wk \varphi^\wk}U^\wk \Sigma^{-1} & 0 \\
	0 & \Sigma B^T A^{-1} & 0 & \Sigma \tilde{C} \Sigma \\
	\Sigma^{-1} \partial^2_{\varphi^\st \varphi^\wk} U^\wk & 0 &\Sigma^{-1}\partial^2_{\varphi^\wk \varphi^\wk} U^\wk \Sigma^{-1} & 0 
	\end{bmatrix}. 
\end{multline}


The quantities in \eqref{eq:xs-diff} and \eqref{eq:dxs-diff} are estimated as follows. 
\begin{lemma}\label{lem:matrix-bound} Fix $\cB^\st, \kappa>1$.
Assume $q>2$. Then there exists a constant $M_1= M_1(\cB^\st, Q, \kappa, q,d-m)$ such that for the  parameters $\sigma_j = |k_j^\wk|^{-\frac{q+1}{3}}$, uniformly over $\R^m\times \R^m\times \R^{d-m}\times \R^{d-m}$, the following hold. 
\begin{enumerate}
\item For any $ 1\le i \le m$ and $1 \le j \le d-m$, $\|\partial_{\varphi^\wk_j}U^\wk\|_{C^0}, \|\partial^2_{\varphi_i^\st  \varphi_j^\wk} U^\wk\|_{C^0}\le M_1 |k_j^\wk|^{-q}$; 

for any $1 \le i, j \le d-m$, $\|\partial^2_{\varphi_i^\wk \varphi_j^\wk} U^\wk\|_{C^0} \le M_1 \sup\{ |k_i^\wk|^{-q}, |k_j^\wk|^{-q}\}$. 
\item $ \|A \partial_{\varphi^\wk}U^\wk\|_{C^0} \le M_1 \sup_j\{|k_j^\wk|^{-q}\}$,  $\| A \partial^2_{\varphi^\st \varphi^\st}U^\wk \|_{C^0} \le M_1 \sup_j \{|k_j^\wk|^{-q}\}$. 
\item $ \|B \partial_{\varphi^\wk} U^\wk\|_{C^0} \le M_1 \sup_j \{ |k_j^\wk|^{-(q-1)}\}$, $ \|B \partial^2_{\varphi^\st \varphi^\wk} U \|_{C^0}\le  M_1 \sup_j \{ |k_j^\wk|^{-(q-1)}\}$ . 

\item $\|B \partial^2_{\varphi^\st \varphi^\wk} U \Sigma^{-1}\|_{C^0} \le  M_1 \sup_j \{ |k_j^\wk|^{-\frac{2q-4}{3}}\}$. 
\item $\| \Sigma^{-1} \partial^2_{\varphi^\wk \varphi^\wk}U \Sigma^{-1}\|_{C^0} \le M_1 \sup_j \{ |k_j^\wk|^{-\frac{q-2}{3}}\}$. 
\item $\|\Sigma B^T A^{-1}\|_{C^0} \le M_1 \sup_j \{ |k_j^\wk|^{-\frac{q-2}{3}}\}$. 
\item $\|\Sigma \tilde{C} \Sigma\|_{C^0} \le M_1 \sup_j \{ |k_j^\wk|^{-\frac{2q-4}{3}}\}$. 
\end{enumerate}
\end{lemma}
We first prove Theorem~\ref{thm:resc-est} using our lemma. 
\begin{proof}[Proof of Theorem~\ref{thm:resc-est}]
Noting that $\Pi_{(\varphi^\st, v^\st)}(\tilde{X}^s - X^\st_L)$ is the first and third line of \eqref{eq:xs-diff}, using item 2 and 3 of Lemma~\ref{lem:matrix-bound} we get 
\[
	\|\Pi_{(\varphi^\st, v^\st)}(\tilde{X}^s - X^\st_L) \| \le M \sup_j \{|k_j^\wk|^{-(q-1)}\} = M^* \lM(\cB^\wk)^{-(q-1)},
\]
for any constant  $M \ge 2 M^*_1$, where $M_1$ is from Lemma~\ref{lem:matrix-bound}. 

Since $D(\tilde{X}^s - X^\st_L)$ is bounded, up to a universal constant, the sum of the norms of all the non-zero blocks in \eqref{eq:dxs-diff}, using Lemma~\ref{lem:matrix-bound} items 4-8, we get 
\[
		\|D\tilde{X}^s - DX_L^\st\| \le M \sup_j \{|k_j^\wk|^{-\frac{q-2}{3}}\}  =  M^* \lM(\cB^\wk)^{-\frac{q-2}{3}},
\]
where $M$ depends only on $M_1$. 
\end{proof}

The rest of the section is dedicated to proving Lemma~\ref{lem:matrix-bound}. 

\begin{proof}[Proof of Lemma~\ref{lem:matrix-bound}]
Denote $\bar{M} = |k_1^\st| + \cdots + |k_m^\st|$, which depends only on $\cB^\st$. 
\myheading{Item 1}. We have 
\begin{multline*}
	\|\partial_{\varphi_j^\wk} U^\wk\|_{C^0} \le \sum_{l=1}^{d-m} \|\partial_{\varphi^\wk_j}U_l^\wk\|_{C^0} \le \sum_{l \ge j} \|\partial_{\varphi^\wk_j}U_l^\wk\|_{C^0} \\
	\le \kappa \sum_{l\ge j} |k_l^\wk|^{-q} \le (d-m) \kappa^{q+1} |k_j^\wk|^{-q}, 
\end{multline*}
where the second inequality is due to $U_l^\wk$ depending only on $(\varphi_1^\wk, \cdots, \varphi_l^\wk)$, and the last two inequalities uses the definition of $\Omega^{m,d}_{\kappa, q}$, see section~\ref{sec:intro-dom}.  
By the same reasoning, we have 
	\[
	\|\partial^2_{\varphi_i^\st \varphi_j^\wk}U^\wk\| \le \sum_{l \ge j}\|U_l^\wk\|_{C^2} \le (d-m) \kappa^{q+1} |k_j^\wk|^{-q}, 
	\]
	\[
	\| \partial^2_{\varphi_i^\wk \varphi_j^\wk}U^\wk\| \le \sum_{l \ge \sup\{i,j\}} \|U_l^\wk\|_{C^2} \le (d-m) \kappa^{q+1} \sup\{ |k_i^\wk|^{-q}, |k_j^\wk|^{-q}\}
	\]
the second and third estimate follows.

\myheading{Item 2}.We have 
\begin{multline*}
	|(A \partial_{\varphi^\st} U^\wk)_j| = |\sum_i(k_i^\st)^T Q k_j^\st \partial_{\varphi_j^\st}U^\wk| \le  m\bar{M}^2 \|Q\| \|\partial_{\varphi_j^\st}U\| \\
	\le  m(d-m) \bar{M}^2 \|Q\| \kappa^{q+1}|k_j^\wk|^{-q},
\end{multline*}
where the last line is due to item 1.
Similarly, 
\[
	|(A \partial^2_{\varphi^\st \varphi^\st}U^\wk )_{ij}| \le \bar{M}^2 \|Q\| \|\partial_{\varphi^\st_i\varphi_j^\st}U\|\le  (d-m) \bar{M}^2 \|Q\| \kappa^{q+1}|k_j^\wk|^{-q}. 
\]
 Since the vector or matrix norm is bounded by the supremum of all matrix entries, up to a constant depending only on dimension, item 2 follows. In the sequel, we apply the same reasoning and only estimate the supremum of matrix/vector entries. 

\myheading{Item 3}. Similar to item 2,
\begin{multline*}
	|(B\partial_{\varphi^\wk} U)_{j}| = |\sum_i(k_i^\st)^T Q k_j^\wk \partial_{\varphi^\wk_j} U^\wk| \le (d-m)\bar{M}\|Q\| |k_j^\wk|  \|\partial_{\varphi_j^\wk} U^\wk\|  \\
	 \le (d-m)\bar{M} \|Q\|(d-m) \kappa^{q+1} |k_j^\wk|^{-(q-1)}, 
\end{multline*}
while 
\[
	|(B \partial^2_{\varphi^\st\varphi^\wk} U^\wk)_{ij}| = |(k_i^\st)^T Q k_j^\wk \partial^2_{\varphi^\st_i\varphi^\wk_j} U^\wk| \le (d-m)\kappa^{q+1}\bar{M} \|Q\| |k_j^\wk|^{-(q-1)}. 
\]

\myheading{Item 5}. 
\begin{multline*}
	|(B\partial^2_{\varphi^\wk \varphi^\wk}U\Sigma^{-1} )_{ij}| = |\sum_l (k_i^\st)^T Q k_l^\wk  \partial^2_{\varphi^\wk_l \varphi^\wk_j} U \sigma_j^{-1}| 
	\le \bar{M}\|Q\| \sum_{l \ge j} |k_l^\wk| \sigma_j^{-1} |\partial^2_{\varphi_l^\wk \varphi_j^\wk} U| 	\\
	 \le  \bar{M}\|Q\|  (d-m)^2 \kappa^{q+2}  |k_j^\wk| |k_j^\wk|^{-q} |k_j^\wk|^{\frac{q+1}{3}} =  \bar{M}\|Q\|  (d-m) \kappa^{q+2}  |k_j^\wk|^{- \frac{2q-4}{3}},
\end{multline*}
where the inequality of the second line uses $|k_l^\wk| \le \kappa |k_j^\wk|$, item 1 and the choice of $\sigma_j$.

\myheading{Item 6}. Using item 1 and choice of $\sigma_j$, we have
\begin{multline*}
	|(\Sigma^{-1} \partial^2_{\varphi^\wk \varphi^\wk} U^\wk \Sigma^{-1})_{ij}| = |\sigma_i^{-1} \partial^2_{\varphi_i^\wk \varphi_j^\wk}U^\wk \sigma_j^{-1}| \\
	 \le (d-m) \kappa^{q+1}  \sigma_i^{-1} \sigma_j^{-1} \sup\{|k_i^\wk|^{-q}, |k_j^\wk|^{-q} \} \\
	 \le (d-m) \kappa^{q+1}  \sup \{ |k_i^\wk|^{- \frac{q-2}{3}}, |k_j^\wk|^{- \frac{q-2}{3}}\}. 
\end{multline*}

\myheading{Item 7}. We have
\[
	|(\Sigma B^T)_{ij}| = |\sigma_i (k_i^\wk)^T Q k_j^\st| \le \bar{M}\|Q\| \sup_j \{|k_j^\wk| \sigma_j\}  =  \bar{M}\|Q\| \sup_j \{ |k_j^\wk|^{-\frac{q-2}{3}}  \}  
\]
and uses $\|\Sigma B^T A^{-1}\| \le \|\Sigma B^T\| \|A^{-1}\|$, noting that $\|A^{-1}\|$  depends only on $Q$ and $\cB^\st$. 

\myheading{Item 8}. Recall 
$\tilde{C} = C - B^T A^{-1} B$. We have 
\[
	|(\Sigma C \Sigma)_{ij}| = |\sigma_i (k_i^\wk)^T Q k_j^\wk \sigma_j| \le (\sup_j \sigma_j |k_j^\wk|)^2 \|Q\| \le \|Q\| \sup_j \{ |k_j^\wk|^{- \frac{2q-4}{3}}\}.
\]
Suppose $S_1, S_2$ are  positive definite symmetric matrices with $S_1 \ge S_2$,  for any $v\in \R^{d-m}$,
\[
	v^T S_1 v = v^T (S_1 - S_2 + S_2) v \ge v^T S_2 v,
\]
we obtain $\|S_1\| \ge \|S_2\|$. Since $C - B^T A^{-1} B \ge 0$, we have $\Sigma C \Sigma - \Sigma B^T A^{-1} B  \Sigma \ge 0 $. Apply the observation to the matrices $\Sigma C \Sigma$ and $\Sigma B^T A^{-1} B  \Sigma$ we get 
\[
	\|\Sigma \tilde{C} \Sigma\|\le  \|\Sigma C \Sigma\| + \|\Sigma B^T A^{-1} B  \Sigma\| \le 2 \| \Sigma C \Sigma\|. 
\]
Item 8 follows. 
\end{proof}

\subsection{The slow Lagrangian}
\label{sec:slow-lag}

We derive the special form of the slow Lagrangian described in section~\ref{sec:intro-var}. We fix $\cB^\st, \ \kappa>1$ and 
$(\cB^\wk, p, U^\st, \cU^\wk) \in \Omega^{m,d}_{\kappa, q}(\cB^\st)$. 
Denote $H^s  = \cH^s(\cB^\st,\cB^\wk, p, U^\st, \cU^\wk)$,  $H^\st = \cH^\st(p, U^\st)$ and the associated Lagrangian is denoted $L^s$ and $L^\st$. 

As before we write 
\[
	H^s(\varphi, I) = K(I) - U^\st(\varphi^\st) - U^\wk(\varphi^\st, \varphi^\wk), \quad  H^\st(\varphi^\st, I^\st) = K(I^\st, 0) - U^\st(\varphi^\st), 
\]
and
\[
	L^s(\varphi, v) = L_0^s(v) + U^\st(\varphi^\st) + U^\wk(\varphi^\st, \varphi^\wk), \quad L^\st(\varphi^\st, v^\st) = 
L_0^\st(v^\st) + U^\st(\varphi^\st), 
\]
where $\partial^2_{vv}L_0^s = (\partial^2_{II}K)^{-1}$, $\partial^2_{v^\st v^\st} L_0^\st = (\partial^2_{I^\st I^\st}K)^{-1}$ for 
$v=\partial_I K$ and $v^\st=\partial_{I^\st}K$. Recall the notation
\[
	\partial^2_{II}K = 
	\begin{bmatrix}
	A & B \\ B^T & C
	\end{bmatrix}, \quad A = \partial^2_{I^\st I^\st}K,\ B = \partial^2_{I^\st I^\wk}K,\ C = \partial^2_{I^\wk I^\wk}K.
\]
\begin{lemma}\label{lem:ls-split}
With the above notations we have 
\begin{enumerate}
\item 
\begin{equation}
\label{eq:Ls-split}
\beal &L^s(v, \varphi) = L^\st(\varphi^\st,v^\st) + \\
&\qquad \quad \frac12 (v^\wk - B^T A^{-1}v^\st) \cdot  \tilde{C}^{-1} (v^\wk - B^T A^{-1}v^\st) + U^\wk(\varphi^\st, \varphi^\wk),
\enal 
\end{equation}
where 
\[
	\tilde{C} = C- B^T A^{-1} B. 
\]
\item Let $c = (c^\st, c^\wk) \in \R^m \times \R^{d-m}$. We denote \footnote{We stress here that no coordinate change is performed: $w^\wk$ is simply an abbreviation for $v^\wk - B^T A^{-1} v^\st$. }
\be \label{c-bar}
	\bar{c} = c^\st + A^{-1} B c^\wk, \quad w^\wk = v^\wk - B^T A^{-1}v^\st,
\ee
then 
\begin{multline}
\label{eq:Lsc-split}
L^s(v, \varphi) - c \cdot v = L^\st(\varphi^\st, v^\st) - \bar{c} \cdot v^\st  \\
+ \frac12 (w^\wk -  \tilde{C} c^\wk) \cdot \tilde{C}^{-1} (w^\wk - \tilde{C}c^\wk) - \frac12 c^\wk \cdot \tilde{C} c^\wk + U^\wk(\varphi^\wk, \varphi^\st). 
\end{multline}
\end{enumerate}
\end{lemma}
\begin{proof}
We have the following identity in block matrix inverse, which can be verified by a direct computation. 
\[
	\bmat{A & B \\ B^T & C}^{-1} = \bmat{A^{-1} & 0 \\ 0 & 0} + \bmat{-A^{-1}B  \\ \Id}  \tilde{C}^{-1} \bmat{ - B^T A^{-1} & \Id}. 
\]
Then 
\begin{align*}
L_0^s(v^\st, v^\wk) &= \frac12 \bmat{ (v^\st)^T (v^\wk)^T} \left(  \bmat{A^{-1} & 0 \\ 0 & 0} + \bmat{\Id \\ -A^{-1}B} \tilde{C}^{-1} \bmat{\Id & - B^T A^{-1}} \right) \bmat{v^\st \\ v^\wk} \\
 &= \frac12 v^\st \cdot A^{-1} v^\st +  \frac12 (v^\wk - B^T A^{-1}v^\st) \cdot  \tilde{C}^{-1} (v^\wk - B^T A^{-1}v^\st) \\
 & = L_0^\st(v^\st) +  \frac12 (v^\wk - B^T A^{-1}v^\st) \cdot  \tilde{C}^{-1} (v^\wk - B^T A^{-1}v^\st),
\end{align*}
and \eqref{eq:Ls-split} follows. 

Moreover, 
\begin{align*}
&L_0^s - (c^\st, c^\wk)\cdot (v^\st, v^\wk) \\
&= L_0^\st(v^\st) - (c^\st + A^{-1}B c^\wk) \cdot v^\st +  \frac12 w^\wk \cdot \tilde{C}^{-1} w^\wk - c^\wk \cdot v^\wk + A^{-1}B c^\wk \cdot v^\st \\
& = L_0^\st(v^\st) - \bar{c} \cdot v^\st + \frac12 w^\wk \cdot \tilde{C}^{-1} w^\wk - c^\wk \cdot (v^\wk - B^T A^{-1}v^\st) \\
&= L_0^\st(v^\st) - \bar{c} \cdot v^\st + \frac12 w^\wk \cdot \tilde{C}^{-1} w^\wk - ( \tilde{C} c^\wk) \cdot \tilde{C}^{-1}w^\wk \\
& = L_0^\st(v^\st) - \bar{c} \cdot v^\st + \frac12 (w^\wk -  \tilde{C} c^\wk) \cdot \tilde{C}^{-1} (w^\wk - \tilde{C}c^\wk) - \frac12 c^\wk \cdot \tilde{C} c^\wk.
\end{align*}
We obtain \eqref{eq:Lsc-split}. 
\end{proof}

The Euler-Lagrange flow of $L^s$ satisfies the following estimates. 
\begin{lemma}\label{lem:EL-est} Fix $\cB^\st,\ \kappa>1$. 
Assume that $q>1$,  $L^s = L_{\cH^s(\cB^\st,\cB^\wk, p, U^\st, \cU^\wk)}$, with  $(\cB^\wk, p, U^\st, \cU^\wk) \in \Omega^{m,d}_{\kappa,q}(\cB^\st)$. Let $\gamma = (\gamma^\st, \gamma^\wk):[0,T] \to \T^d$ satisfy the Euler-Lagrange equation of $L^s$. 
\begin{enumerate}
\item There exists a constant $M_1 = M_1(\cB^\st, Q, \kappa, q)$ such that 
\[
	\|\ddot{\gamma}^\st - A \partial_{\varphi^\st} U^\st(\gamma^\st)\|_{C^0} \le M_1 (\mu(\cB^\wk))^{-(q-1)}. 
\]
\item There exists a constant $M_2 = M_2(\cB^\st, Q, \kappa, q, \|U^\st\|)$ such that 
\[
	\|\ddot{\gamma}^\st\|_{C^0} \le M_2. 
\]
\end{enumerate}
\end{lemma}
\begin{proof}
Observe that the $(\varphi^\st, v^\st)$ component of the Euler-Lagrange vector field of $L^s$ is precisely the vector field $\Pi_{\varphi^\st, v^\st} \tilde{X}^s$ in Theorem~\ref{thm:resc-est}. The Euler-Lagrange equation of $L^\st$ (which is $X^\st$ in Theorem~\ref{thm:resc-est}) is $\ddot{\varphi}^\st = A\partial_{\varphi^\st} U^\st$.  Hence item 1 is a rephrasing of the first conclusion of  in Theorem~\ref{thm:resc-est}. 

Since $\|A \partial_{\varphi^\st} U^\st\| \le \|A\| \|U^\st\|$, and $\|A\|$  depends only on $\cB^\st$ and $Q$, item 2 follows directly from item 1. 
\end{proof}


\section{Weak KAM solutions of dominant Hamiltonians and 
convergence}
\label{sec:con-weak-kam}

In this section, we provide some basic information about 
the weak KAM solution of the dominant system. 

In section~\ref{sec:intro-weak-kam}, we give an overview on the relevant weak KAM theory. Recall that in section~\ref{sec:slow-lag}, 
we derive the relation between the slow Lagrangian and 
the strong Lagrangian. In section~\ref{sec:proj-comp}, we obtain 
a compactness result for the strong component of a minimizing curve.  In section~\ref{sec:intro-app-lip} to \ref{sec:conv-weak-kam}, we prove Theorem~\ref{thm:semi-cont} with some technical statements deferred to section~\ref{sec:tech-est}. 

\subsection{Weak KAM solutions of Tonelli Lagrangian}
\label{sec:intro-weak-kam}

For an extensive exposition of the topic, we refer to \cite{Fat08}. 

\breakheading{Tonelli Lagrangian.} The Lagrangian function $L = L(\varphi, v): \T^d \times \R^d \to \R$ is called Tonelli if it satisfies the following conditions. 
\begin{enumerate}
\item (smoothness) $L$ is $C^r$ with $r \ge 2$.
\item (fiber convexity) $\partial^2_{vv}L$ is strictly positive definite. 
\item (superlinearity) $\lim_{\|v\|\to \infty} |L(x,v)|/\|v\| = \infty$. 
\end{enumerate}
The Lagrangians considered in this paper are Tonelli. 

\breakheading{Minimizers.} An absolutely continuous curve $\gamma: [a,b] \to \T^d$ is called \emph{minimizing}  for the Tonelli Lagrangian $L$ if 
\[
	\int_a^b L(\gamma, \dot{\gamma}) dt = \min_\xi \int_a^b L(\xi, \dot{\xi}) dt,
\]
where the minimization is over all absolutely continuous curves $\xi:[a,b] \to \T^d$ with $b>a$, such that $\xi(a) = \gamma(a)$, $\xi(b) = \gamma(b)$. The functional 
$$
\mathbb{A}(\gamma) = \int_a^b L(\gamma, \dot{\gamma}) dt
$$ 
is called \emph{the action functional}. The curve $\gamma$ is called an \emph{extremal} if it is a critical point of the action functional. A minimizer is extremal, and it satisfies the Euler-Lagrange equation 
\[
	\frac{d}{dt}(\partial_vL(\gamma, \dot{\gamma})) = \partial_\varphi L (\gamma, \dot{\gamma}). 
\]
\

\breakheading{Tonelli Theorem and a priori compactness.} By the Tonelli Theorem (c.f \cite{Fat08}, Corollary 3.3.1),  for any $[a,b] \subset \R$ with $b>a$, $\varphi, \psi \in \T^d$, there always exists a $C^r$ minimizer. Moreover, there exists $D >0$ depending only on a lower bound of $b-a$ such that $\|\dot{\gamma}\| \le D$ (\cite{Fat08} Corollary 4.3.2). This property is called the a priori compactness. 
\

\breakheading{The alpha function and minimal measures.} A measure $\mu$ on $\T^d \times \R^d$ is called a closed measure (see \cite{Sor}, Remark 4.40) if for all $f \in C^1(\T^d)$, 
\[
	\int df(\varphi) \cdot v\, d\mu(\varphi, v) =0. 
\]
This notion is equivalent to the more well known notion of holonomic measure defined by Ma\~ne (\cite{Man97}). 

For $c \in H^1(\T^d, \R) \simeq \R^d$, the alpha function 
\[
	\alpha_L(c) = - \inf_\nu \int (L(\varphi, v) - c \cdot v) d\nu(\varphi, v),
\]
where the minimization is over all closed Borel probability measures. When $L = L_H$ we also use the notation $\alpha_H(c)$.  A measure $\mu$ is called a $c-$\emph{minimizing} if it reaches the infimum 
above. A minimizing measure always exists, and is invariant under the Euler-Lagrange 
flow (c.f \cite{Man97, Ber08}). Hence this definition of the alpha function is equivalent 
to the one given in section~\ref{sec:intro-var}, where the minimization is over invariant probability measures. 

\breakheading{Rotation number and the beta function.}  The rotation number $\rho$ of a closed measure $\mu$ is defined by the relation 
\[
	\int (c \cdot v) d\mu(\varphi, v) = c \cdot \rho, \quad \text{ for all } c\in H^1(\T^d, \R). 
\]
For $h \in H_1(\T^d, \R) \simeq \R^d$, the beta function is
\[
	\beta_L(h) = \inf_{\rho(\nu) =h} \int L(\varphi, v) d\nu(\varphi, v). 
\]
When $L = L_H$ we use the notation $\beta_H(h)$. 
The alpha function and beta function are Legendre duals: 
\[
	\beta_L(h) = \sup_{c \in \R^d}\{c \cdot h - \alpha_L(c)\}. 
\]
\

\breakheading{The Legendre-Fenichel transform.}
Define the Legendre-Fenichel
transform associated to the beta function 
\be \label{LF-transform}
\beal 
\cL\cF_\beta:H_1(\T^d,\R)\to \qquad \qquad \qquad \qquad 
\\
\textup{the collection of nonempty, compact convex subsets of }
H^1(\T^d,\R),
\enal 
\ee 
defined by 
\[
\cL\cF_\beta(h)=\{c\in H^1(\T^n,\R):\ \beta_L(h)+\alpha_L(c)=c\cdot h\}.
\]
\

\breakheading{Domination and calibration.} For $\alpha \in \R$, a function $u:\T^d \to \R$ is dominated by $L + \alpha$ if for all $[a,b]\subset \R$ and piecewise $C^1$ curves $\gamma:[0,T] \to \T^d$, we have 
\[
	u(\gamma(b)) - u(\gamma(a)) \le \int_a^b L(\gamma, \dot{\gamma}) dt + \alpha(b-a). 
\]
A  piecewise $C^1$ curve $\gamma:I \to \R$ defined on an interval $I \subset \R$ is called $(u, L, \alpha)$-calibrated if for any $[a,b]\subset I$, 
\[
	u(\gamma(b)) - u(\gamma(a)) = \int_a^b L(\gamma, \dot{\gamma}) dt + \alpha(b-a). 
\]

\breakheading{Weak KAM solutions.} A function $u: \T^d \to \R$ is called a weak KAM solution of $L$ if there exists $\alpha \in \R$ such that the following hold. 
\begin{enumerate}
\item $u$ is dominated by $L + \alpha$. 
\item For all $\varphi \in \T^d$, there exists a $(u, L, \alpha)$-calibrated curve $\gamma: (-\infty, 0] \to \T^d$ with $\gamma(0) = \varphi$. 
\end{enumerate}
This definition of the weak KAM solution is equivalent to the one given in section~\ref{sec:intro-var} (see \cite{Fat08}, Proposition 4.4.8), and the constant $\alpha = \alpha_L(0)$, where $\alpha_L$ is the alpha function.

\breakheading{Peierls' barrier.} For $T>0$, we define the function $h^T_L: \T^d \times \T^d \to \R$ by 
\[
	h^T_L(\varphi, \psi) = \min_{\gamma(0)= \varphi, \gamma(T) = \psi}\int_0^T (L(\gamma , \dot{\gamma}) + \alpha_L) dt. 
\]
Peierls' barrier is $h_L(\varphi, \psi) = 
\lim_{T \to \infty} h_L^T(\varphi, \psi)$. The limit exists, and the function $h_L$ is Lipschitz in both variables. Denote $h_{L,c} = h_{L - c\cdot v}$.

\breakheading{Mather, Aubry and Ma\~ne sets}. 
These sets are defined by Mather (see \cite{Mather93}). Here 
we only introduce the projected version. Define the projected Aubry and  Ma\~ne
sets as   
\[
	\cA_L(c) = \{x\in \T^d: \quad h_{L,c}(x,x) =0\}, 
\]
\[
 \cN_L(c) = \left\{ y \in \T^d: \quad \min_{x,z\in \cA_L(c)}
	\left( h_{L,c}(x,y) + h_{L,c}(y,z) - h_{L,c}(x,z) \right) =0
	\right \} . 
\]
The Mather set is $\tilde{\cM}_L(c) = \overline{\bigcup_{\mu}\supp(\mu)}$ is the closure of the support of all $c-$minimal measures. Its projection $\pi \tilde{\cM}(c)={\cM}(c)$ onto $\T^d$ is called the projected Mather set. Then $$
\cM_L(c) \subset \cA_L(c) \subset \cN_L(c). 
$$
When $L = L_H$ we also use the subscript $H$ to identify these sets. 

%

\ 
\breakheading{Static classes.} For any $\varphi, \psi \in \cA_L(c)$, Mather defined the following equivalence relation: 
$$
\varphi \sim \psi\  \textup{ if }\ h_{L,c}(\varphi, \psi) + h_{L,c}(\psi, \varphi) =0. 
$$
The equivalence classes defined by this equivalence condition are called the static classes. The static classes are linked to the family of weak KAM solutions, in particular, if there is only one static class, then the weak KAM solution is unique up to a constant.

In this section, we provide a few useful estimates in 
weak KAM theory, and prove Theorem~\ref{thm:semi-cont}.
In section~\ref{sec:proj-comp}, we prove a projected version of the a priori compactness property. We then introduce an approximate version of Lipschitz property and use it to prove Theorem~\ref{thm:semi-cont}.

\subsection{Minimizers of strong and slow Lagrangians,
their a priori compactness}
\label{sec:proj-comp}

We prove a version of the a priori compactness theorem for the strong component.

\begin{proposition}\label{prop:proj-comp} Fix $\cB^\st, \kappa >1$. 
For any $R>0$, there exists $M = M(\cB^\st, Q, R, \kappa)$ 
such that the following hold. For any
\[
	(\cB^\wk, p, U^\st, \cU^\wk) \in \Omega^{m,d}_{\kappa, q} (\cB^\st)
\cap \{ \| U^\st\|_{C^2} \le R\}
\ \textup{ 
and }\ 
	L^s = L_{\cH^s(\cB^\st,\cB^\wk, p, U^\st, \cU^\wk)},  
\]
let $T \ge \frac12$,  $c = (c^\st, c^\wk)\in\R^m \times \R^{d-m}$ and  $\gamma = (\gamma^\st, \gamma^\wk): [0, T]  \to \T^d$ be a minimizer of $L^s - c\cdot v$. Then 
for $\bar c = c^\st + A^{-1}Bc^\wk$, we have 
\[
	\| \dot{\gamma}^\st - A \bar{c} \| \le M. 
\]
\end{proposition}

We first state a lemma on the strong component of the action and 
relate minimizers of the slow system with those of the strong one. 
\begin{lemma}
\label{lem:strong-act-bound} In the notations of Proposition 
\ref{prop:proj-comp} for $T \ge \frac12$  and $c\in\R^d$, let 
$\gamma = (\gamma^\st, \gamma^\wk) : [0,T] \to \T^d$ be 
a minimizer for the lagrangian $L^s - c\cdot v$. Then 
\[
	\int_0^T (L^\st - \bar{c} \cdot v^\st)(\gamma^\st, \dot{\gamma}^\st)dt \le \min_\zeta\int_0^T (L^\st - \bar{c} \cdot v^\st)(\zeta, \dot{\zeta})dt  + 2T\| U^\wk\|_{C^0},
\]
where the minimization is over all absolutely continuous $\zeta:[0,T] \to \T^m$ with $\zeta(0) = \gamma^\st(0)$, $\zeta(T) = \gamma^\st(T)$. 
\end{lemma}
\begin{proof}
Let $\gamma_0^\st:[0,T] \to \T^m$ be such that 
\[
	\int_0^T (L^\st - \bar{c} \cdot v^\st)(\gamma_0^\st, \dot{\gamma}_0^\st)dt  = \min_\zeta\int_0^T (L^\st - \bar{c} \cdot v^\st)(\zeta, \dot{\zeta})dt 
\]
with $\zeta(0) = \gamma^\st(0)$, $\zeta(T) = \gamma^\st(T)$.  Define  $\gamma_0 = (\gamma_0^\st, \gamma_0^\wk): [0,T] \to \T^d$, by 
\[
	\gamma_0^\wk(t) = \gamma^\wk(t) - A^{-1} B \gamma^\st(t) + A^{-1}B \gamma_0^\st(t)  .
\]
Note that 
\begin{equation}
\label{eq:vel-cancel}
	\gamma_0^\wk(0) = \gamma^\wk(0), \quad \gamma_0^\wk(T) = \gamma^\wk(T), \quad \dot{\gamma}^\wk_0 - A^{-1}B \dot{\gamma}_0^\st = \dot{\gamma}^\wk - A^{-1}B \dot{\gamma}^\st. 
\end{equation}
Using \eqref{eq:Lsc-split} and \eqref{eq:lst-square}, we have
\begin{multline}
\label{eq:ls-long}
L^s - c \cdot v + \frac12 c^\wk \cdot \tilde{C}^{-1} c^\wk  =  
 L^\st - \bar{c} \cdot v^\st  \\
 + \frac12(v^\wk - B^T A^{-1} v^\st - \tilde{C}^\wk) \cdot \tilde{C} (v^\wk - B^T A^{-1} v^\st - \tilde{C}^\wk) + U^\wk
\end{multline}
Since $\gamma$ is a minimizer for $L^s - c \cdot v$, 
\[
	\int_0^T (L^s - c \cdot v)(\gamma, \dot{\gamma}) dt \le \int_0^T (L^s - c \cdot v)(\gamma_0, \dot{\gamma}_0) dt. 
\]
By \eqref{eq:ls-long}, we have
\begin{align*}
& \int_0^T (L^\st - \bar{c} \cdot v^\st )(\gamma^\st, \dot{\gamma}^\st)dt + \int_0^T U^\wk(\gamma(t)) dt \\
& + \frac12(\dot{\gamma}^\wk - B^T A^{-1} \dot{\gamma}^\st - \tilde{C}^\wk) \cdot \tilde{C} (\dot{\gamma}^\wk - B^T A^{-1} \dot{\gamma}^\st - \tilde{C}^\wk)  \\
\le & \int_0^T (L^\st - \bar{c} \cdot v^\st )(\gamma^\st_0, \dot{\gamma}^\st_0)dt + \int_0^T U^\wk(\gamma_0(t)) dt \\
& + \frac12(\dot{\gamma}^\wk_0 - B^T A^{-1} \dot{\gamma}^\st_0 - \tilde{C}^\wk) \cdot \tilde{C} (\dot{\gamma}^\wk_0 - B^T A^{-1} \dot{\gamma}^\st_0 - \tilde{C}^\wk). 
\end{align*}
By \eqref{eq:vel-cancel}, the second and fourth line of the above inequality cancels, therefore
\[
	\int_0^T (L^\st - \bar{c} \cdot v^\st)(\gamma^\st, \dot{\gamma}^\st)dt \le \int_0^T (L^\st - \bar{c} \cdot v^\st)(\gamma_0^\st, \dot{\gamma}_0^\st)dt  + 2T\| U^\wk\|_{C^0}. 
\]
\end{proof}

\begin{proof}[Proof of Proposition~\ref{prop:proj-comp}]		
First, observe that any segments of a minimizer is still a minimizer. By dividing the interval $[0,T]$ into subintervals, it suffice to prove our proposition for $T \in [\frac12, 1)$. 

We first produce an upper bound for 
\[
	 \min_\zeta\int_0^T (L^\st - \bar{c} \cdot v^\st + \frac12 \bar{c} \cdot A \bar{c})(\zeta, \dot{\zeta})dt. 
\]
By completing the squares as in Lemma~\ref{lem:ls-split}, we have 
\begin{equation}
\label{eq:lst-square}
	L^\st - \bar{c} \cdot v^\st + \frac12 \bar{c} \cdot A \bar{c} = \frac12 (v^\st - A\bar{c})\cdot A^{-1}(v^\st - A\bar{c}) + U^\st(\varphi^\st). 
\end{equation}
We then take 
\[
\zeta_0(t) = \gamma^\st(0) + t A\bar{c} +  \frac{t}{T}y
\]
where $y \in [0,1)^d$ is such that $\zeta_0(0) + T A\bar{c} +  y = \gamma^\st(T) \mod \Z^m$. We then have $\dot\zeta_0 - A \bar{c} = \frac{1}{T}y$, so 
\[
	\int_0^T(L^\st - \bar{c} \cdot v^\st + \frac12 \bar{c} \cdot A \bar{c}) (\zeta_0, \dot{\zeta}_0)dt \le \frac{1}{2T} \|A^{-1}\| \|y\|^2 + T\|U^\st\|_{C^0} \le d \|A^{-1}\| + \|U^\st\|_{C^0}
\]
using $T \in [0,1)$ and $\|y\|^2 \le d$.

Using Lemma~\ref{lem:strong-act-bound}, and   adding $\frac12 \bar{c}\cdot A\bar{c}$ to the Lagrangian to both sides, we obtain 
\begin{multline*}
	\int_0^T (L^\st - \bar{c} \cdot v^\st + \frac12 \bar{c} \cdot A \bar{c})(\gamma^\st, \dot{\gamma}^\st)dt \le 2T \|U^\wk\| + d\|A^{-1}\| + \|U^\st\|_{C^0} \\
	 \le d\|A^{-1}\| + \|U^\st\|_{C^0} + 2\|U^\wk\|_{C^0}
\end{multline*}
since $T \in [\frac12,1)$.

We now use the above formula get an $L^2$ estimate on $(\dot{\gamma}^\st - A\bar{c})$ and use the Poincar\'{e} estimate 
to conclude. Using the above formula and \eqref{eq:lst-square}, 
we have 
\[
	\int_0^T (\dot{\gamma}^\st - A\bar{c})\cdot A^{-1}(\dot{\gamma}^\st - A\bar{c}) dt \le d \|A^{-1}\| + 2\|U^\st\|_{C^0} + 2\|U^\wk\|_{C^0}. 
\]
Using the fact that $A^{-1}$ is strictly positive definite, we get 
$$
\|\dot{\gamma}^\st - A\bar{c}\|_{L^2} \le \|A\|(d \|A^{-1}\| + 2\|U^\st\|_{C^0} + 2\|U^\wk\|_{C^0})=:M_1.
$$ 
Then
\begin{equation}
\label{eq:average}
	\left\|  \frac{1}{T} \int_0^T (\dot{\gamma}^\st - A\bar{c})\,dt \right\|^2  \le \frac{1}{T^2} \int_0^T \|\dot{\gamma}^\st - A\bar{c}\|^2\, dt\le 4M_1. 
\end{equation}
Moreover, from Lemma~\ref{lem:EL-est}, 
\[
		\|\ddot{\gamma}^\st\| \le M_2(\cB^\st, Q, \kappa, q, R).
\]
The Poincar\'e estimate gives, for some uniform constant $D>0$, 
\[
	\left\|  (\dot{\gamma}^\st - A\bar{c}) - \frac{1}{T} \int_0^T (\dot{\gamma}^\st - A\bar{c})dt \right\|_{L^\infty} \le \| \ddot{\gamma}^\st\|_{L^\infty} \le D M_2. 
\]
Combine with \eqref{eq:average} and we conclude the proof. 
\end{proof}

\subsection{Approximate Lipschitz property of weak KAM
solutions}\label{sec:intro-app-lip}

The weak KAM solutions of the slow Hamiltonian is Lipshitz, however, it is not clear if the Lipschitz constant is bounded as $\lM(\cB^\wk) \to \infty$. To get uniform estimates, we consider the following weaker notion. 

\begin{definition}
For $D, \delta >0$, a function $u: \R^d \to \R$ is called  $(D, \delta)$ approximately Lipschitz if 
\[
	|u(x) - u(y)| \le D \|x - y\| + \delta, \quad x, y \in \R^d. 
\]
For $u: \T^d \to \R$, the approximate Lipschitz property is defined by its lift to $\R^d$. 
\end{definition}

In Proposition~\ref{prop:appr-lip-weak} and \ref{prop:appr-lip-strong} we state the approximate Lipschitz property of a weak KAM solution in weak and strong angles. 
\begin{proposition}
\label{prop:appr-lip-weak} Fix $\cB^\st, \ \kappa>1$. 
Assume that $q> 2(d-m)$. For $R>0$, there exists a constant $M = M(\cB^\st, Q, \kappa, q, R) >0$, such that for all 
\[
	(\cB^\wk, p, U^\st, \cU^\wk) \in \Omega^{m,d}_{\kappa, q} (\cB^\st)
\cap \{ \|U^\st\| \le R\},
\]
and 
\[
	\delta(\cB^\wk) = M \lM(\cB^\wk)^{-(\frac{q}{2}-d+m)},
\]
let  $u= u(\varphi^\st,\varphi^\wk):\T^m \times \T^{d-m} \to \R$ be a weak KAM solution of
\[
	L_{\cH^s}(\cB^\st,\cB^\wk, p, U^\st, \cU^\wk) - c \cdot v.
\]
Then for all $\varphi^\st \in \T^m$, the function $u(\varphi^\st, \cdot)$ is $(\delta, \delta)$ approximately Lipschitz. 
\end{proposition}

\begin{proposition}
\label{prop:appr-lip-strong}
There exists a constant $M' = M'(\cB^\st, Q, \kappa, q, R) >0$, let $\delta'(\cB^\wk) = M' (\lM(\cB^\wk))^{-(\frac{q}{2}-d+m)}$, and $u$ 
be the weak KAM solution described in Proposition~\ref{prop:appr-lip-weak}. 
Then for all $\varphi^\wk \in \T^{d-m}$, the function 
$u(\cdot, \varphi^\wk)$ is $(M',\delta')$ approximately Lipschitz. 
\end{proposition}

The proof of these statements are deferred to section~\ref{sec:tech-est}. 


\subsection{The alpha function and rotation number estimate}
\label{sec:alpha-rot}

In this section we provide a few useful estimates in weak KAM
theory and prove Theorem \ref{thm:semi-cont} using 
Propositions~\ref{prop:appr-lip-weak} and ~\ref{prop:appr-lip-strong}.
Recall that the notations 
$c=(c^\st,c^\wk),\ \bar c=c^\st+A^{-1}Bc^\wk$.  

\begin{proposition}\label{prop:Ls-alpha}
We have 
\[
	\left|\alpha_{H^s}(c) - \alpha_{H^\st}(\bar{c}) + \frac12 (\tilde{C} c^\wk) \cdot c^\wk  \right|\le  \|U^\wk\|_{C^0}, 
\]
\end{proposition}
\begin{proof}

Let $\mu$ be a minimal measure for $L^s - c\cdot v$. Let $\pi$ denote the natural projection from $(\varphi^\st, \varphi^\wk, v^\st, v^\wk)$ to $(\varphi^\st, v^\st)$. By Lemma \ref{lem:ls-split} we have
\begin{equation}
\label{eq:alpha-Ls}
\begin{aligned}
& - \alpha_{H^s}(c)  = \int (L^s - c \cdot v) d\mu  \\
&= \int (L^\st - \bar{c}  \cdot v^\st) d\mu\circ \pi - \frac12 c^\wk \cdot \tilde{C} c^\wk  \\
& \qquad  + \int \left( \frac12 (w^\wk - \tilde{C} c^\wk) \cdot \tilde{C}^{-1} (w^\wk - \tilde{C} c^\wk)  + U^\wk \right) d\mu \\
& \ge - \alpha_{H^\st}(\bar{c}) - \|U^\wk\|_{C^0} - \frac12c^\wk \cdot \tilde{C} c^\wk . 
\end{aligned}
\end{equation}
On the other hand, let $\mu^\st$ be an ergodic minimal measure for $L^\st - \bar{c} \cdot v^\st$.  For an $L^\st-$Euler-Lagrange orbit $\varphi^\st(t)$  in the support of $\mu^\st$,  and any $\varphi^\wk_0 \in \T^{d-m}$, define
\begin{equation}
\label{eq:lift-gamma}
	\varphi^\wk(t) =  \varphi^\wk_0 + B^T A^{-1} \varphi^\st(t) + \tilde{C}c^\wk t, \quad t \in \R 
\end{equation}
and write $\gamma = (\gamma^\st, \gamma^\wk)$. We take 
a weak-$*$ limit point $\mu^s$ of the probability measures 
$\frac{1}{T}(\gamma, \dot{\gamma})|_{[0,T]}$ as $T\to +\infty$. Then $\mu^s$ 
is a closed measure (see section~\ref{sec:intro-weak-kam}).

Since on the support of $\mu^s$,  $v^\wk - B^T A^{-1} v^\st - \tilde{C} c^\wk =0$, we have
\begin{align*}
& - \alpha_{H^s}(c) \le  \int (L^s - c \cdot v) d\mu^s \\
& =  \int (L^\st - \bar{c}\cdot v^\st) d\mu^\st + \int U^\wk d\mu - \frac12 \tilde{C}c^\wk \cdot c^\wk \\
& \le - \alpha_{H^\st}(\bar{c}) + \|U\|_{C^0} - \frac12\tilde{C}c^\wk \cdot c^\wk. 
\end{align*}
\end{proof}

The following proposition establishes relations between rotation numbers of minimal measures of the slow and strong systems. 
\begin{proposition}\label{prop:rotation}
Let $\mu^s$ be an ergodic minimal measure of $L^s - c \cdot v$, and let $(\rho^\st, \rho^\wk)$ denote its rotation number. Then 
\[
	0 \le \frac12(\tilde{C} (\rho^\wk - B^T A^{-1}\rho^\st-\tilde{C}c^\wk)) \cdot (\rho^\wk - B^T A^{-1}\rho^\st- \tilde{C}c^\wk) \le \|U^\wk\|_{C^0}
\]
and 
\[
	0 \le \alpha_{H^\st}(\bar{c}) + \beta_{H^\st}(\rho^\st) - \bar{c}\cdot \rho^\st \le \|U^\wk\|_{C^0}. 
\]
\end{proposition}

\begin{proof}
Using \eqref{eq:alpha-Ls} and the conclusion of Proposition~\ref{prop:Ls-alpha}, we have 
\begin{multline}\label{eq:bound-rot-int}
	 \|U^\wk\|_{C^0} 	  \ge  \int (L^\st - \bar{c}  \cdot v^\st + \alpha_{H^\st}(\bar{c})) d\mu^s\circ \pi \\
	  + \int \frac12(\tilde{C}^{-1} (w - \tilde{C} c^\wk)) \cdot (w - \tilde{C} c^\wk) d \mu^s \circ \pi. 
\end{multline}
Note the first of the two integrals is non-negative by definition, we obtain 
\[
	0 \le \int \frac12 (w^\wk - \tilde{C} c^\wk) \cdot \tilde{C}^{-1} (w^\wk - \tilde{C} c^\wk)  d\mu^s \le \|U^\wk\|_{C^0}. 
\]
Denote $\bar{w}^\wk : = \int w^\wk d\mu^s = \rho^\wk - B^T A^{-1}\rho^\st$, and rewrite the left hand side of the last formula as 
\begin{multline*}
	\frac12(\tilde{C}^{-1} (\bar{w}^\wk - \tilde{C} c^\wk)) \cdot (\bar{w}^\wk - \tilde{C} c^\wk) +  \int \tilde{C}^{-1}(\bar{w}^\wk - \tilde{C}c^\wk) \cdot (w^\wk - \bar{w})d\mu^s \\
	 + \frac12\int (\tilde{C}^{-1}(w^\wk - \bar{w}^\wk)) \cdot (w^\wk - \bar{w}^\wk) d\mu^s. 
\end{multline*}
Note that the second term vanishes and the third term is non-negative. Therefore 
\[
	\frac12(\tilde{C}^{-1} (\bar{w}^\wk - \tilde{C} c^\wk)) \cdot (\bar{w}^\wk - \tilde{C} c^\wk) \le \|U^\wk\|_{C^0}
\]
which is the first conclusion. 

For the second conclusion, using \eqref{eq:bound-rot-int}, we get 
\[
	\|U^\wk\|_{C^0} \ge \int (L^\st - \bar{c}  \cdot v^\st + \alpha_{H^\st}(\bar{c})) d\mu\circ \pi = \int L^\st d\mu \circ \pi - \bar{c}\cdot \rho^\st + \alpha_{H^\st}(\bar{c}).  
\]
Using $\int L^\st d\mu\circ \pi \ge \beta_{H^\st}(\rho^\st)$ we get the upper bound of the second conclusion. The lower bound holds by definition. 
\end{proof}

\subsection{Convergence of weak KAM solutions}
\label{sec:conv-weak-kam}
We now prove  Theorem~\ref{thm:semi-cont}. Fix $\cB^\st$ and 
$\kappa>1$. 

Let $(\cB^\wk_i, p_i, U^\st_i, \cU^\wk_i) \in 
\Omega^{m,d}_{\kappa,q}(\cB^\st)$ and $c_i = (c_i^\st, c_i^\wk)$ be 
a sequence satisfying the assumption of the theorem, namely $\lM(\cB^\wk_i) \to \infty$, $p_i \to p_0$, $U_i^\st \to U_0^\st$ in $C^2$, and $c^\st + A_i^{-1} B_i c_i^\wk \to \bar{c}$. 

\myheading{Item 1}. Let $u_i$ be the weak KAM solution to $L_i^s = L_{\cH^s(\cB^\st,\cB^\wk_i, p_i, U^\st_i, \cU^\wk_i)} - c_i \cdot v$. 
We first show the sequence $\{u_i\}$ is equi-continuous. 

Let $M^*$ be a constant larger than the constants in both  Proposition~\ref{prop:appr-lip-weak} and \ref{prop:appr-lip-strong}. Using both propositions, for any $\varphi = (\varphi^\st, \varphi^\wk)$, $\psi = (\psi^\st, \psi^\wk)$, 
\[
	|u_i(\varphi^\st, \varphi^\wk) - u_i(\psi^\st, \psi^\wk)| \le M^*\|\varphi^\st - \psi^\st\| + \delta_i \|\varphi^\wk - \psi^\wk\| + 2\delta_i,
\]
where $\delta_i = M^*(\lM(\cB^\wk_i))^{-\frac{q}2-d+m}$. 

Since $\delta_i \to 0$ as $i \to \infty$, for any $0 <\varepsilon <1$ there exists $M>0$ such that for all $i >M$, $3\delta_i < \frac{\varepsilon}2$. It follows that if $\|\varphi - \psi\| < \frac{\varepsilon}{2D}<1$, then 
\[
	|u_i(\varphi^\st, \varphi^\wk) - u_i(\psi^\st, \psi^\wk)| < \varepsilon. 
\]
Since $\{u_i\}_{i \le M}$ is a finite family, it is equi-continuous. In particular, there exist $\sigma >0$ such that
\[
	|u_i(\varphi) - u(\psi)| < \varepsilon,\quad \text{ if } 1 \le i \le M, \, \|\varphi - \psi\|< \sigma. 
\]
This proves equi-continuity. Moreover, since $u_i$ are all periodic, $u_i - u_i(0)$ are equi-bounded, therefore Ascoli's theorem applies and the sequence is pre-compact in uniform norm. 

\myheading{Item 2.}  Let $u$ be any accumulation point of $u_i - u_i(0)$, without loss of generality, we assume $u_i - u_i(0)$ converges to $u$ uniformly. Proposition~\ref{prop:appr-lip-weak} implies that 
\[
	\lim_{i \to \infty}\sup_{\varphi^\st}(\max u_i(\varphi^\st, \cdot) - \min_i u_i(\varphi^\st, \cdot)) \le 2 \lim_{i \to \infty} \delta_i =0,
\]
therefore $u$ is independent of $\varphi^\wk$. 

\myheading{Item 3.}  From item 2, there exists  $u^\st(\varphi^\st) = \lim_{i \to \infty} u_i(\varphi^\st, \varphi^\wk)$. We show $u^\st$ is a weak KAM solution of $L^\st_0 - \bar{c} \cdot v^\st = L_{\cH^\st(p_0, U_0^\st)} - \bar{c} \cdot v^\st$. Denote $L_i^\st = L_{\cH^\st(p_i, U^\st_i)}$, we have $L_i^\st \to L_0^\st$ in $C^2$. 

We first show that $u^\st$ is dominated by $L_0^\st - \bar{c} \cdot v^\st$.  Let $\xi^\st :[0,T] \to \T^m$ be an extremal curve of $L^\st_0$. In the same way as  \eqref{eq:lift-gamma} in the proof of  Proposition~\ref{prop:Ls-alpha}, we define $\xi_i = (\xi_i^\st, \xi_i^\wk): [a, b] \to \T^m$ such that  $\xi_i^\st(a) = \xi^\st(a)$, $\xi_i^\st(b) = \xi^\st(b)$ and $\dot{\xi}_i^\st - B^T_i A_i^{-1} \dot{\xi} - \tilde{C} c^\wk_i =0$. Since for $\bar{c}_i = c_i^\st + A^{-1}_i B_i c^\wk_i$,  $u_i$ are dominated by $L^s_i - c_i \cdot v + \alpha_{H^s_i}(c_i)$, we have 
\begin{align*}
&u_i(\xi_i(b)) - u_i(\xi_i(a)) \le \int_a^b (L^s_i - c_i \cdot v^s+ \alpha_{H^s_i}(c_i))(\xi_i, \dot{\xi}_i) dt \\
& = \int_a^b (L^\st_i - \bar{c}_i \cdot v^\st)(\xi^\st_i, \dot{\xi}_i^\st) dt + \int_a^b (U^\wk_i(\xi_i) + \alpha_{H^s_i}(c_i) - \frac12 \tilde{C}_i c^\wk_i \cdot c^\wk_i) dt,
\end{align*}
where the equality is due to $\dot{\xi}_i^\st - B^T_i A_i^{-1} \dot{\xi}_i - \tilde{C} c^\wk_i =0$. Using the fact that $\|U^\wk_i\|_{C^0} \to 0$, $L_i^\st \to L_0^\st$,  and from Proposition~\ref{prop:Ls-alpha}, $\alpha_{H^s_i} - \frac12 \tilde{C}_i c^\wk_i \cdot c^\wk_i \to \alpha_{H^\st}(\bar{c})$ as $i \to \infty$, we get 
\begin{equation}
\label{eq:Lst-dom}
	u^\st(\xi^\st(b)) - u^\st(\xi^\st(a)) \le \int_a^b (L^\st_0 - \bar{c}\cdot v^\st + \alpha_{H^\st}(\bar{c})) dt.  
\end{equation}
Therefore $u^\st$ is dominated by $L^\st_0 - \bar{c}\cdot v^\st$.

Secondly, we show that for any $\varphi^\st \in \T^m$, there exists a $(u^\st, L_0^\st, \bar{c})$-calibrated curve $\gamma^\st: (-\infty, 0]\to \T^m$ with $\gamma^\st(0) = \varphi^\st$. 
Because $u_i$ are weak KAM solutions, for each $i$ there exists a $(u_i, L_i^s - c_i \cdot v, \alpha(H_i^s))$-calibrated curve $\gamma_i = (\gamma_i^\st, \gamma_i^\wk):(-\infty, 0] \to \T^d$.  By Proposition~\ref{prop:proj-comp},  all $\gamma_i^\st$ are uniformly Lipschitz, so there exists a subsequence that converges in $C^1_{loc}((-\infty,0], \T^d)$. Assume without loss of generality that $\gamma_i^\st \to \gamma^\st$, since $\gamma_i = (\gamma_i^\st, \gamma_i^\wk)$ is extremal for $L^s_i$, we have 
\[
	\ddot{\gamma}_i^\st  = \frac{d}{dt}(A_i I^\st + B_i I^\wk) = A_i \partial_{\varphi^\st} U_i^\st + B_i \partial_{\varphi^\st} U^\wk_i. 
\]
By our assumption, as $i\to \infty$, $A_i \to A := \partial^2_{v^\st v^\st}L_0^\st$, and by Lemma~\ref{lem:EL-est} $\|B_i\|\|U^\wk_i\|_{C^2} \to 0$, we have
\begin{equation}
\label{eq:str-extremal}
	\ddot{\gamma}^\st = A \partial_{\varphi^\st}U^\st(\gamma^\st),
\end{equation}
which is the Euler-Lagrange equation for $L^\st_0$.

On the other hand, since $\gamma_i$ are $(u_i, L_i^s - c_i \cdot v, \alpha(H_i^s))$ calibrated,  for any $[a,b] \subset (-\infty, 0]$, 
\begin{align*}
&u_i(\gamma_i(b)) - u_i(\gamma_i(a)) = \int_a^b (L^s_i - c_i \cdot v^s+ \alpha_{H^s_i}(c_i))(\gamma_i, \dot{\gamma}_i) dt \\
& \ge  \int_a^b (L^\st - \bar{c}_i \cdot v^\st)(\gamma_i^\st, \dot{\gamma}_i^\st) dt + \int_a^b (U^\wk_i + \alpha_{H^s_i}(c_i) - \frac12 \tilde{C}_i c^\wk_i \cdot c^\wk_i)(\gamma_i, \dot{\gamma}_i) dt 
\end{align*}
Take limit again to get 
\[
	u^\st(\gamma^\st(b)) - u^\st(\gamma^\st(a)) \ge \int_a^b (L^\st_0 - \bar{c}\cdot v^\st + \alpha_{H^\st}(\bar{c})) (\gamma^\st, \dot{\gamma}^\st) dt.
\]
Because $\gamma^\st$ is an $L^\st$ extremal curve (see \eqref{eq:str-extremal}), \eqref{eq:Lst-dom} hold for $\gamma^\st$.  Combining with last displayed formula, \eqref{eq:Lst-dom} becomes an equality. Then $\gamma^\st$ is a calibrated curve for $L^\st - \bar{c}\cdot v^\st + \alpha_{H^\st}(\bar{c})$, and $u^\st$ is a weak KAM solution.

\section{The Ma\~ne and the Aubry sets and the barrier function}
\label{sec:cor-aub}

We prove the following result.
\begin{proposition} \label{prop:aubry-mane} Fix $\cB^\st$ and 
$\kappa>1$. Assume that  $(\cB^\wk_i, p_i, U^\st_i, \cU^\wk_i)$ 
satisfies the assumptions of Theorem~\ref{thm:semi-cont}.
Denote $H_i^s = \cH_i^s(\cB^\st,\cB^\wk_i, p_i, U^\st_i, \cU^\wk_i)$, $H^\st_0 = \cH^\st(p_0, U_0^\st)$, $L^s_i = L_{H^s_i}$  and 
$L^\st_0 = L_{H^\st_0}$. 
\begin{enumerate}
\item  Any limit point of $\varphi_i \in \cN_{H^s_i}(c_i)$ is contained in $\cN_{H^\st_0}(\bar{c})\times \T^{d-m}$. 
\item If  $\cA_{H^\st_0}(\bar{c})$ contains only finitely many static classes, then any limit point of $\varphi_i \in \cA_{H^s_i}(c_i)$ is contained in $\cA_{H^\st_0}(\bar{c})\times \T^{d-m}$.
\item Assume that $\cA_{H^\st}(\bar{c})$ contains only one static class. Let  $\varphi_i = (\varphi_i^\st, \varphi_i^\wk) \in \cA_{H_i^s}(c_i)$ be such that $\varphi_i^\st \to \varphi^\st \in \cA_{H^\st_0}(\bar{c})$. Then for any $\psi = (\psi^\st, \psi^\wk)\in \T^d$, 
\[
	\lim_{i \to \infty}h_{L^s_i, c_i}(\varphi, \psi) = h_{L^\st_0, \bar{c}}(\varphi^\st, \psi^\st). 
\]
\item Let $(\rho^\st_i, \rho^\wk_i)$ be the rotation number of any $c_i-$minimal measure of $L^s_i$. Then we have 
\[
	 \lim_{i \to \infty}(\rho^\wk_i - B^T_iA_i^{-1} \rho^\st_i - \tilde{C}_ic^\wk_i ) =0,
\]
and any accumulation point $\rho$ of $\rho^\st_i$ is contained in the set $\partial \alpha_{H^\st}(\bar{c})$.
\end{enumerate}
\end{proposition} 

The proof of item 2 requires additional discussion and is presented in Section~\ref{sec:semi-aubry}. In Section~\ref{sec:mane} we prove item 1, 3 and 4. 

\subsection{The Ma\~ne set and barrier function}
\label{sec:mane}

We first state an alternate definition of the Aubry and Ma\~ne sets due to Fathi (see also \cite{Ber08}). Let $u$ be a weak KAM solution for the Lagrangian $L$. We define $\bcG(L,u)$ to be the set of points $(\varphi, v)\in \T^d \times \R^d$ such that there exists a $(u, L, \alpha_L)$-calibrated curve $\gamma: (-\infty, 0] \to \T^d$, such that $(\varphi, v) = (\gamma(0), \dot{\gamma}(0))$. Let $\phi_t$ denote the Euler-Lagrange flow of $L$, then 
\begin{equation}
\label{eq:aubry-mane}
	\tilde{\cI}(L, u) = \bigcap_{t \le 0} \phi_t(\overline{\cG}(L,u)), \quad \tilde{\cA}_L = \bigcap_{u} \tilde{\cI}(L,u), \quad \tilde{\cN}_L = \bigcup_u \tilde{\cI}(L,u),
\end{equation}
where the union and intersection are over all weak KAM solutions of $L$. The Aubry set and Ma\~ne set of $c\in H^1(\T^d, \R)$ is defined as 
\[
	\tilde{\cA}_L(c) = \tilde{\cA}_{L - c\cdot v}, \quad \tilde{\cN}_L(c) = \tilde{\cN}_{L - c\cdot v}.
\]
The projected Aubry and Ma\~ne sets are the projection of these sets to $\T^d$.

We now turn to the setting of Proposition~\ref{prop:aubry-mane}. Let $L_i^s$, $L_0^\st$, $c_i$, $\bar{c}$ be as in the assumption. The strategy of the proof is similar to the one in \cite{Ber10}. 

\begin{lemma} 
\label{lem:semi-cont-I}
Let $u_i$ be a weak KAM solution of $L^s_i - c_i \cdot v$. Assume that $\tilde{\varphi}_i = (\varphi_i, v_i) \in \tilde{\cI}(L_i^s - c_i \cdot v, u_i)$ satisfies $\tilde{\varphi}_i \to \tilde{\varphi} = (\varphi, v) = (\varphi^\st, \varphi^\wk, v^\st, v^\wk)$, and $u_i(\varphi^\st, \varphi^\wk) \to u^\st(\varphi^\st)$. Then 
\[
	(\varphi^\st, v^\st) \in \tilde{\cI}(L^\st - \bar{c}\cdot v^\st, u^\st). 
\]
\end{lemma}
\begin{proof}
We first show that $(\varphi_i, v_i) \in \bcG(L_i^s - c_i \cdot v, u_i)$ implies $(\varphi^\st, v^\st )\in \bcG(L^\st - \bar{c} \cdot v^\st, u^\st)$. Indeed, there exists $\gamma_i: (-\infty, 0] \to \T^d$, each $(u_i, L^s_i - c_i \cdot v,  \alpha_{L_i^s}(c_i))$-calibrated, with $(\gamma_i, \dot{\gamma}_i)(0) = (\varphi, v)$. We follow the same line as proof of item 3 in Theorem~\ref{thm:semi-cont} (section~\ref{sec:con-weak-kam}), then by restricting to a subsequence, $\gamma_i^\st$ converges in $C^1_{loc}((-\infty,0], \T^d)$ to a $(u^\st, L^\st_0 -\bar{c}\cdot v^\st, \alpha_{H^\st_0}(\bar{c}))$-calibrated curve $\gamma^\st$. In particular $(\gamma^\st_i, \dot{\gamma}_i^\st) \to (\gamma^\st, \dot{\gamma}^\st)$, which implies $(\varphi^\st, v^\st) \in \bcG(L^\st - \bar{c} \cdot v^\st, u^\st)$. 

Let $\phi^i_t$ denote the Euler-Lagrange flow of $L_i^s$, and $\phi^\st_t$ the flow for $L^\st$. Let $\pi$ denote the projection to the strong components $(\varphi^\st, v^\st)$, then from Lemma~\ref{lem:EL-est} $\pi \phi^i_t \to  \phi^\st_t$ uniformly. As a result for a fixed $T>0$ and $(\varphi_i , v_i) \in \tilde{\cI}(L_i^\st - c_i\cdot v, u_i)$, we have 
\[
	(\varphi_i, v_i) = (\varphi^\st_i, \varphi^\wk_i, v^\st_i, v^\wk_i) \in \phi_{-T}^i\left(\,\bcG(L_i^s - c_i \cdot v, u_i) \right),	
\]
hence $(\varphi^\st_i, v^\st_i) \to (\varphi^\st, v^\st) \in \varphi^\st_{-T}\left( \bcG(L^\st - \bar{c} \cdot v^\st, u^\st) \right)$. Since $T>0$  is arbitrary, we obtain $(\varphi^\st, v^\st) \in \tilde{\cI}(L^\st_0 - \bar{c} \cdot v^\st, u^\st)$. 
\end{proof}

\begin{proof}[Proof of Proposition~\ref{prop:aubry-mane}, part I]
 We first prove item 1. Suppose $\tilde{\varphi}_i \in \tilde{\cN}_{H^s_i}(c_i)$, then there exists weak KAM solutions $u_i$ of $L^s_i - c_i\cdot v$, such that $(\varphi_i, v_i) \in \tilde{\cI}(L^s_i - c_i\cdot v, u_i)$. By Theorem~\ref{thm:semi-cont}, after restricting to a subsequence, we have $u_i(\varphi^\st, \varphi^\wk) \to u^\st(\varphi^\st)$. By Lemma~\ref{lem:semi-cont-I}, $(\varphi^\st_i, v^\st_i) \to (\varphi^\st, v^\st)$ implies $(\varphi^\st, v^\st) \in \tilde{\cI}(L^\st_0 - \bar{c}\cdot v^\st, u^\st)\subset \tilde{\cN}_{H^\st_0}(\bar{c})$. 

For item 3, suppose $\varphi_i = (\varphi_i^\st, \varphi_i^\wk) \in \cA_{H^s_i}(c_i)$ satisfies $\varphi_i^\st \to \varphi^\st \in \cA_{H^\st_0}(\bar{c})$. Then $h_{L^s_i, c_i}(\varphi_i, \cdot)$ is a weak KAM solution of $L^s_i - c\cdot v$ (see \cite{Fat08}, Theorem 5.3.6). By Theorem~\ref{thm:semi-cont}, by restricting to a subsequence, there exists a weak KAM solution $u^\st$ of $L^\st_0 - \bar{c}\cdot v^\st$ such that 
\[
	\lim_{i \to \infty} h_{L^s_i, c_i}(\varphi_i, \psi^\st, \psi^\wk) - h_{L^s_i, c_i}(\varphi_i, 0,0)   =  u^\st(\psi^\st). 
\]
We may further assume that $h_{L^s_i, c_i}(\varphi_i, 0,0) \to C \in \R$. Since $\cA_{H^\st_0}(\bar{c})$ has only one static class, there exists a constant $C_1>0$ such that
\[
	u^\st(\psi^\st) + C_1 = h_{L^\st, \bar{c}}(\varphi^\st, \psi^\st). 
\] 
Using the fact that $\varphi_i \in \cA_{H^s_i}(c_i)$, we get $h_{L^s_i, c_i}(\varphi_i, \varphi_i) =0$. Taking the limit, 
\[
	u^\st(\varphi^\st) = - C_1 = h_{L^\st, \bar{c}}(\varphi^\st, \varphi^\st) - C = -C. 
\]
Therefore 
\[
	\lim_{i \to \infty} h_{L^s_i, c_i}(\varphi_i^\st, \varphi_i^\wk, \psi^\st, \psi^\wk) = h_{L^\st, \bar{c}}(\varphi^\st, \psi^\st). 
\]

Item 4: Let $\rho_i = (\rho_i^\st, \rho_i^\wk)$ be the rotation number of minimal measures of $L_i^s - c_i \cdot v$, then  from Proposition~\ref{prop:rotation}, 
\[
	\lim_{i \to \infty} \rho^\wk_i - B_i^T A_i^{-1} \rho^\st_i - \tilde{C}_i c^\wk_i =0. 
\]
Moreover, assume that $\rho_i^\st \to \rho^\st \in \R^m$, then by taking limit in the second conclusion of Proposition~\ref{prop:rotation}, we get 
\[
	\alpha_{H^\st_0}(\bar{c}) + \beta_{H^\st_0}(\rho^\st) - \bar{c} \cdot \rho^\st =0,
\]
using Fenchel duality, $\rho^\st$ is a subdifferential of the convex function $\alpha_{H^\st_0}$ at $\bar{c}$. 
\end{proof}

\subsection{Semi-continuity of the Aubry set}
\label{sec:semi-aubry}

Our strategy of the proof mostly follow \cite{Ber10}. 

Given a compact metric space $\cX$, a semi-flow $\phi_t$ on $\cX$, and $\varepsilon, T>0$, an $(\varepsilon, T)-$chain consists of $x_0, \cdots, x_N \in \cX$ and $T_0, \cdots, T_{N-1} \ge T$, such that $d(\phi_{T_i} x_i, x_{i+1}) < \varepsilon$. We say that $x \fC_\cX y$  if for any $\varepsilon, T>0$, there exists an $(\varepsilon, T)-$chain with $x_0 = x$ and $x_N = y$. The relation $\fC_X$ is called the chain transitive relation (see \cite{Con88}).

The family of maps $\bphi_t = \phi_{t}$ defines a semi-flow on the set $\overline{\cG(L - c \cdot v, u)}$, and therefore defines a chain transitive relation. Given $\varphi, \psi \in \T^d$ and a weak KAM solution $u$ of $L - c \cdot v$,  we say that $\varphi \fC_u \psi$ if there exists $\tilde{\varphi} = (\varphi, v), \tilde{\psi} = (\psi, w) \in \T^d \times \R^d$ such that 
\[
	\tilde{\varphi} \fC_\cX \tilde{\psi},  \text{ where } \cX = \overline{\cG(L - c \cdot v, u)} . 
\]

Item 1 in the following Proposition is due to Ma\~ne, and item 2 is due to Mather. The version presented here is contained in \cite{Ber10}. 
\begin{proposition}
\label{prop:chain-trans}
Let $L$ be a Tonelli Lagrangian, then:
\begin{enumerate}
\item Let  $\varphi \in \cA_L(c)$ and $u$ be a weak KAM solution of $L - c\cdot v$, we have $\varphi \fC_u \varphi$. 
\item Suppose $\cA_L(c)$ has only finitely many static classes, and there exists a weak KAM solution $u$ such that $\varphi \fC_u \varphi$. Then $\varphi \in \cA_L(c)$. 
\end{enumerate}
\end{proposition}

Proposition~\ref{prop:chain-trans} implies that, when $\cA_L(c)$ has finitely many static classes, the Aubry set coincides with the set $\{ \varphi: \varphi \fC_u \varphi\}$. We will prove semi-continuity for this set. 
\begin{definition}
Let $\cX$ be a compact metric space with a semi-flow $\phi_t$. A family of piecewise continuous curves $x_i: [0, T_i] \to \cX$ is said to accumulate locally uniformly to $(\cX, \phi_t)$ if for any sequence $S_i \in [0,T_i]$, the curves $x_i(t + S_i)$ has a subsequence which converges uniformly on compact sets to a trajectory of $\phi_t$. 
\end{definition}
\begin{lemma}\cite{Ber10}
Suppose $x_i: [0, T_i] \to \cX$ accumulates locally uniformly to $(X, \phi_t)$, $x_i(0) \to x$ and $x_i(T_i) \to y$, then $x \fC_X y$. 
\end{lemma}

\begin{proof}[Proof of Proposition~\ref{prop:aubry-mane}, part II]
We prove item 2. Let $\varphi_i = (\varphi_i^\st, \varphi_i^\wk) \in \cA_{H^s_i}(c_i)$ and $\varphi_i^\st \to \varphi^\st$, we show that $\varphi^\st \in \cA_{H^\st_0}(\bar{c})$. According to Proposition~\ref{prop:chain-trans}, $\varphi_i \fC_u \varphi_i$. Let $\tilde{\varphi}_i$ be the unique point in $\tilde{\cA}_{H^s_i}(c_i)$ projecting to $\varphi_i$, then there exists weak KAM solutions $u_i$ of $L_i^s - c_i\cdot v$, such that  $\tilde{\varphi}_i \fC \tilde{\varphi}_i$ in $\bcG(L_i^s - c_i\cdot v, u_i)$. Fix $\varepsilon_i \to 0$ and $M_i \to \infty$, then for each $i$, there exists 
\[
	T_{i,1} < \cdots < T_{i, N_i}, \quad T_{i,j+1} - T_{i,j} >M_i, 
\]
and a piecewise $C^1$ curve $\gamma_i = (\gamma^\st_i, \gamma^\wk_i): [0, T_i] \to \T^d$, satisfying
\begin{enumerate}
\item $\gamma_i|(T_{i,j}, T_{i,j+1})$ satisfies the Euler-Lagrange equation of $L_i^s$;
\item 
\[
	d\bigl( ((\gamma_i(T_{i,j}-), \dot{\gamma}_i(T_{i,j}-)), \, ((\gamma_i(T_{i,j}+), \dot{\gamma}_i(T_{i,j}+)) \bigr) < \varepsilon_i. 
\]
\end{enumerate}
Using Lemma~\ref{lem:EL-est}, the projection of the Euler-Lagrange flow of $L^s_i$ to $(\varphi^\st, v^\st)$ converges uniformly over compact interval to the Euler-Lagrange flow of $L^\st_0$. This, combined with item 2 and Lemma~\ref{lem:semi-cont-I}, implies that $(\gamma_i^\st, \dot{\gamma}_i^\st)$ accumulates locally uniformly to 
\[
	(\bcG(L^\st_0 - \bar{c} \cdot v^\st, u^\st), \phi^\st_{-t})
\]
where $\phi^\st_t$ is the Euler-Lagrange flow of $L^\st_0$. Therefore $\varphi^\st_i \to \varphi^\st$ impies $\varphi^\st \fC_u \varphi^\st$. Using Proposition~\ref{prop:chain-trans} again, we get $\varphi^\st \in \cA_{H^\st_0}(\bar{c})$. 
\end{proof}

\section{Technical estimates on weak KAM solutions}
\label{sec:tech-est}

In this section we prove Proposition~\ref{prop:appr-lip-weak} and \ref{prop:appr-lip-strong}. 
For $(\cB^\wk, p, U^\st, \cU^\wk) \in \Omega^{m,d}_{\kappa, q}(\cB^\st)\cap \{ \|U^\st\| \le R\}$, recall the notations $H^s = \cH^s(\cB^\st,\cB^\wk, p, U^\st, \cU^\wk)$, $H^\st = \cH^\st(p, U^\st)$, $L^s = L_{H^s}$, $L^\st = L_{H^\st}$.

\subsection{Approximate Lipschitz property in the strong component}
\label{sec:almost-lip-strong}

In this section we show that Proposition~\ref{prop:appr-lip-weak} implies Proposition~\ref{prop:appr-lip-strong}. Proposition~\ref{prop:appr-lip-weak} is proven in the next two sections.

We first state a lemma of action comparison between an extremal curve and its ``linear drift''. 
\begin{lemma}
\label{lem:lin-drift} 
Let $L: \T^d \times \R^d \to \R$ be a Tonelli Hamiltonian, $T \ge 1$, and  $\gamma:[0,T]\to \T^d$  be an extremal curve. Then for any $1 \le i \le d$, $h >0$, and a unit vector $f \in \R^d$, 
\begin{multline*}
	\int_0^T L(\gamma + \frac{th}{T} f, \dot{\gamma} + \frac{h}{T}f)dt - \int_0^T L(\gamma, \dot{\gamma}) dt \\
	\le (\partial_{v} L(\gamma(T), \dot{\gamma}(T)) \cdot f )h  + \left(\|f\cdot (\partial^2_{v v}L) f\| \frac{1}{T} + \|f \cdot (\partial^2_{\varphi v}L) f\|  + T \|f \cdot (\partial^2_{\varphi \varphi} L) f\|\right) h^2. 
\end{multline*}
\end{lemma}
\begin{proof}
We compute
\begin{multline*}
L(\gamma + \frac{th}{T} f, \dot{\gamma})  - L(\gamma, \dot{\gamma} + \frac{h}{T}  f) 
 \le \partial_{\varphi} L(\gamma, \dot{\gamma})\cdot \frac{th}{T}f  + \partial_{v} L(\gamma, \dot{\gamma}) \frac{h}{T}f \\
  \|f\cdot (\partial^2_{v v}L) f\|  \frac{h^2}{T^2} + \|f \cdot (\partial^2_{\varphi v}L) f\| \frac{th^2}{T^2}  + \|f \cdot (\partial^2_{\varphi \varphi} L) f\| \frac{t^2 h^2}{T^2}. 
\end{multline*}
It follows from the Euler-Lagrange equation that 
\[
	 \partial_{\varphi} L(\gamma, \dot{\gamma})\cdot \frac{th}{T} + \partial_{v} L(\gamma, \dot{\gamma}) \frac{h}{T} = \frac{d}{dt} \left( \partial_{v}L \frac{th}{T} \right), 
\]
and our estimate follows from direct integration. 
\end{proof}

The following lemma establishes a relation between ``approximate semi concavity'' with approximate Lipschitz property. 
\begin{lemma}\label{lem:app-semi-conc}
For $D, \delta>0$, assume that  $u: \T^d \to \R$ satisfies that for all $\varphi\in \T^d$, there exists $l \in \R^d$ such that 
\[
	u(\varphi + y) - u(\varphi) \le  l \cdot y + D \|y\|^2 + \delta, \quad y\in \R^d, 
\]
Then $\|l\| \le \sqrt{d}(D+\delta)$, and $u$ is $(2\sqrt{d}(D+\delta), \delta)$ approximately Lipschitz. 
\end{lemma}
\begin{proof}
Assume that $l = (l_1, \cdots, l_d)$. For each $1 \le i \le d$, we pick $y = -e_i \frac{l_i}{|l_i|}$, where $e_i$ is the coordinate vector in $\varphi_i$. Then 
\[
	0 = u(\varphi + e_i) - u(\varphi) \le -|l_i| +  D + \delta,
\]
so $|l_i| \le D + \delta$. As a result $\|l\| \le \sqrt{d}(D + \delta)$. For any $y \in [0,1]^d$, we have  $\|y\|\le \sqrt{d}$ and 
\[
	u(\varphi + y) - u(\varphi) \le (\sqrt{d}(D + \delta) + D\|y\|)\|y\| + \delta < 2\sqrt{d}(D + \delta) \|y\| + \delta.
\]
\end{proof}

\begin{proof}[Proof of Proposition~\ref{prop:appr-lip-strong}]
Since $u$ is a weak KAM solution, for any $\varphi \in \T^d$, let $\gamma = (\gamma^\st, \gamma^\wk): (-\infty, 0] \to \T^d$ be a $(u, L^s - c \cdot v, \alpha_H(c))$-calibrated curve with $\gamma(0)= \varphi = (\varphi^\st, \varphi^\wk)$. Then for any $T>0$
\[
	u(\varphi) = u(\gamma(-T)) + \int_{-T}^0 (L^s - c\cdot v + \alpha_{H^s}(c))(\gamma, \dot{\gamma}) dt. 
\]
Using \eqref{eq:Lsc-split}, we get 
\begin{multline}\label{eq:u-lower}
u(\varphi) = u(\gamma(-T)) + \int_{-T}^0 (L^\st - \bar{c} \cdot v^\st)(\gamma^\st, \dot{\gamma}^\st) dt + (\alpha_{H^s}(c) - \frac12 c^\wk \cdot \tilde{C}^{-1} c^\wk)T\\
 +  \int_{-T}^0 \frac12 (\dot{\gamma}^\wk - B^T A^{-1} \dot{\gamma}^\st - \tilde{C} c^\wk) \cdot \tilde{C}^{-1} (\dot{\gamma}^\wk - B^T A^{-1} \dot{\gamma}^\st - \tilde{C} c^\wk) + U^\wk(\gamma(t)) dt . 
\end{multline}

We now produce an upper bound using a special test curve. Let $\gamma^\st_0: [-T, 0] \to \T^m$ be such that 
\begin{equation}
\label{eq:gamma-0-st}
	\int_{-T}^0 (L^\st - \bar{c} \cdot v^\st)(\gamma^\st_0, \dot{\gamma}^\st_0) dt  = \min_{\zeta}\int_{-T}^0 (L^\st - \bar{c} \cdot v^\st)(\zeta, \dot{\zeta}) dt
\end{equation}
where the minimum is over all $\zeta(-T) = \gamma^\st(-T)$ and $\zeta(0) = \gamma^\st(0)$. 

We define $\xi = (\xi^\st, \xi^\wk): [-T, 0] \to \T^d$ as follows. 
\begin{enumerate}
\item For $y\in \R^d$, 
\[
	\xi^\st(t) = \gamma^\st_0(t) + \frac{T+t}{T}y.
\]
The curve $\xi^\st$ is a linear drift over $\gamma^\st_0$ with $h = \|y\|$ and $f = \frac{y}{\|y\|}$ (see Lemma~\ref{lem:lin-drift}). 
\item Define
\[
	\xi^\wk(t) = \gamma^\wk(-T) + B^T A^{-1}(\xi^\st(t) - \gamma^\st_0(-T)) + \tilde{C}c^\wk (T+t).  
\]
We note that $\xi^\wk(-T) = \gamma^\wk(-T)$ and 
\[
	\dot{\xi}^\wk_0 - B^T A^{-1}\xi^\st - \tilde{C} c^\wk =0.
\]
\end{enumerate}

Using the fact that $u$ is dominated by $L^s - c\cdot v + \alpha_{H^s}(c)$, we have
\begin{align*}
&u(\varphi^\st + y, \xi^\wk(0)) \le  u(\gamma(-T)) + \int_{-T}^0 (L^s - c\cdot v + \alpha_{H^s}(c))(\xi, \dot{\xi}) dt \\
& =  u(\gamma(-T))  + \int_{-T}^0 (L^\st - \bar{c} \cdot v^\st)(\xi^\st, \dot{\xi}^\st) dt  \\
& + \int_{-T}^0 \frac12 (\dot{\xi}^\wk - B^T A^{-1} \dot{\xi}^\st - \tilde{C} c^\wk) \cdot \tilde{C}^{-1} (\dot{\xi}^\wk - B^T A^{-1} \dot{\xi}^\st - \tilde{C} c^\wk)  dt \\ 
& + (\alpha_{H^s}(c) - \frac12 c^\wk \cdot \tilde{C}^{-1} c^\wk)T + \int_{-T}^0 U^\wk(\xi) dt
\end{align*}
and note that the third line in the above formula vanishes, using the definition of $\xi^\wk$. Combine with \eqref{eq:u-lower}, we get 
\begin{multline*}
	u(\varphi^\st + y, \xi^\wk(0)) - u(\varphi^\st, \varphi^\wk) \\
	\le \int_{-T}^0 (L^\st - \bar{c} \cdot v^\st)(\xi^\st, \dot{\xi}^\st) dt  - \int_{-T}^0 (L^\st - \bar{c} \cdot v^\st)(\gamma^\st, \dot{\gamma}^\st) dt + 2 \|U^\wk\|_{C^0}. 
\end{multline*}
From \eqref{eq:gamma-0-st} we get 
\begin{multline*}
	u(\varphi^\st + y, \xi^\wk(0)) - u(\varphi^\st, \varphi^\wk) \\
	\le \int_{-T}^0 (L^\st - \bar{c} \cdot v^\st)(\xi^\st, \dot{\xi}^\st) dt  - \int_{-T}^0 (L^\st - \bar{c} \cdot v^\st)(\gamma^\st_0, \dot{\gamma}^\st_0) dt + 2 \|U^\wk\|_{C^0}. 
\end{multline*}
Since $\gamma_0^\st$ is an extremal of $L^\st - \bar{c}\cdot v^\st$, the linear drift lemma (Lemma~\ref{lem:lin-drift}) applies. Noting that $\|\partial^2_{v^\st v^\st}L^\st\| \le \|A^{-1}\|$, $\|\partial^2_{\varphi^\st \varphi^\st}L\| \le \|U^\st\|_{C^2} \le R$, and $\partial^2_{\varphi^\st v^\st}L =0$. We obtain from Lemma~\ref{lem:lin-drift} that 
\begin{multline*}
	\int_{-T}^0 (L^\st - \bar{c} \cdot v^\st)(\xi^\st, \dot{\xi}^\st) dt  - \int_{-T}^0 (L^\st - \bar{c} \cdot v^\st)(\gamma^\st_0, \dot{\gamma}^\st_0) dt \\ 
	 \le l \cdot y + (\|A^{-1}\| + \|U^\st\|_{C^2}) \|y\|^2, 
\end{multline*}
where $l = \partial_v L^\st(\gamma^\st_0(0), \dot{\gamma}^\st_0(0))$. Note that $\|A^{-1}\| + \|U^\st\|_{C^2}$ is a constant depending only on $\cB^\st, Q, R$. 

We now invoke Proposition~\ref{prop:appr-lip-weak} to get 
\[
	|u(\varphi^\st + y, \xi^\wk(0)) - u(\varphi^\st + y, \varphi^\wk)| \le \delta |\xi^\wk(0) - \varphi^\wk| + \delta \le 2\delta,
\]
where $\delta = M^*_1 \mu(\cB^\wk)^-(\frac{q}{2}-d+m)$ for some $M^*_1 = M^*_1(\cB^\st, Q, \kappa, q, R)$. Combine all the estimates, we get 
\[
	u(\varphi^\st + y, \varphi^\wk) - u(\varphi^\st, \varphi^\wk) \le l \cdot y + (\|A^{-1}\| + \|U^\st\|_{C^2}) \|y\|^2 + 2\delta + 2\|U^\wk\|. 
\]
We note that in $\Omega^{m,d}_{\kappa,q}$ we have $\|U^\wk\|_{C^2} \le \sum_{i=1}^{d-m}\|U^\wk_i\|_{C^2} \le (d-m)\kappa (\lM(\cB^\wk))^{-q}$. We may choose $M^*_2 = M^*_2(\cB^\st, Q, \kappa, q, R)$, such that
\[
2\delta + 2\|U^\wk\| \le M^*_2 \lM(\cB^\wk))^{-(\frac{q}{2}-d+m)} =: \delta'.
\]

We now apply Lemma~\ref{lem:app-semi-conc} to get $u(\cdot, \varphi^\wk)$ is 
\[
	(2\sqrt{d}(\|A^{-1}\| + \|U^\st\|_{C^2} + \delta'), \delta')
\]
approximately Lipschitz. Define $M' = 2 \sqrt{d} (\|A^{-1}\| + \|U^\st\|_{C^2} + M^*_2)$, and the Proposition follows.
\end{proof}

\subsection{Finer decomposition of the slow Lagrangian}
For the proof of Proposition~\ref{prop:appr-lip-weak}, we need a finer decomposition of the Lagrangian $L^s$ which treat all $\varphi^\wk_i$, $1 \le i \le d-m$ separately. First, we have the following linear algebra identity. (The proof is direct calculation)
\begin{lemma}\label{lem:block-diag}
Let $S =\bmat{ A & B \\ B^T & C}$ be a nonsingular symmetric matrix in block form. Then 
\[
	\bmat{\Id & 0 \\ - B^T A^{-1} & \Id } \bmat{ A & B \\ B^T & C} \bmat{\Id & - A^{-1} B \\  0 & \Id}  = \bmat{A & 0 \\ 0 & \tilde{C}}, 
\]
where $\tilde{C} = C - B^T A^{-1} B$. In particular, $\tilde{C}$ is positive definite if $S$ is. 
\end{lemma}
We write $H^s(\varphi, I) = K(I) - U(\varphi) = K(I) - U^\st(\varphi^\st) - U^\wk(\varphi)$ and $S = \partial^2_{II}K$.  We describe a coordinate change block diagonalizing $\partial^2_{II}K$. Write $S$ in the following block form
\[
	S = 
	\bmat{ X_{d-m} & y_{d-m} \\ y^T_{d-m} & z_{d-m} }, \quad 
	X_{d-m}\in M_{(d-1)\times (d-1)},\ y_{d-m} \in \R^{d-1},\ z_{d-m}\in \R, 
\]
and for each $1 \le i \le d-m-1$, further decompose each $X_{i+1}$ as
\[
	X_{i+1} =
	\bmat{ X_i & y_i \\ y_i^T & z_i}, \quad X_i \in M_{(m+i-1)\times (m+i-1)},\ y_i \in \R^{m+i-1},\ z_i \in \R. 
\]
Note that in this notation, $X_1 =  \partial^2_{I^\st I^\st}K = A$ (see \eqref{eq:block-ABC}). 

Define, for $1 \le i \le d-m$, 
\[
	E_i = 
	\bmat{ \Id_{m+i-1} & - X_i^{-1} y_i & 0 \\
           0 & 1 & 0 \\
           0 & 0 & \Id_{d-m-i}
	}, 
\]
where $\Id_i$ denote the $i\times i$ identity matrix. Then by Lemma~\ref{lem:block-diag}
\begin{multline*}
	E_{d-m}^T S  E_{d-m} =  \\
	\bmat{ \Id_{d-1} & 0 \\ -y_{d-m}^T X_{d-m}^{-1} & 1}
	\bmat{ X_{d-m} & y_{d-m} \\ y^T_{d-m} & z_{d-m} } \bmat{ \Id_{d-1} & - X_{d-m}^{-1}y_{d-m} \\ 0 & 1} =
	\bmat{ X_{d-m} & 0 \\ 0 & \tilde{z}_{d-m} }, 
\end{multline*}
where $\tilde{z}_{d-m} = z_{d-m} - y_{d-m}^T X_{d-m}^{-1} y_{d-m}$. Moreover, for each $1 \le i \le d-m-1$, 
\begin{multline} \label{eq:Xi}
	\bmat{\Id_{m+i-1} & 0 \\ - y_i^T X_i^{-1} & 1} X_{i+1} \bmat{ \Id_{m+i-1} & - X_i^{-1} y_i \\ 0 & 1  } \\
	=	\bmat{\Id_{m+i-1} & 0 \\ - y_i^T X_i^{-1} & 1} \bmat{X_i & y_i \\ y^T_i & z_i} \bmat{ \Id_{m+i-1} & - X_i^{-1} y_i \\ 0 & 1  } 
	= \bmat{ X_i & 0  \\ 0 & \tilde{z}_i}.
\end{multline}

Let 
\begin{equation}
\label{eq:E}
	E= E_{d-m} \cdots E_1 = \left[
	\begin{array}{ccccc}	\Id_m & - X_1^{-1}y_1 & \multirow{2}{*}{$- X_2^{-1}y_2$} & \multirow{3}{*}{$\cdots$} & \multirow{4}{*}{$- X_{d-m}^{-1}y_{d-m}$} \\
	  & 1 & &  & \\
	  & & 1& & \\
	  & & & \ddots & \\
	  & & & & 1		
	\end{array} \right],
\end{equation}
then recursive computation yields
\begin{equation}
\label{eq:S-diag}
		E^T S  E = E_1^T \cdots E_{d-m}^T S E_{d-m} \cdots E_1 \\
	= \bmat{ X_1 &  &  &  \\
	          &  \tilde{z}_1 & & \\
	          & & \ddots &  \\
	          & & & \tilde{z}_{d-m}
	} =:\tilde{S}.
\end{equation}

We summarize the characterization of the Lagrangian in the following lemma. For $v = (v^\st, v_1^\wk, \cdots, v_{d-m}^\wk) \in \R^m\times \R^d$, we define
\begin{equation}
\label{eq:floor}
	\flr{v}_0 = v^\st, \quad \flr{v}_i = (v^\st, v^\wk_1, \cdots, v^\wk_i), \, 1 \le i \le d-m. 
\end{equation}
\begin{lemma}
\label{lem:lag-fine-dec}
For $v, c \in \R^d$ we denote $w = E^T v$ and $\eta = E^{-1} c$, where $E$ is defined in \eqref{eq:E}. Explicitly, we have 
\begin{equation}
\label{eq:v-to-w}
	w = \bmat{w^\st \\ w^\wk_1 \\ \vdots \\ w_{d-m}^\wk} = 
	\bmat{ v^\st \\ v_1^\wk - y_1^T X_1^{-1} \flr{v}_0 \\ \vdots \\ v_{d-m}^\wk  - y_{d-m}^T X_{d-m}^{-1} \flr{v}_{d-m-1}  }, 
\end{equation}
and
\[
	\eta^\st = c^\st + A^{-1} B c^\wk, \quad \eta = (\eta^\st, \eta^\wk), c = (c^\st, c^\wk), 
\]
 where $A, B$ are defined in \eqref{eq:block-ABC}.
Then we have 
\begin{multline}
\label{eq:Ls-fine-decomp}
L^s(\varphi, v) - c \cdot v  \\
= L^\st(\varphi^\st, v^\st) - \eta^\st \cdot v^\st + \sum_{i=1}^{d-m} \left(  \frac12 \tilde{z}_i^{-1}(w_i^\wk - \tilde{z}_i\eta_i^\wk)^2  - \frac12 z_i (\eta_i^\wk)^2 + U_i^\wk(\varphi)  \right). 
\end{multline}
\end{lemma}
\begin{remark}
This is a finer version of Lemma~\ref{lem:ls-split}. In particular, the strong component $L^s - \eta^\st \cdot v^\st$ is identical to the $L^s - \bar{c} \cdot v^\st$ defined in Lemma~\ref{lem:ls-split}. 
\end{remark}

\begin{proof}
Formula \eqref{eq:v-to-w} can be read directly from the definition \eqref{eq:E} and $w = E^T v$. To show $\eta^\st = c^\st + A^{-1}B c^\wk$, we compute
\[
	\bmat{A & 0 \\ 0 & \tilde{C}} \bmat{\eta^\st \\ \eta^\wk} = \tilde{S} \eta = \tilde{S} E^{-1} c = E^T S c = \bmat{\Id_m & 0 \\ * & *} \bmat{A & B \\ B^T & C} \bmat{c^\st \\ c^\wk}.
\]
The first block of the above equation yields $A \eta^\st = A c^\st + B c^\wk$, hence $\eta^\st = c^\st + A^{-1}B c^\wk$. 

We now prove \eqref{eq:Ls-fine-decomp}. We have 
\begin{align*}
& L^s(\varphi, v) - c\cdot v =  \frac12 v^T S^{-1} v - c^T v + U^\st + U^\wk \\
& = \frac12 (E^T v) \tilde{S}^{-1} (E^T v) - (E^{-1} c)^T (E^T v) + U^\st + U^\wk \\
& = \left( \frac12 w^\st \cdot A^{-1} w^\st - \eta^\st \cdot w^\st + U^\st  \right) + \sum_{i=1}^{d-m} \left(  \frac12 \tilde{z}_i^{-1}(w_i^\wk)^2 - \eta_i^\wk w_i^\wk + U_i^\wk \right). 
\end{align*}
In the above formula, the first group is equal to $L^\st - \eta^\st \cdot v^\st$, noting $w^\st = v^\st$. Moreover
\[
	\frac12 \tilde{z}_i^{-1}(w_i^\wk)^2 - \eta_i^\wk w_i^\wk = 
	\  \frac12 \tilde{z}_i^{-1}(w_i^\wk - \tilde{z}_i \eta^\wk_i)^2 - \frac12 \tilde{z}_i (\eta_i^\wk)^2, \quad 1  \le i \le d-m, 
\]
and \eqref{eq:Ls-fine-decomp} follows. 
\end{proof}

We derive some useful estimates.
\begin{lemma}
\label{lem:z-j-bounds}
There exists $M^* = M^*(\cB^\st, Q, \kappa, q) >1$ such that, for
\[
	L^s = L_{\cH^s}(\cB^\wk, p, U^\st, \cU^\st), \quad (\cB^\wk, p, U^\st, \cU^\st) \in \Omega^{m,d}_{\kappa, q},
\]
 the following hold. 
\begin{enumerate}
\item For each $1 \le i \le d-m$, we have $\sum_{j=i}^{d-m}\|U_j^\wk\|_{C^2} \le M^* |k_i^\wk|^{-q}$.
\item For each $1 \le i \le d-m$, $\tilde{z}_i^{-1} \le M^* |k_i^\wk|^{2i}$. 
\end{enumerate}
\end{lemma}
\begin{proof}
For item 1, note that for each $j \ge i$, $|k_i^\wk| \le \kappa |k_j^\wk|$, hence
\[
	\|U_j^\wk\|_{C^2} \le \kappa |k_j^\wk|^{-q} \le \kappa^{1+q} |k_i^\wk|^{-q}. 
\]
Item 1 holds for any $M^* \ge (d-m) \kappa^{1+q}$. 

For item 2,   inverting \eqref{eq:Xi} we get  
\[
	 X_{i+1}^{-1} 
	= \bmat{ \Id_{m+i-1} & - X_i^{-1} y_i \\ 0 & 1  } \bmat{X_i^{-1} & 0 \\ 0 & \tilde{z}_i^{-1}}\bmat{\Id_{m+i} & 0 \\ - y_{i+1}^T X_{i+1}^{-1} & 1}.
\]
Denote $f = (0, \cdots, 0, 1)\in \T^{m+i}$, then 
\[
	f^T X_{i+1} f = f^T \bmat{ \Id_{m+i-1} & - X_i^{-1} y_i \\ 0 & 1  } \bmat{X_i^{-1} & 0 \\ 0 & \tilde{z}_i^{-1}}\bmat{\Id_{m+i} & 0 \\ - y_{i+1}^T X_{i+1}^{-1} & 1} f = \tilde{z}_i^{-1}. 
\]
Moreover, using the definition (see \eqref{eq:Kn})
\[
S = \partial^2_{II}K = \bmat{k_1^\st & \cdots & k_m^\st & k_1^\wk & \cdots & k_{d-m}^\wk}^T Q \bmat{k_1^\st & \cdots & k_m^\st & k_1^\wk & \cdots & k_{d-m}^\wk}, 
\]
we have
\begin{multline*}
	X_{i+1} = \bmat{k_1^\st & \cdots & k_m^\st & k_1^\wk & \cdots & k_i^\wk}^T Q \bmat{k_1^\st & \cdots & k_m^\st & k_1^\wk & \cdots & k_i^\wk} \\
	= \bmat{\bark_1^\st & \cdots & \bark_m^\st & \bark_1^\wk & \cdots & \bark_i^\wk}^T Q_0 \bmat{\bark_1^\st & \cdots & \bark_m^\st & \bark_1^\wk & \cdots & \bark_i^\wk} =: \bar{P}^T Q_0 \bar{P} ,
\end{multline*}
where $\bark$ is the first $n$ components of $k$. We have assumed  $Q_0 \ge D^{-1} \Id$ for $D>1$. By Lemma~\ref{lem:int-norm}, there exists a constant $c_n>1$ depending only on $n$ such that 
\begin{multline*}
	\|X_{i+1}^{-1}\| = (\min_{\|v\|=1} v^T X_{i+1} v)^{-1} = (\min_{\|v\|=1}v^T\bar{P}Q_0 \bar{P})^{-1}  
	\le D \|\bar{P}^{-1}\|^2 \\
	\le D c_n |k_1^\st|^2 \cdots |k_m^\st|^2 |k_1^\wk|^2 \cdots |k_i^\wk|^2 \le D c_n \bar{M}^m \kappa^{i-1} |k_i^\wk|^{2i},
\end{multline*}
where $\bar{M} = |k_1^\st| + \cdots + |k_m^\st|$ depend only on $\cB^\st$. 
\end{proof}

\subsection{Approximate Lipshitz property in the weak component}

In this section we prove Proposition~\ref{prop:appr-lip-weak}. We fix $(\cB^\wk, p, U^\st, \cU^\st) \in \Omega^{m,d}_{\kappa, q} \cap \{ \|U^\st\|_{C^2} \le R\}$, and write $L^s = L_{\cH^s}(\cB^\wk, p, U^\st, \cU^\st)$. 

For $c\in \R^d$, we define
\begin{multline}
\label{eq:ls-c-i}
L^s_{c,i}(\varphi^\st, \varphi_1^\wk, \cdots, \varphi_i^\wk, v^\st, v_1^\wk, \cdots, v_i^\wk)  = L_{c,i}^s(\flr{\varphi}_i, \flr{v}_i) \\
= L^\st(\varphi^\st, v^\st) - \eta^\st \cdot v^\st + \sum_{j=1}^i \left(  \frac12 \tilde{z}_j^{-1}(w_j^\wk - \tilde{z}_j \eta_j^\wk)^2  - \frac12 z_j (\eta_j^\wk)^2 + U_j^\wk(\varphi)  \right),
\end{multline}
then
\begin{multline}
\label{eq:lsc-i-split}
L^s(\varphi, v) - c \cdot v  \\
= L_{c,i}^s(\flr{\varphi}_i, \flr{v}_i) + \sum_{j={i+1}}^{d-m} \left(  \frac12 \tilde{z}_j^{-1}(w_j^\wk - \tilde{z}_j\eta_j^\wk)^2  - \frac12 z_j (\eta_j^\wk)^2 + U_j^\wk(\varphi)  \right).
\end{multline}

Our proof of Proposition~\ref{prop:appr-lip-weak} follows an inductive scheme. Following our notational convention,  denote $e_i^\wk = e_{i+m}$, which is the coordinate vector of $\varphi_i^\wk$. 
\begin{lemma}\label{lem:last-var}
Let $u: \T^d \to \R$  be a weak KAM solution of $L^s - c\cdot v$. Then for 
\[
	\delta_{d-m}:= 2 (\tilde{z}_{d-m}^{-1}\|U_{d-m}^\wk\|_{C^2})^{\frac12},
\]
we have $u$ is $\delta_{d-m}-$semi-concave and $\delta_{d-m}-$Lipschitz in $\varphi_{d-m}^\wk$.
\end{lemma}
\begin{proof}
First we have
\[
	\partial^2_{\varphi_{d-m}^\wk \varphi_{d-m}^\wk} L^s = \partial^2_{\varphi_{d-m}^\wk \varphi_{d-m}^\wk}U_{d-m}^\wk, \quad \partial^2_{\varphi_{d-m}^\wk v_{d-m}^\wk}L^s =0, \quad \partial^2_{v_{d-m}^\wk v_{d-m}^\wk }L^s = \tilde{z}_{d-m}^{-1}. 
\]
The first two equality follows directly from the definition, while the last one uses \eqref{eq:v-to-w} and \eqref{eq:Ls-fine-decomp}.

For any $\varphi \in \T^d$, let $\gamma: (-\infty, 0] \to \T^d$ be a $(u, L^s, c)$-calibrated curve with $\gamma(0)= \varphi$. Then for any $T>0$
\[
	u(\varphi) = u(\gamma(-T)) + \int_{-T}^0 (L^s - c\cdot v + \alpha_{H^s}(c))(\gamma, \dot{\gamma}) dt. 
\]
Using the definition of the weak KAM solution, 
\[
	u(\varphi + h e_i^\wk) \le u(\gamma(-T)) + \int_{-T}^0 (L^s - c\cdot v + \alpha_{H^s}(c))(\gamma + \frac{th}{T}e_{d-m}^\wk, \dot{\gamma} + \frac{h}{T} e_{d-m}^\wk) dt. 
\]

Subtract the two estimates, and apply Lemma~\ref{lem:lin-drift} to $L^s-c\cdot v + \alpha_{H^s}(c)$ and $\gamma$, we get 
\begin{align*}
u(\varphi+ h e_i^\wk) - u(\varphi) &
 \le 
(\partial_{v_{d-m}^\wk}L^s(\gamma(0) , \dot{\gamma}(0))- c_{d-m}) h  \\
&+ \left(\|\partial^2_{v_{d-m} v_{d-m}}L^s\| \frac{1}{T} + \|\partial^2_{\varphi_{d-m}^\wk v_{d-m}^\wk}L^s\|  + T \|\partial^2_{\varphi_{d-m}^\wk \varphi_{d-m}^\wk} L^s\|\right) h^2 \\
&\le (\partial_{v_{d-m}^\wk}L^s(\gamma(0)) , \dot{\gamma}(0)- c_{d-m}^\wk) h + \left( \tilde{z}_{d-m}^{-1}/T + \|U_{d-m}^\wk\|_{C^2}T \right) h^2,
\end{align*}
Take $T = (\tilde{z}_{d-m} \|U_{d-m}^\wk\|_{C^2})^{-\frac12}$, and write $l = \partial_{v_{d-m}^\wk}L^s(\gamma(0)) , \dot{\gamma}(0)- c_{d-m}^\wk$, we get 
\[
	u(\varphi + he_i^\wk) - u(\varphi) \le lh + \frac12 \delta_{d-m} h^2. 
\]
The semi-concavity estimate follows. Using the fact that $u$ is $\Z^d$ periodic, we take $h = l/|l|$ to get $|l| \le \frac12 \delta_{d-m}$. Therefore for $|h| \le 1$,
\[
	|u(\varphi +  h e_i^\wk) - u(\varphi)| \le (\frac12 \delta_{d-m} + \frac12\delta_{d-m}h) h \le \delta_{d-m} h.
\] 
This is the Lipschitz estimate. 
\end{proof}

We now state the inductive step. 
\begin{proposition}\label{prop:var-induction}
Let $u: \T^d \to \R$  be a weak KAM solution of $L^s - c\cdot v$. Assume that for a given $1 \le i \le d-m-1$,  $u$ is $(\delta_j, \delta_j)$ approximately Lipschitz in $\varphi^\wk_j$ for all $i+1 \le j \le d-m$. Then for 
\[
	\sigma_i =  \left( \tilde{z}_i^{-1} \sum_{j=i}^{d-m}\|U_j^\wk\|_{C^2}\right)^{\frac12},  \quad 
	\delta_i = \sqrt{d}(6 \sigma_i + 4 \sum_{j=i+1}^{d-m} \delta_j),
\]
we have $u$ is $(\delta_i, \delta_i)$ approximately Lipschitz in $\varphi_i^\wk$. 
\end{proposition}
\begin{proof}
The proof is very similar to the proof of Proposition~\ref{prop:appr-lip-strong}, but uses the finer decomposition in this section. 

Since $u$ is a weak KAM solution, then given any $\varphi\in \T^d$, there exists a calibrated curve $\gamma: (-\infty, 0] \to \T^d$ with $\gamma(0) = \varphi$. Then for any $T>0$
\[
	u(\varphi) = u(\gamma(-T)) + \int_{-T}^0 (L^s - c\cdot v + \alpha_{H^s}(c))(\gamma, \dot{\gamma}) dt. 
\]
Let  $h \in \R$, $\chi \in \R^d$, and a $C^1$ curve $\xi: [-T, 0]\to \T^d$ satisfies
\[
	\xi(-T) = \gamma(-T), \quad \xi(0) = \varphi + h e_i^\wk + \chi, 
\]
then 
\begin{equation}
\label{eq:u-diff}
\begin{aligned}
 u(\varphi + h e_i^\wk + \chi) &\le u(\gamma(-T)) + \int_{-T}^0 (L^s - c\cdot v + \alpha_{H^s}(c))(\xi, \dot{\xi}) dt \\
& \le u(\gamma(-T)) + \int_{-T}^0 (L^s - c\cdot v + \alpha_{H^s}(c))(\gamma, \dot{\gamma}) dt \\
& + \int_{-T}^0 (L^s - c\cdot v)(\xi, \dot{\xi}) - \int_{-T}^0 (L^s - c\cdot v)(\gamma, \dot{\gamma}) dt \\
& = u(\varphi) + \int_{-T}^0 (L^s - c\cdot v)(\xi, \dot{\xi}) - \int_{-T}^0 (L^s - c\cdot v)(\gamma, \dot{\gamma}) dt. 
\end{aligned}
\end{equation}
We will first give the precise definition of $\xi$, then estimate \eqref{eq:u-diff}, before finally obtain the desired estimate. 

\myheading{Definition of $\xi$.} Recall the Lagrangian $L^s_{c,i}: \T^{m+i} \times \R^{m+i} \to \R$ defined in \eqref{eq:ls-c-i}. Let $\xi: [-T, 0] \to \T^{m+i}$ be an $L^s_{c,i}$ minimizing curve satisfying the constraint
\[
	\zeta(-T) = \flr{\gamma}_i (-T), \quad \zeta(0) = \flr{\gamma}_i(0),
\]
where $\flr{\cdot}_i$ is defined in \eqref{eq:floor}. For $h \in \R$, we define $\xi$ in the following way. 

\begin{enumerate}
\item The first $m+i$ components of $\xi$ is $\zeta$ with an added linear drift in $e_i^\wk$, more precisely,
\begin{equation}
\label{eq:xi-flr-i}
	\flr{\xi}_i(t) = \zeta(t) + \frac{th}{T} e_i^\wk. 
\end{equation}
\item We define the other components inductively. For $i < j \le d-m$, suppose $\flr{\xi}_{j-1}(t) = (\xi^\st, \xi_1^\wk, \cdots, \xi_{j-1}^\wk)(t)$ has been defined. We define
\[
	\xi_j^\wk(t) = \gamma_j^\wk(t) + y_j^T X_j^{-1} \flr{\xi}_{j-1}(t) - y_j^T X_j^{-1} \flr{\gamma}_{j-1}(t). 
\]
\end{enumerate}
For each $i < j \le d-m$, we have
\begin{equation}
\label{eq:dot-xi}
\begin{cases}
\xi_j^\wk(-T) = \gamma_j^\wk(-T), \\
\dot{\xi}_j^\wk - y_j^T X_j^{-1} \flr{\dot{\xi}}_{j-1} = \dot{\gamma}_j^\wk - y_j^T X_j^{-1} \flr{\dot{\gamma}}_{j-1}. 
\end{cases}
\end{equation}
We define  $\chi = \xi(0) - \varphi - h e_i^\wk$, and note that from \eqref{eq:xi-flr-i},
\[
	\flr{\chi}_i = \flr{\xi}_i(0) - \flr{\gamma}_i(0) - h e_i^\wk =0. 
\]

\myheading{Action comparison.} We now compute
\begin{equation}\label{eq:lsci-comp}
\begin{aligned}
& \int_{-T}^0 (L^s - c\cdot v)(\xi, \dot{\xi}) dt - \int_{-T}^0 (L^s - c\cdot v)(\gamma, \dot{\gamma}) dt \\
& = \int_{-T}^0 L^s_{c,i}(\flr{\xi}_i, \flr{\dot{\xi}}_i) dt - \int_{-T}^0 L^s_{c,i} (\flr{\gamma}_i, \flr{\dot{\gamma}}_i) dt + \sum_{j=i+1}^{d-m} \int_{-T}^0 \left( U_j^\wk(\xi(t)) - U_j^\wk(\gamma(t))  \right) dt\\
& +  \frac12  \sum_{j=i+1}^{d-m} \tilde{z}_j^{-1} \int_{-T}^0 \left( (\xi_j^\wk - y_j^T X_j^{-1} \flr{\dot{\xi}}_{j-1}- \tilde{z}_j \eta_j^\wk)^2 - (\gamma_j^\wk - y_j^T X_j^{-1} \flr{\dot{\gamma}}_{j-1} - \tilde{z}_j \eta_j^\wk )^2\right)\\
& \le \int_{-T}^0 L^s_{c,i}(\flr{\xi}_i, \flr{\dot{\xi}}_i) dt - \int_{-T}^0 L^s_{c,i} (\flr{\gamma}_i, \flr{\dot{\gamma}}_i) dt + 2T \sum_{j=i+1}^{d-m} \|U_j^\wk\|_{C^0} . 
\end{aligned}
\end{equation}
In the above formula, the  equality is due to \eqref{eq:lsc-i-split}. Moreover, observe that from \eqref{eq:dot-xi}, the third line of the above formula vanishes. The inequality follows by replacing $U_j^\wk$ with its upper bound $\|U_j^\wk\|_{C^0}$. 

We now have
\begin{multline*}
\int_{-T}^0 L^s_{c,i}(\flr{\xi}_i, \flr{\dot{\xi}}_i) dt - \int_{-T}^0 L^s_{c,i} (\flr{\gamma}_i, \flr{\dot{\gamma}}_i) dt \\
= \int_{-T}^0 L^s_{c,i}(\flr{\xi}_i, \flr{\dot{\xi}}_i) dt  - \int_{-T}^0 L^s_{c,i}(\zeta, \dot{\zeta}) dt 
 + \int_{-T}^0 L^s_{c,i}(\zeta, \dot{\zeta}) dt  - \int_{-T}^0 L^s_{c,i} (\flr{\gamma}_i, \flr{\dot{\gamma}}_i) dt \\
 \le \int_{-T}^0 L^s_{c,i}(\flr{\xi}_i, \flr{\dot{\xi}}_i) dt  - \int_{-T}^0 L^s_{c,i}(\zeta, \dot{\zeta}) dt,
\end{multline*}
noting that  $\zeta$ is minimizing for $L^s_{c,i}$. 

Since $\zeta$ is minimizing and hence extremal for $L^s_{c,i}$, from the definition of $\xi$ in \eqref{eq:xi-flr-i}, Lemma~\ref{lem:lin-drift} applies. Hence
\[
	\int_{-T}^0 L^s_{c,i}(\flr{\xi}_i, \flr{\dot{\xi}}_i) dt  - \int_{-T}^0 L^s_{c,i}(\zeta, \dot{\zeta}) dt \le 
	l \cdot h + \left(  \frac{1}{T} \tilde{z}_i^{-1} + T \|\sum_{j=i}^{d-m}U_j^\wk\|_{C^2} \right) h^2,
\]
where $l = \partial_{v_i} (L^s_{c,i})(\zeta(0), \dot{\zeta}(0))$. As in the proof of Lemma~\ref{lem:last-var}, we choose $ T = \left( \tilde{z}_i \sum_{j=i}^{d-m} \|U_j^\wk\|_{C^2} \right)^{- \frac12}$, we get 
\[
	\int_{-T}^0 L^s_{c,i}(\flr{\xi}_i, \flr{\dot{\xi}}_i) dt  - \int_{-T}^0 L^s_{c,i}(\zeta, \dot{\zeta}) dt \le 
	l \cdot h +  \sigma_i h^2, \quad \sigma_i = \left( \tilde{z}_i^{-1} \sum_{j=i}^{d-m} \|U_j^\wk\|_{C^2} \right)^{\frac12}. 
\]
Combine with \eqref{eq:lsci-comp}, and use the upper bound $ \sum_{j=i+1}^{d-m} \|U_j^\wk\|_{C^0}  \le \sum_{j=i}^{d-m}\|U_j^\wk\|_{C^2}$, we get 
\[
	\int_{-T}^0 (L^s - c\cdot v)(\xi, \dot{\xi}) dt - \int_{-T}^0 (L^s - c\cdot v)(\gamma, \dot{\gamma}) dt \le l \cdot h +  \sigma_i h^2 + 2\sigma_i. 
\]

\myheading{Estimating the weak KAM solution.} Combine the last formula with \eqref{eq:u-diff}, we get 
\[
	u(\varphi + h e_i^\wk + \chi) - u(\varphi) \le l \cdot h + \sigma_i h^2 + \sigma_i. 
\]
Since $\flr{\chi}_i =0$, using the inductive assumption, 
\[
	|u(\varphi + h e_i^\wk + \chi) - u(\varphi + h e_i^\wk) | \le  2\sum_{j=i+1}^{d-m} \delta_j. 
\]
Therefore
\[
	u(\varphi + h e_i^\wk) - u(\varphi) \le l \cdot h + \sigma_i h^2 + 2 \sigma_i + 2\sum_{j=i+1}^{d-m} \delta_j. 
\]
We now use Lemma~\ref{lem:app-semi-conc} to get for 
\[
	\delta_i = 2\sqrt{d}(3 \sigma_i + 2\sum_{j=i+1}^{d-m} \delta_j),
\]
$u$ is $(\delta_i, \delta_i)$ approximately Lipschitz in $\varphi^\wk_i$.
\end{proof}

\begin{proof}[Proof of Proposition~\ref{prop:appr-lip-weak}]
We have shown by induction that for all $1 \le i \le d-m$, $u$ is $(\delta_i, \delta_i)$ approximately Lipschitz in $\varphi_i^\wk$, where $\delta_i$ are defined inductively in Lemma~\ref{lem:last-var} and Proposition~\ref{prop:var-induction}. 

By Lemma~\ref{lem:z-j-bounds}, for each $1 \le i \le d-m$
\[
	 \sigma_i = (\tilde{z}_i^{-1}\|U_i^\wk\|_{C^2})^{\frac12} \le M^* |k_i^\wk|^{-\frac{q}{2} + i-m}. 
\]
Then $\delta_{d-m} = 2 \sigma_{d-m} \le M^* |k_{d-m}^\wk|^{-\frac{q}2 + d-m}$. For each $1 \le i \le d-m$, we have 
\[
	\delta_i  = \sqrt{d}(6\sigma_i + 4 \sum_{j=i+1}^{d-m} \delta_i) \le (6\sqrt{d})^{i-m} \sum_{j=i}^{d-m} \sigma_i \le M^* (6\sqrt{d})^{i-m} |k_i^\wk|^{-\frac{q}2 + d-m} .
\]
For any $\varphi^\wk, \psi^\wk \in \T^{d-m}$ and $\varphi^\st \in \T^m$, 
\[
	|u(\varphi^\st, \varphi^\wk) - u(\varphi^\st, \psi^\wk)| \le \sum_{i=1}^{d-m}\delta_i|\varphi^\wk_i - \psi^\wk_i| + \sum_{i=1}^{d-m}\delta_i 
\]
Since $\sum_{i=1}^{d-m}\delta_i \le (d-m)M^* (6\sqrt{d})^{i-m} (\lM(\cB^\wk))^{-\frac{q}2 + d-m}$, the proposition follows by replacing $M^*$ by $(d-m)M^* (6\sqrt{d})^{i-m} $. 
\end{proof}

\appendix

\section{Diffusion path with dominant structure}
\label{sec:diff-path}

\subsection{Diffusion path for Arnold diffusion}
\label{sec:intro-path}

Our main motivation is to prove Arnold diffusion for a ``typical'' nearly integrable system of the form (\ref{pert-hamiltonian}). The word ``typical'' here means the cusp residual condition introduce by Mather (\cite{Mather03}). 

\begin{definition}
For $r\ge 3$, we say that a property $\cG$ hold for a cusp residual set of $C^r$ nearly integrable systems $H_\varepsilon = H_0 + \varepsilon H_1$, if:
\begin{itemize}
\item $\cG$ is an open property in $C^r$ topology;
\item There exists an open and dense set $\cV \subset \{ \|H_1\|_{C^r} =1 \}$, and a positive 
function $\varepsilon_0: \cV \to \R^+$, such that $\cG$ is $C^r$-dense on 
$\cU = \{H_0 + \varepsilon H_1: H_1 \in \cV, 0 < \varepsilon < \varepsilon_0(H_1)\}$. 
\end{itemize}
\end{definition}

We would like to show that the property of topological instability is cusp residual. 
Instabilities for multidimensional Hamiltonian systems ($n\ge 3$) are studied in 
\cite{Moe96,GK12,CY09,BKZ11,GK12,KZ14,DLS13,Tre04,Tre12,Zhe, Mar1, Mar2,KS14}. 


The main conjecture of Arnold diffusion in finite regularity may be formulated as follows.
\begin{conjecture}
There exists $r_0 >0$ such that for each $n \ge 2$, $\gm>0$, $r_0 \le r<\infty$,  
for a cusp residual set of $C^r$ nearly integrable system, the system 
admits an orbit $(\theta_\varepsilon, p_\varepsilon)(t)$ such that $\{p_\varepsilon(t)\}_{t\in \R}$ is $\gamma-$dense on the unit ball 
$B^n := \{ \|p\| \le 1\}$. 
\end{conjecture}

The conjecture is a theorem for $n=2$, we refer the reader to \cite{Ch13,KZ13} and reference therein.
The proof in $n=2$ follows two steps: 

\emph{Step 1},  define the set $\cV$, which contains the set of ``nondegenerate'' $H_1$. 
For $H_1 \in \cV$,  $H_0 + \varepsilon H_1$ possesses certain open structure of instability, 
such as NHICs and the AM property mentioned below.  

\emph{Step 2},  show that for any $H_0 + \varepsilon H_1$ with $H_1 \in \cV$ and $\varepsilon$ sufficiently small, one can make an arbitrarily small perturbation to $H_0 + \varepsilon H_1$ such that there exists diffusion orbits. 

In Theorem~\ref{thm:am-property}, we prove a weaker version of \emph{Step 1}. The heart of the argument is the construction of a diffusion path, on which all the essential resonances has a dominant structure. We expect the same diffusion path can be used to prove the full conjecture. To avoid excessive length, we will give an outline of the proof with key statements, and the full details will appear later.

A diffusion path $\Path$ is a subset in $\R^n$ that the diffusion orbit $p_\varepsilon(t)$ roughly shadows. We pick a diffusion path that travels along a collection of $(n-1)-$resonances or, equivalently, along a collection 
of connected 1-dimensional resonant curves. 
\begin{definition}
A diffusion path $\Path$ is a compact connected subset of 
\[
\bigcup \{ \Gamma_{\Lbi{n-1}}: \quad \Lbi{n-1} \in \cL^{(n-1)} \}	
\]
where $\cL^{(n-1)}:= \{\Lbi{n-1}_i\}_{i=1}^N$ 
 is a collection of rank $n-1$ irreducible resonant lattices (and each $\Gamma_{\Lbi{n-1}}$ is a 
1-dimensional resonant curve) . \footnote{A remark on notation: the supscript ${}^{(n-1)}$ is not used as an index, but rather an indication for the rank of the lattice. } 
\end{definition}

We define the AM property of a mechanical system relative to an integer homology class. 
\begin{itemize}
\item Let $H=K-U$ be a mechanical system on $\T^n \times \R^n$, 
\item $h$ be an integer homology class, 
\item $S_E=\{H=E\}$ be energy surface.
\item $\min U =0$, and the minimum is unique.
\end{itemize}
Denote by $\pi:\T^n\times \R^n\times \T \to \R^n$ 
the natural projection onto the action component. 

Recall that homology and cohomology are related by Legendre-Fenichel
tranform $\cL\cF_\beta(h)\subset H^1(\T^n,\R)$ (see (\ref{LF-transform})). 
By a result of Diaz Carneiro \cite{DC} for each cohomology 
$c\in H^1(\T^n,\R)$ the Aubry set $\cA(c)\subset S_{\al(c)}$. 

\begin{definition}
Let $\rho>0$. We say that $(H,h,\rho)$ has the AM  property if 
for any $\lb$ such that $c\in \cL\cF_\beta(\lb h)$ and $\al_H(c)\ge \rho$
the  Aubry set $\cA(c)$ is a finite union of hyperbolic periodic orbits 
such that each of these periodic orbits as a closed curve has 
homology $h$. 
\end{definition}

\begin{remark}
Note that in the definition we do not consider the energy $0 \le E < \rho$. 

The AM property is far from being generic. In Section 
\ref{sec:diffusion-amproperty} we discuss variety of 
ways the AM property can fail for an open class of systems. 
\end{remark}

Let $\Lbi{n-1}\subset \Sgi{n}$ be two irreducible lattices 
of rank $n-1$ and $n$ respectively. Let $\cBe{n}=[l_1,\dots,l_n]$
be an ordered basis of $\Sgi{n}$ and $\cBi{n-1} = [k_1,\dots,k_{n-1}]$ is 
be an ordered basis of $\Lbi{n-1}$. Then $\Lbi{n-1}$ and $\Sgi{n}$
induce a (unique up to a sign) irreducible integer homology class denoted 
$h(\cBi{n-1}, \cBe{n})\in  \Z^n \simeq H_1(\T^n, \Z)$ 
(see (\ref{induced-homology}) for details).

We now state the main theorem of this section. 
\bthm  \label{thm:am-property} There exists $r_0 >0, C>0$ such that 
for each $n\ge 2,\, \rho>0, \, r_0\le r<\infty$, for an open and dense set of   
$H_1$ from  $\{\|H_1\|_{C^r}=1\}\subset C^r(\T^n \times B^n\times \T)$,
there exists a diffusion path $\Path=\Path(H_1,\rho)$ with a
a finite set $\mathcal{E}$ called the punctures or strong resonances, with the following properties. 
\begin{enumerate}
\item $\Path$ is $\rho$-dense in $B^n$, i.e. $\rho$-neighborhood 
of $\Path$ contains $B^n$. 

\item For each $1$-dimensional resonant curve $\Gm_i\subset \Gm_{\Lbi{n-1}}\cap \Path$ 
there is a $3$-dimensional NHWIC $\widetilde \cC^{\,3}_i$ whose projection onto the action 
component dist$(\pi \widetilde \cC^{\,3}_i,\Gm_i)\le C\sqrt \eps$.
\item (Away from strong resonances)
For each $c\in \Gm_i$ with dist$(c,\Sigma_n)\ge C\sqrt \eps$, we have 
$\cA(c)$ belongs to $\widetilde \cC^{\,3}_i$.

\item (At strong resonance) Each puncture $p_0\in \mathcal{E}$ is given by a rank $n$ 
irreducible lattice $\Sgi{n}$, i.e. $\{p_0\} =  \Gm_{\Sgi{n}}$.

\item  Let $p_0\in \Path\cap \Gm_{\Sgi{n}}$ be a puncture. Then 
$p_0\subset \Gm_{\Lbi{n-1}}\cap \Path \ne p_0$ for some 
rank $n-1$ irreducible lattice $\Lbi{n-1}$, with bases $\cBi{n-1}$ and $\cBe{n}$.  For the induced homology 
$h=h(\cBi{n-1}|\cBe{n})$ and the slow mechanical system 
$H=H_{p_0,\cBe{n}}$, defined in (\ref{eq:slow-system}), we have 
that $(H,h,\rho)$ have AM property. 
\end{enumerate} 
\ethm

We have the following remarks.
\begin{itemize}
\item If $n=2$, a stronger version of Theorem~\ref{thm:am-property} hold. Namely, one can prove that for an \emph{fixed} diffusion path, there exists a cusp residue set of systems $H_0 + \varepsilon H_1$ for which the theorem hold. Whether this statement generalizes to higher degrees of freedom is an open question. 

In our formulation, it is essential that  
the choise of diffusion path $\Path$ \emph{does depend} on 
the perturbation $\eps H_1$. 

\item Item 3 says that $3$-dimensional cylinders 
$\widetilde \cC^{\,3}_i$ are minimal in the sense that 
they contain the Aubry sets with frequency vector from 
$\Gm_i$ away from maximal essential resonances.  

\item It turn out that away from strong resonances for each $h\in \Gm_i$ with 
dist$(h,\Sigma_n)\ge C\sqrt \eps$ and $c\in \cL\cF_\beta(h)$ we have 
not only that $\cA(c)$ belongs to $\widetilde \cC^{\,3}_i$, but also it is 
a Lipschitz graph over a certain $2$-torus $\T^2_c$, i.e. for some submersion 
$\pi_c:\T^n\times B^n\times \T \to \T^2$ we  have that 
$\pi_{\cA(c)}:\cA(c)\to \T^2_c$ 
is one-to-one and the inverse is Lischitz. This is similar but more involved than 
what is presented in \cite{BKZ11}. See discussion of $n=3$ in \cite{KZ14}.  

\item 
$3$-dimensional cylinders $\widetilde \cC^{\,3}_i$
for $H_\eps$ correspond to $2$-dimensional cylinders $\cC^2$
for averaged Hamiltonians. 

\item The cylinder $\widetilde \cC^{\,3}_i$ might 
consists of several connected components. At each maximal essential 
resonance $\Gm_{\Lbi{n}}$ this cylinder can have two connected 
components: one on each local component of 
$ \Gm_{\Lbi{n-1}}\setminus \Gm_{\Lbi{n}}$.


\item The union of hyperbolic periodic orbits gives rise to a NHIC. 

\item In section \ref{sec:diffusion-amproperty} we discuss the role 
of AM property for proving diffusion as well as the number of ways it 
can be violated. 

\item Notice that at each strong resonance, due to our definition of AM property,
 we do not discuss the case low energy $0\le E < \rho$.
This is why Theorem~\ref{thm:am-property} does not complete Step 1. For $n=2$ 
a full description can be done, see \cite{KZ13,Ch13} and references therein. For $n=3$,  
construction of NHIC for away from critical energy in general and normally hyperbolic 
invariant manifolds (NHIM) for critical energy for simple homologies 
is discussed in \cite{KZ14}, sect. 6.3. We expect these 
methods extend to arbitrary $n\ge 3$ (see also \cite{Tu14}). 

\item We point out that presence of NHIC and NHIM is still not 
sufficient for diffusion as we need to construct the jump from one homology
to another (see sect. 12 \cite{KZ13}). In the case $n=3$   it requires a lot 
more work (see sect. 8 \cite{KZ14}). We expect to generalize this construction 
of the jump from \cite{KZ14} to any $n\ge 3$. 
\end{itemize}

\subsection{Nondegeneracy conditions for Arnold diffusion}
\label{sec:nondeg}

We now describe the set $\cV$ in Theorem~\ref{thm:am-property}, using the conditions [H1] and [H2] to be defined later. Let $\rho>0$ and $r_0 \le r<\infty$. 
We say that $H_1 \in \cV$ if $\|H_1\|_{C^r} = 1$, and 
there \emph{exists} a diffusion path $\Path$
that is $\rho-$dense in $B^n$, with the following properties.
\begin{itemize}
\item For each $\Lbi{n-1} \in \cL^{(n-1)}$, and each connected component $\Gamma$ of $\Path \cap \Gamma_{\Lbi{n-1}}$,  there exists $\lambda>0$ such that  function $H_1$ satisfies condition [H1$\lambda$] on $\Gamma$. 

\item For each $\lambda>0$ and $\Lbi{n-1} \in \cL^{{n-1}}$, there exists a finite set of rank $n$ resonant lattices $\Ess(\Lbi{n-1}, \lambda)$, with the property $\Sgi{n} \supset \Lbi{n-1}$ for each $\Sgi{n} \in \Ess(\Lbi{n-1}, \lambda)$. Then   $\Gamma_{\Sgi{n}}$ is a single point contained in  $\Gamma_{\Lbi{n-1}}$. The collection $\mathcal{E} =\{\Gamma_{\Sgi{n}}: \Sgi{n} \in \Ess(\Lbi{n-1}, \lambda)\}$ is 
the set of punctures in Theorem~\ref{thm:am-property}. 

\item Let $\lambda>0$ be such that [H1$\lambda$] is satisfied for $H_1$. 
For each $\Sgi{n} \in \Ess(\Lbi{n-1}, \lambda)$ and  $\Lbi{n-1} \in\cL^{(n-1)}$ 
such that $\Gamma_{\Sgi{n}} \in \Path$, we choose basis $\cBe{n}$ and $\cBi{n-1}$. 
We say that $H_1$ satisfies condition [H2] at $\Gamma_{\Sgi{n}}$ if for all such  $\Lbi{n-1} \subset \Sgi{n}$, 
\[
		(H_{p_0, \cBi{n}}^s, h(\cBi{n-1}|\cBi{n}), \rho)
\]
satisfies the AM property. 
\end{itemize}
The condition [H1$\lambda$], and the definition of $\Ess(\Lbi{n-1}, \lambda)$ and $h(\cBi{n-1}|\cBi{n})$ will be explained below. For the moment we only remark that for a fixed $\Path$, the condition that [H1$\lambda$] holds for some $\lambda >0$ is open and dense; for a fixed $\Gamma_{\Sgi{n}}$, the AM property is open but not always dense. However, it is a dense condition if the lattice $\Sgi{n}$ satisfies a domination property. The main idea is then, to pick a particular
$H_1$-dependent diffusion path $\Path$, such that all the essential resonances on this path has this domination property.

\subsubsection*{The condition [H1]}

We now describe our first set of non-degeneracy condition. For 
$\Lbi{n-1} \in \cL^{(n-1)}$, let us fix a basis $\cB$. For $\lambda>0$ and 
a connected compact subset $\Gai{n-1} \subset \Gamma_{\Lbi{n-1}}$,  
we say that $H_1$ satisfies condition [H1$\lambda$] on $\Gai{n-1}$ if 
\begin{itemize}
\item For all $p \in \Gai{n-1}$, the function $Z_\cB(\cdot, p)$ has at most two global maxima. 
\item At each global maxima $\varphi^*$ of $Z_{\cB}(\cdot, p)$, the Hessian $\partial^2_{\varphi \varphi}Z_\cB(\varphi^*, p)\le - \lambda \Id$ as quadratic forms. 
\item Suppose $p_0$ is such that there are two global maxima $\varphi^*_1(p_0)$ and $\varphi^*_2(p_0)$. Then they extend to local 
maxima for nearby $p \in \Gai{n-1}$. We assume that 
the functions $Z_\cB(\varphi^*_1(p), p)$ and $Z_\cB(\varphi^*_2(p), p)$ 
have different derivatives along $\Gai{n-1}$, with the difference at least 
$\lambda$. 
\end{itemize}
We say that $H_1$ satisfies [H1] on $\Gai{n-1}$ if it satisfies [H1$\lambda$] for some $\lambda>0$. 
These conditions are introduced by Mather (\cite{Mather03}) for $n=2$ and assumed in \cite{BKZ11}. 
We note that the quantitative version [H1$\lambda$] of the condition depends on the choice of basis, 
while the qualitative version [H1] does not. 

For $H_1$ satisfying [H1$\lambda$], there exists a finite set of rank $n$ lattices containing $\Lbi{n-1}$,  which we will call $\Ess(\Lbi{n-1}, \lambda)$. More precisely, assume that the basis for $\Lbi{n-1}$ is $\{k_1, \cdots, k_{n-1}\}$ and there exists $M= M(\Lbi{n-1}, \lambda)>0$ such that 
\[
	\Ess(\Lbi{n-1}, \lambda) = \{ \Lambda_{k_1, \cdots, k_{n-1}, k'}: \quad |k'| \le M \}. 
\]
For each $\Lbi{n} \in \Ess(\Lbi{n-1}, \lambda)$,  $\Gamma_{\Lbi{n}}$ is 
a point contained in $1$-dimensional curve $\Gamma_{\Lbi{n-1}}$. 
The condition [H1] implies the existence of NHIC away from punctures, see \cite{BKZ11}. It is not hard to see that item 1-5 of Theorem~\ref{thm:am-property} are direct consequences of our non-degeneracy conditions. The difficulty in Theorem~\ref{thm:am-property} is in showing these conditions are open and dense. 

\subsubsection*{Induce homology and non-degeneracy}

Fix $\Lbi{n} \in \Ess(\Lbi{n-1}, \lambda)$, and let $\cBi{n}$ be an ordered 
bases of $\Lbi{n}$, $\cBi{n-1}$ is an ordered basis of $\Lbi{n-1}$ and 
$\{p_0\} = \Gamma_{\Lbi{n}}$. Our second set of non-degeneracy condition 
concerns the slow system $H^s_{p_0, \cBi{n}} : \T^n \times \R^n \to \R$, for 
a particular integer homology class $h(\cBi{n-1}|\cBi{n}) \in H^1(\T^n, \Z)$, 
uniquely defined modulo the sign. We give a more general definition here. 
\begin{definition}
For $2 \le s \le n$, irreducible lattices $\Lbi{s-1} \subset \Lbi{s}$, with 
corresponding basis $\cBi{s-1} = [k_1, \cdots, k_{s-1}]$ and 
$\cBi{s} = [l_1, \cdots, l_s]$. Since $k_i \in \Lbi{s}$, there exists a unique 
collection $a_i \in \Z^s\setminus 0, \ i=1,\dots,s-1,$ such that 
\[
	k_i = \bmat{l_1 &  \cdots & l_s} a_i.
\]
Then $h(\cBi{s-1}|\cBi{s}) \in \Z^s$ is defined by the relations
\be \label{induced-homology}
	a_i \cdot h(\cBi{s-1}|\cBi{s}) =0, \quad 1 \le i \le s-1. 
\ee
\end{definition}
This definition is  determined by the resonance relation 
\[
	k_i \cdot (\omega(p),1) = 0, \quad 1 \le i \le s-1, 
\]
after converting to the basis $[l_1, \cdots, l_s]$.

We  require the triplet
\[
	(H_{p_0, \cBi{n}}^s, h(\cBi{n-1}|\cBi{n}), \rho)
\]
satisfies the AM property. We have the following consequences of the AM property:
\begin{itemize}
\item {\it (Robustness)} The non-degeneracy condition is open. 

\item {\it (Minimality)} The condition guarantees, among other things, 
existence of an ordered collection of minimal \emph{2-dimensional} 
NHICs with heteroclinic connections of neighbors. Each cylinder is 
minimal in the sense that it is foliated 
by periodic orbits minimizing action of a certain variational problem.

\item {\it (Hyperbolicity)} Each cylinder is hyperbolic in the sense that it 
consists of hyperbolic periodic orbits.
\end{itemize}

\subsection{Properties of the nondegeneracy condition}
\label{sec:prop-nondeg}

Suppose $H^s(\cB^\st, \cB^\wk, p, U^\st, \cU^\wk)$ is a dominant system. Then the AM property extends nicely from the strong system to the slow system. More precisely, the following properties hold. 

\myheading{Property A0.} For $H = K - U$, the AM property for $(H, h,\rho)$ is 
an open condition in both $K$ and $U$. 

\myheading{Property A1.} (Genericity in $2$-degrees of freedom) For a fixed quadratic form $K$ on $\R^2$ and  $h \in \Z^2$, there exists an open and dense set of $U \in C^2(\T^2)$ on which 
$(H = K-U, h,\rho)$ have AM property.

\myheading{Property A2.} (Dimension reduction using hyperbolic fixed point) Consider the data $(\cB^\st, Q_0, \kappa, q,\rho)$ 
and the space of corresponding dominant system 
$\Omega^{m, m+2}_{\kappa, q}(\cB^\st)$. Assume that  
$U^\st_0 \in \T^m$ admits at most two non-degenerate minima. Note that each 
corresponds to a hyperbolic fixed point of $H^\st$. 

Then there exists $M>0$ and $\delta>0$ depending only on 
$\cB^\st, Q_0, \kappa, q, p_0, U_0^\st,\rho$ such that the following hold. 
For each $\cB^\wk$ with 
$$\lM(\cB^\wk) > M, \ \|p-p_0\|<\delta, \ \|U^\st- U^\st_0\|_{C^2}<\delta,
\  \text{and }h \in \Z^2,
$$ 
for an \emph{open and dense set} of $\cU^\wk$ (in the space 
$\Omega^{m, m+2}_{\kappa, q}(\cB^\st)$ restricted to fixed $\cB^\wk, p, U^\st$), the triple
\[
	\left( \cH^s(\cB^\st,\cB^\wk, p, U^\st, \cU^\wk), g,\rho \right), \quad \text{ with } g = (0, \cdots, 0, h),
\]
have AM property.

\myheading{Property A3.} (Dimension reduction using AM property) Consider the data $(\cB^\st, Q_0, \kappa, q,\rho)$ 
and the space of corresponding dominant system 
$\Omega^{m, m+1}_{\kappa, q}(\cB^\st)$. Assume that 
$p_0 \in \R^n$, $U_0^\st \in C^r(\T^m)$, $h \in \Z^m$ satisfies 
\[
	\left( \cH^\st(p_0, U^\st_0), h,\rho \right) \text{ have AM property.}
\]

Then there exists $M> \sup_{k \in \cB^\st}|k|$, $\delta>0$ depending only on $\cB^\st, Q_0, \kappa, q, p_0, U_0^\st, h,\rho$ such that the following hold. For each $\cB^\wk$ with 
$$
\lM(\cB^\wk) > M, \ \|p-p_0\|<\delta, \ \|U^\st- U^\st_0\|_{C^2}<\delta,
$$
and any nonzero pair of integers $z, w$, 
the following hold. 

For an \emph{open and dense set} of $\cU^\wk$ (in the  space 
$\Omega^{m, m+1}_{\kappa, q}(\cB^\st)$ restricted to a fixed set 
of $\cB^\st, p, U^\st$), the triple
\[
	\left(  \cH^s(\cB^\st,\cB^\wk, p, U^\st, \cU^\wk), g_0,\rho \right), \text{ where } g_0 = (zh, w), \text{ has AM property.}
\]

\myheading{Remarks}:
\begin{enumerate}
\item The proof of Theorem~\ref{thm:am-property} uses only Properties A0-A3 instead of the precise definition of AM property. Therefore the proof applies if we take Properties A0-A3 as ansatz. We expect that the properties required for the full diffusion problem satisfy the same ansatz and  our construction applies to the full diffusion problem. 

\item The list of properties A0 - A3 provides a setup for proving non-degeneracy using induction over degrees of freedom. Assume that $U^\st$ admits a non-degenerate minimum, then property A2 allows to extend this system by two more degrees of freedom, provided the homology $g$ is only nontrivial in the weak variables. If $H^\st$ is nondegenerate in a nontrivial homology $h$, property A3 allows to extend by one degree of freedom, provided the new homology $g$ is trivial in the weak variable. 
\item Let us explain the proof briefly. Property A1 is a known result. This property is used in Arnold diffusion in $2\frac12$ degrees of freedom, and we refer to \cite{Mather10}, \cite{Mather11}, \cite{KZ13}, \cite{Ch13} for more details. 

\item Property A2 uses the first type of dimension reduction. 
The assumption ensures that $H^\st$ admits at most two minimal hyperbolic saddles. 
An arbitrarily small perturbation ensures that only one of them is minimal. Using 
Theorem~\ref{thm:nhic-persist}, one obtain that $H^s$ admits a minimal \emph{four-dimensional} 
NHWIC $\cC^4$. Furthermore, Theorem~\ref{thm:semi-cont}  and 
Proposition~\ref{prop:aubry-mane} provide variational characterization for 
the cylinder. Then the restricted system to $\cC^4$ behaves like a system 
with \emph{two degrees of freedom}, and an analog of property A1 can 
be proven. In particular, there will be an ordered collection of minimal 
two-dimensional NHIC's contained in $\cC^4$. 

\item For property A3,  when the triple $(H^\st, h, \rho)$ satisfies the AM property, 
the strong system admits a family of \emph{two dimensional} NHICs. Because 
there is only one weak component, Theorem~\ref{thm:nhic-persist} implies 
that $H^s$ admits a minimal \emph{four-dimensional} NHWIC. Similar
to the previous case, the idea from property A1 can be applied to prove nondegeneracy. 

\item One can say that in the case A2 or A3, the slow system $H^s$ is ``dominated'' by the strong system $H^\st$. 

\end{enumerate}

\subsection{Construction of a diffusion path  and surgery of 
resonant manifolds}
\label{sec:const-path}

To prove Theorem~\ref{thm:am-property}, it remains to construct a diffusion path with our non-degeneracy conditions. 
\begin{proposition}
\label{prop:diff-path}
For each $\gamma>0$, there exists an open and dense set $\cV \subset \{\|H_1\|_{C^r}=1\}$, such that for any $H_1 \in \cV$, there exists a $\gamma-$dense diffusion path $\Path$, such that the non-degeneracy conditions [H1] and [H2] are satisfied along $\Path$.
\end{proposition}

The proof of Proposition~\ref{prop:diff-path} occupies the rest of this section. Since our nondegeneracy conditions are assumed to be open, it suffices to prove density. We fix an arbitrary relative open set $\cU_0 \subset \{\|H_1\|_{C^r}=1\}$, we will show there exists $H_1 \in \cU_0$ such that the conclusions hold. 
The proof follows an inductive scheme. The strategy is as follows:
\begin{enumerate}
\item At step $s$ we have   a finite collection of integer irreducible lattices 
$\cLi{s} = \{\Lbi{s}_i\}$, i.e. each 
$\Lbi{s}_i := \Span_\Z\{k_1^i, \cdots, k_s^i\}$ has rank $s$ 
and is spanned by integer vectors $k_1^i, \cdots, k_s^i $. 
The union of 
corresponding codimension $s$ resonant manifolds ($\Gamma_{\Lbi{s}}$) is called $\cPi{s}$. 

The lattices $\Lbi{s}$ has a hierarchical structure in the sense that there is an unique element $\Lbi{s-1} \in \cLi{s-1}$ such that $\Lbi{s-1}\subset \Lbi{s}$. As a result,  $\cPi{s} \subset\cPi{s-1}$, we will choose  $\cLi{s}$ such that $\cPi{s}$ is $(1 -2 \cdot 4^{-s})\gamma-$dense subset in $\cPi{s-1}$ and $(1-4^{-s})\gamma-$dense in $B^n$. 

\item A set of essential resonances $\cEi{s+1}$ is a collection of irreducible 
lattices of rank $s+1$. Each element $\Sgi{s+1} \in \cEi{s+1}$ contains at lease one element $\Lbi{s} \in \cLi{s}$. Roughly speaking, the essential lattices are the collection of lattices $\Sgi{s+1}$ that contain but does \emph{not} dominate some $\Lbi{s} \in \cLi{s}$.  


Essential resonances $\Sgi{s+1}$ correspond to a codimension $s+1$ 
resonant manifold $\Gamma_{\Sgi{s+1}}$ contained in  $\cPi{s}$. 

\item A nondegenerate set $\cN_s \subset \cPi{s}$, which is open, 
connected and $(1-2\cdot 4^{-s})\gamma-$dense in $B^n$, such that the following hold:
\begin{enumerate}
 \item For all $\Lbi{s} \in \cLi{s}$ with basis $\cBi{s} = [k_1, \cdots, k_s]$, and $p_0 \in \cN_s \cap \Gamma_{\cLi{s}}$, the slow system 
 \[
 H^s_{p_0, \cBi{s}}
 \]
 is nondegenerate in the sense of Property A2. 
 \item For all essential lattice $\Sgi{s+1} \in \cEi{s+1}$ with basis $\cBe{s+1}$, and $p_0 \in \cN_{s}\cap \Gamma_{\Sgi{s+1}}$, and every $\Lbi{s} \in \cLi{s}$ with $\Lbi{s} \subset \Sgi{s+1}$ and basis $\cBi{s}$, the triple
 \[
 (H^s_{p_0, \cBe{s+1}}, h(\cBi{s}, \cBe{s+1})), \rho)
 \]
 satisfies the AM property, setting up to apply Property A3. 
\end{enumerate}
\item The next generation of resonant lattices are carefully defined so that we can use Property A2, A3 to extend the non-degeneracy in item 3 to the next generation. 
\item The induction finishes at step $n-1$, when we obtain an open, connected and $\gamma-$dense set $\cN_{n-1} \subset \cPi{n-1}$ 
in $B^n$ which 
consists of $1$-dimensional resonant manifolds and will be our diffusion path, and all essential resonances have the AM property.  
\end{enumerate}

\subsubsection{An initial step of the induction}

Since the union of all $1-$resonant manifolds are dense and locally connected, for each $\gamma>0$ we can pick $\cLi{1} = \{\Lbi{1}_i\}$ 
such that the set 
\[
	\cPi{1} = \bigcup_{\Lbi{1} \in \cLi{1}} \Gamma_{\cLi{1}}\cap B^n
\]
 is connected and $\gamma/2-$dense in $B^n$. For each lattice 
$\cLi{1}$ denote its basis by $\cBi{1}=[k_1]$,
i.e. $\cLi{1}=\Span_\R(k_1)\cap \Z^{n+1}$. Denote by 
$\cKi{1}=\{\cBi{1}_i\}$ the union of basis vectors.

For any two sets $E_1, E_2 \subset \Z^{n+1}$, we define 
\[
	E_1 \vee E_2 = \Span_\R \{E_1 \cup E_2\} \cap \Z^{n+1}
\]
to be the smallest irreducible lattice containing $E_1$ and $E_2$.

We define a first non-degeneracy set $\cY_1^\lambda(H_1, \cPi{1}) \subset \cPi{1}$ by the following condition: For any $\Lbi{1} \in \cLi{1}$ with 
basis $\cBi{1}$ and  $p \in \Gamma_{\Lbi{1}} \cap \cPi{1}$, 
the averaged potential $U_{p, \cBi{1}}$ 
has at most two $\lambda-$nondegenerate minima.

\begin{figure}[t]
\centering
\def\svgwidth{2.8in}
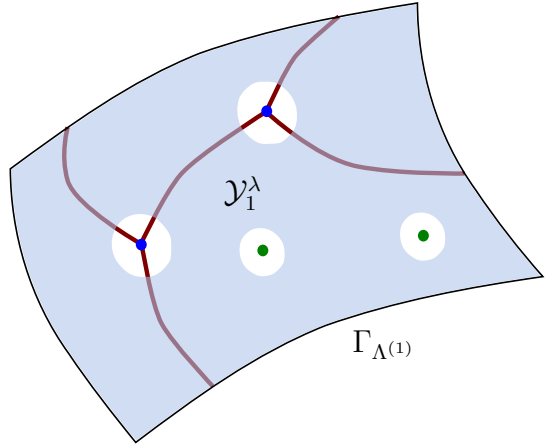
\caption{The first nondegeneracy set for $n=3$: the shaded part is the $\lambda-$nondegenerate set. On the bold lines  there are two minima for $U_{p, \cBi{1}}$, at the blue dots are there are three minima. At the green dots the minimum is degenerate. }
\label{fig:first-nondeg}
\end{figure}

\begin{lemma}
There exists a relative open set $\ \cU_1 \subset \cU_0$ and $\lambda_1>0$, such that for $H_1 \in \cU_1$, the nondegeneracy set $\cY_1^{\lambda_1}(H_1, \cPi{1})$ is open, connected, $4^{-2}\gamma-$dense in $\cPi{1}$, and 
$(1-2\cdot 4^{-1}\gamma)-$dense in $B^n$. 
\end{lemma}
The set $\cY_1^{\lambda_1}(H_1, \cPi{1})$ is the shaded set 
on Figure \ref{fig:first-nondeg}.  For brevity in what follows we often 
omit dependence of $\cY_1^{\lambda_1}$ on $H_1$ and $\cPi{1}$.

For each $p_0 \in \Gamma_{\Lbi{1}}\cap \cY_1^{\lambda_1}$,
where $\Lbi{1}$'s basis is $\cBi{1}$, the assumption of Property A2 
is satisfied for $\cB^\st = \cBi{1}$, $p_0$ and $U^\st = U_{p_0, \cBi{1}}$. Moreover, using compactness, for all 
\[
	m=1,\quad \Lbi{1} \in \cLi{1}, \quad p_0 \in \Gamma_{\Lbi{1}}\cap \cY_1^{\lambda_1}, \quad U^\st = U_{p_0, \cBi{1}},
\]
there exists a \emph{uniform} $M_1 = M_1(\cPi{1}, \cY_1^{\lambda_1})$, 
such that for  all $\cB^\wk = [k_1^\wk, k_2^\wk]$ with 
$\lM(\cB^\wk)>M_1$  the conclusion of Property A2 is satisfied. We assume that $M_1$ is chosen such that 
\[
	M_1(\cPi{1}, \cY_1^{\lambda_1}) > \max_{k_1 \in \cKi{1}}|k_1|. 
\]

We define the first generation of \emph{essential lattices} 
$\cEi{2} = \cEi{2}(\cPi{1}, \cY_1^{\lambda_1})$ as the set of all rank $2$ irreducible lattices $\Sgi{2}$ satisfying the following conditions: there exists $\Lbi{1} \in \cLi{1}$ such that 
\[
	\Sgi{2} \subset \Lbi{1}, \quad M(\Sgi{2}|\Lbi{1}) \le M_1(\cPi{1}, \cY_1^{\lambda_1}). 
\]
The requirement $M_1(\cPi{1}, \cY_1^{\lambda_1}) > \max_{k_1 \in \cKi{1}}|k_1|$  ensures that for any $\Lbi{1}_1, \Lbi{1}_2 \in \cLi{1}$, the lattice $\Lbi{1}_1\vee \Lbi{1}_2$ is automatically essential.  This corresponds to the intersection of  $\Gamma_{\Lbi{1}_1}$ and $\Gamma_{\Lbi{1}_2}$. 

 The essential lattice set contain all lattices that does not ``dominate'' 
the lattices in $\cLi{1}$.  Let us also denote 
\[
	\Gamma_{\cEi{2}} = \bigcup_{\Sgi{2} \in \cEi{2}} \Gamma_{\Sgi{2}}
\]
the union of all resonance manifolds corresponding to the essential lattices.

For each essential lattice $\Sgi{2}$ we fix an ordered basis $\cBe{2}$ 
(the actual choice is irrelevant). We 
define a second nondegeneracy set $\cZ_1(H_1, \Sgi{2}, \cY_1^{\lambda_1})$ to be the set of $p \in \Gamma_{\Sgi{2}} \cap \cY_1^{\lambda_1}$ such that for each $k_1\in \Sgi{2}$ with $\cBi{1} =[k_1]\in \cKi{1}$, the pair
 \[
 	\left( H_{p, \cBe{2}}^s, h(\cBi{1}|\cBe{2}) , \rho \right) 
 \]
 satisfies the AM property. We then define $\cZ_1(H_1, \cEi{2}, \cY_1^{\lambda_1})$ 
to be the union of all $\cZ_1(H_1, \Sgi{2}, \cY_1^{\lambda_1})$ over 
essential resonances $\Sgi{2} \in \cEi{2}$.

%
%

Because $H_{p, \cBe{2}}^s$ has two degrees of freedom, we can use Property A0 to obtain the following lemma. 
\begin{lemma}
There exists a relative open  set $\cU_1' \subset \cU_1$ and 
a relative open $\tilde{\cZ}_1 \subset \Gamma_{\cEi{2}}  \cap \cY_1^{\lambda_1}$ such that the following hold. 
\begin{enumerate}
\item For all  $H_1 \in \cU_1'$, $\tilde{\cZ}_1$ is  compactly 
contained in $\cZ_1(H_1, \cEi{2}, \cY_1^{\lambda_1})$. 
\item  The set 
\[
	\cN_1 :=   \cY_1^{\lambda_1} \cap 
\tilde{\cZ}_1 
\]
is open, connected, and $(1-4^{-1})\gamma-$dense in $\cPi{1}$.
\end{enumerate}
\end{lemma}

We choose $\tilde{\cZ}_1$ compactly contained in $\cZ_1$ so that the nondegeneracy on $\tilde{\cZ}_1$ is \emph{uniform} due to compactness. 
The idea behind the definition of $\cN_1$ is the following: On the set 
of essential resonances $ \Gamma_{\cEi{2}}$, domination does not apply, 
so we should remove it from the nondegeneracy set $\cY_1^{\lambda_1}$. However, in this case the remaining set becomes \emph{disconnected} 
because the essential resonances divide the space (see 
Figure \ref{fig:second-nondeg}, left). Instead we only remove 
only $p$'s with the nearly degenerate essential resonances, i.e. 
$\cY_1^{\lambda_1} \cap \Gamma_{\cEi{2}} \setminus  \tilde{\cZ}_1$
(see Figure \ref{fig:second-nondeg}, right dashed line).

\begin{figure}[t]
\centering
\def\svgwidth{5in}
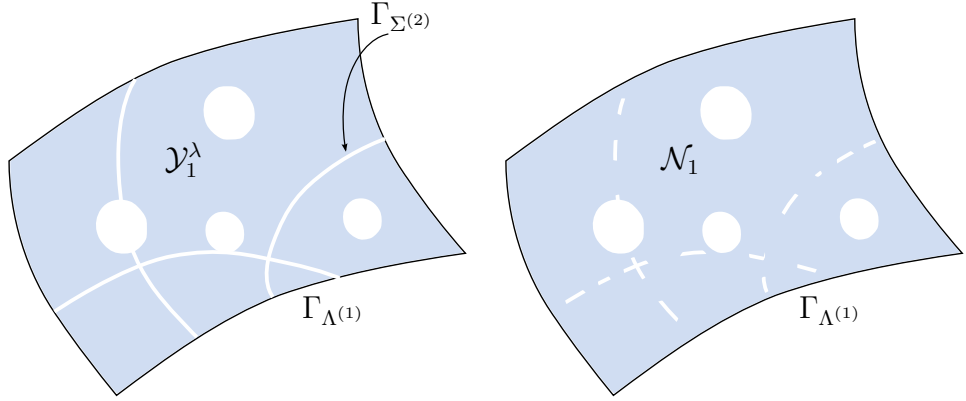
\caption{The final nondegeneracy set of step 1: removing all essential resonances results in a disconnected set, but removing only the degenerate part does not destroy connectivity. }
\label{fig:second-nondeg}
\end{figure}


\subsubsection{Step 2 of the induction}

We completed step 1 with 
\begin{itemize}
\item the collection of rank one lattices $\cLi{1}$, 
with associated bases $\cBi{1}$,

\item a collection of essential rank two lattices $\cEi{2}$, 

\item a dual collection of codimension one resonant manifolds 
$\mathcal P^{(1)}$,

\item the nondegenerate set $\cN_1\subset \mathcal P^{(1)}$
and is $(1-4^{-1}\gamma)-$dense in $B^n$. 
\end{itemize}
By step 1, 
for each essential resonance $\Sgi{2} \supset \cBe{2}$ and $p \in \overline{\Gamma_{\Sgi{2}} \cap \cN_1}$  the pair $\left( H_{p, \cBe{2}}^s, h(\cBi{1}|\cBe{2})  \right) $ is nondegenerate. Therefore, Property A3 applies with 
\[
	m=2,\ \cB^\st = \cBe{2},\   U^\st_0 = U_{p_0, \cBe{2}},\  h = h(\cBi{1}|\cBe{2}).  
\]
Moreover, we can choose a \emph{uniform} constant $N_2 = N_2(\cEi{2}, \cN_1)$ over all $\Sgi{2} \in \cEi{2}$, $\cBi{1} \subset \Sgi{2}$ and $p \in \overline{\Gamma_{\Sgi{2}} \cap \cN_1}$ such that the conclusion of Property A3 hold.

We are now ready to define the set $\cLi{2}$.   We say a rank $2$ lattice $\Lbi{2}$ is 
\emph{admissible} if the following hold. 
\begin{enumerate}
\item There exists  $\Lbi{1} \in \cLi{1}$ such that  $\Lbi{1} \subset \Lbi{2}$. 
\item $\Lbi{2}$ cannot be generated by the previous generation essential resonances, namely
\begin{equation}
\label{eq:lat-generation-2}
	\Lbi{2} \not\subset \bigvee \{ \Sgi{2} \in \cEi{2}\},
\end{equation}
where $\bigvee$ is the smallest irreducible lattice that 
contains all lattices $\Sgi{2} \in \cEi{2}$. 
\item Item 2 ensures that $\Lbi{1} \subset \Lbi{2}$ is unique. Otherwise, suppose we have  $\Lbi{1}_1, \Lbi{1}_2 \subset \Lbi{2}  $ with bases $\cBi{1}_1, \cBi{1}_2$, then $\Lbi{2} = \Lbi{1}_1 \vee \Lbi{1}_2$,  and
\[
M(\Lbi{2}|\Lbi{1}_1) \le \max_{k \in \cBi{1}_1 \cup \cBi{1}_2}\{|k|\} \le M_{s}(\cPi{s}, \cN_{s}),
\]
hence $\Lbi{2} \in \cEi{2}$, which is a violation of item 2. 
\item (ghost property) For each $\Lbi{1} \subset \Lbi{2}$ and 
$\Lbi{1}\subset \Sgi{2}$, we have
\begin{equation}
\label{eq:ess-dom-2}
	M( \Sgi{2} \vee \Lbi{2}|\Sgi{2} ) > N_s(\cEi{s+1}, \cN_s). 
\end{equation}
In particular, for an adapted basis  $\cB^\st, \cB^\wk$  of $\Sgi{2} \subset \Sgi{2} \vee \Lbi{2}$
and $p \in \overline{\Gamma_{\Sgi{2}} \cap \cN_1}$
the conclusion of Property A3 hold\footnote{The name ``ghost'' comes 
from the fact that we test $k_2$ against all possible essential lattices 
$\Sgi{2}$}. We also point out in this case $\Sgi{3} = \Sgi{2} \vee \Lbi{2}$ will be an element of the next generation essential resonance. 
\end{enumerate}

We claim that the lattices that are \emph{not admissible} can be generated by a finite set of integer vectors. Therefore the resonance manifolds of the admissible lattices form a dense set. As a result:
\begin{lemma}\label{lem:step2-net}
There exists an collection of rank two admissible lattices $\cLi{2}$ 
such that 
\[
	\cPi{2} = \bigcup_{\Lbi{2} \in \cLi{2}} \Gamma_{\cLi{2}} \cap \cN_1
\]
is connected, $4^{-2} \gamma -$dense in $\cN_1$ and 
$(1- 2\cdot 4^{-2})\gamma$ dense in $B^n$. 
\end{lemma}

For each $\Lbi{2} \in \cLi{2}$, there is a unique $\Lbi{1} \in \cLi{1}$ with $\Lbi{1} \subset \Lbi{2}$. Since $\Lbi{1}$ comes with a standard basis $\cBi{1}$, we extend it using Proposition~\ref{prop:basis} to obtain a standard basis $\cBi{2}$ of $\Lbi{2}$. We call the collection of all basis $\cKi{2}$. 

Similar to step 1, we define the non-degeneracy set $\cY_2^\lambda(H_1, \cPi{2}) \subset \cPi{2}$ by the following condition: For any lattice 
$\Lbi{2} \in \cLi{2}$, with basis $\cBi{2}$, and  
$p \in \Gamma_{\Lbi{2}} \cap \cPi{2}$, 
the averaged potential $U_{p, \cBi{2}}$ has at most two $\lambda-$nondegenerate minima. 
\begin{lemma}
There exists an open set $\cU_2 \subset \cU_1$ and $\lambda_2>0$, such that for $H_1 \in \cU_2$, the nondegeneracy set $\cY_2^{\lambda_2}(H_1, \cPi{1})$ is open, connected, $4^{-2}\gamma-$dense in $\cPi{2}$, and 
$(1 - 2\cdot 4^{-2})\gamma-$dense in $B^n$.
\end{lemma}
Using compactness, we obtain that for  
\[
	m=2, \quad \cB^\st = \cBi{2} \in \cKi{2}, \quad p_0 \in \Gamma_{\cBi{2}} \cap \cY_2^{\lambda_2}, \quad U^\st = U_{p_0, \cBi{2}}, 
\]
there exists $M_2(\cPi{2}, \cY_2^{\lambda_2})>0$, such that the conclusion of Property A3 applies for all $\lM(\cB^\wk) > M_2$. As before, we require 
\[
	M_2(\cPi{2}, \cY_2^{\lambda_2}) > \max_{k_1, k_2 \in \cBi{2} \in \cKi{2}}(|k_1|,|k_2|). 
\]

We now define essential lattices. It suffices to define bases of these lattices. As in step 1, 
$\cEi{3}$ is the set of all rank $3$ irreducible lattices $\Sgi{3}$ satisfying the following conditions: there exists $\Lbi{2} \in \cLi{2}$ such that 
\[
	\Sgi{3} \subset \Lbi{2}, \quad M(\Sgi{3}|\Lbi{2}) \le M_2(\cPi{2}, \cY_2^{\lambda_2}). 
\]
Starting from $s=2$, the essential resonances come with a hierarchical structure (see Figure  \ref{fig:essential-lattices}).
\begin{itemize}
\item {\it Type $(3,1)$:} We say $\Sgi{3}$ is of type $(3,1)$ if 
there exists $\Sgi{2} \in \cEi{2}$ such that $\Sgi{2} \subset \Sgi{3}$. The collection of $(3,1)$ essential lattices is denoted $\cEi{3}_1$. 

This element $\Sgi{2}$ is necessarily unique, otherwise $\Sgi{3} \supset \Lbi{2}$ can be generated from two elements from  $\cEi{2}$, leading to a contradiction with item 2 in the definition of $\cLi{2}$. Moreover, since $\cLi{2}$ cannot be generated by vectors from $\Sgi{2}$, we have $\Sgi{3} = \Sgi{2} \vee \Lbi{2}$. 

By item 3 in the definition of $\cLi{2}$, 
\[
	M(\Sgi{2}\vee \Lbi{2}|\Sgi{2}) = M(\Sgi{3}|\Sgi{2}) > N_2(\cEi{2}, \cN_1).
\]
Recall that $\Sgi{2}$ comes with fixed basis $\cBe{2}$. We use Proposition~\ref{sec:basis} to extend this basis to an adapted basis $\cBe{3}$ of $\Sgi{3}$. We take this basis as the fixed basis of $\Sgi{3}$. For each $p \in \Gamma_{\Sgi{3}}$,  Property A3 applies. We say that
\[
	H_{p, \cBe{2}} \text{ dominates } H_{p, \cBe{3}}. 
\]
\item {\it Type $(3,2)$:} $\Sgi{3}$ is called type $(3,2)$ if it does not contain any element in $\cEi{2}$. The collection is denoted $\cEi{3}_2$. 

In this case, by definition, $\Sgi{3}\in \cEi{3}_2$ and we have 
\[
	M( \Sgi{3}|\Lambda_{k_1} ) > M_1(\cPi{1}, \cY_1^{\lambda_1}).
\]
We use Proposition~\ref{sec:basis} to extend $k_1$ to an adapted basis $\cBe{3}$ of $\Sgi{3}$, taken as the fixed basis for $\Sgi{3}$. Property A2 applies, and we say that 
\[
	H_{p, \cBi{1}} \text{ dominates } H_{p, \cBe{3}}. 
\]
\end{itemize}
We now have the decomposition
$$
\cEi{3}:=\cEi{3}_1\cup \cEi{3}_2.
$$

\begin{figure}[t]
\centering
\def\svgwidth{3.2in}
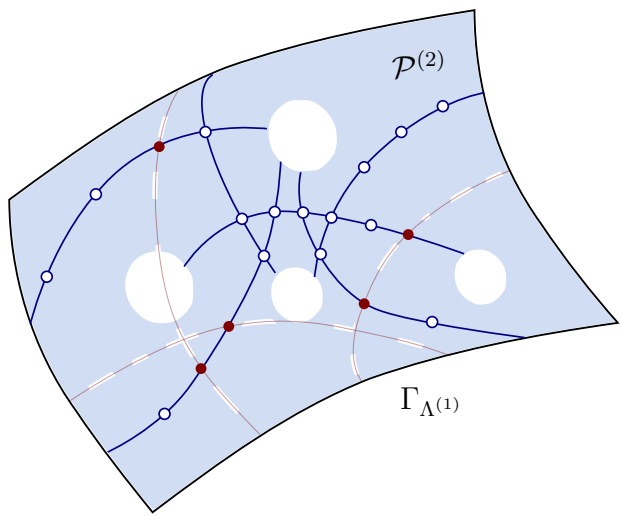
\caption{Hierachy of essential resonances: faint red curves are the previous generation essential resonances, blues line are current generation diffusion path, solid red dots are of type $(3,1)$, hollow blue dots are of type $(3,2)$. }
\end{figure}

 For each $\Sgi{3} \subset \cEi{3}$ with basis $\cBe{3}$, we define the nondegeneracy set $\cZ_2(H_1, \Sgi{3}, \cY_2^{\lambda_2}) \subset 
\Gamma_{\Sgi{3}} \cap \cY_2^{\lambda_2}$ to be the subset such that 
for each $\Lbi{2} \subset \Sgi{3}$ with basis $\cBi{2}$, 
 \[
 	\left( H_{p, \cBe{3}}^s, h(\cBi{2}|\cBe{3}) , \rho \right) 
 \]
 satisfies the AM property. We then define $\cZ_2(H_1, \cEi{3}, \cY_2^{\lambda_2})$ to be the union of all $\cZ_2(H_1, \Sgi{3}, \cY_2^{\lambda_2})$ over essential resonances $\Sgi{3} \in \cEi{3}$. 

If the essential resonance $\Sgi{3} \in \cEi{3}$ is type $(3,1)$, we use Property A3; if $\Sgi{3}$ is of type $(3,2)$, we use Property A2. This allows us to prove the following non-degeneracy lemma. 
\begin{lemma} There exists a relative open set \ $\cU_2' \subset \cU_2$  and a relative open set $\tilde{\cZ}_2$ such that the following hold. 
\begin{enumerate}
\item For each $H_1 \in \cU'$, $\tilde{\cZ}_2$  is compactly contained in $\cZ_2(H_1, \cEi{3}, \cY_2^{\lambda_2})$. 
\item 
The subset
\[
	\cN_2 :=   \cY_2^{\lambda_2} \cap
\tilde{\cZ}_2 
\]
is open, connected,  $4^{-2}\gamma-$dense in $\cPi{2}$ and $(1 - 2\cdot 4^{-2})\gamma-$dense in $B^n$.
\end{enumerate}
\end{lemma}

\subsubsection{Step $s+1$ of the induction}

We completed step $s$ with 
\begin{itemize}
\item the collection of rank $s$ lattices $\cLi{s}$, 
with associated bases $\cBi{s}$,

\item a collection of essential rank s+1 lattices $\cEi{s+1}$, 

\item a dual collection of codimension one resonant manifolds 
$\mathcal P^{(s)}$,

\item the nondegenerate set $\cN_s\subset \mathcal P^{(s)}$,
which is $(1-4^{-s}\gamma)-$dense in $B^n$. 
\end{itemize}

\breakheading{The diffusion path}. 
\begin{itemize}
\item We have the collection of lattices $\cLi{1}, \cdots, \cLi{s}$, with $\cLi{j} = \{\Lbi{j}_i\}$ are irreducible rank $j$ lattices. For each $\Lbi{j} \in \cLi{j}$, there exists a \emph{unique} $\Lbi{j-1}\subset \Lbi{j}$ and such that 
$\Lbi{j-1} \subset  \cLi{j-1}$. 

\item Each $\Lbi{s}\in \cLi{s}$ has a \emph{ordered} basis defined in the following way. For each $\Lbi{1}$ we fix a basis $\cBi{1} =\{k_1\}$ which is unique up to a sign. From the previous property, $\Lbi{s}$ comes with the chain of inclusion
\[
	\Lbi{1} \subset \cdots \subset \Lbi{s}, \quad \Lbi{j} \in \cLi{j},\ 
1 \le j \le s,
\]
and we extend the basis $\cBi{1}$ of $\Lbi{1}$ an increasing set of bases $\cBi{1} \subset \cdots \subset \cBi{s}$ by consecutive application of Proposition~\ref{prop:basis}. 

\item  We use $\cKi{s}$ to denote the collection of standard bases. For each $\cBi{s} = [k_1, \cdots, k_s]$, we denote $|\cBi{s}| = \sup_i |k_i|$. 
\item The diffusion path at step $s$ is 
\[
	\cPi{s} = \bigcup_{\Lbi{s} \in \cLi{s}} \Gamma_{\Lbi{s}}. 
\]
The set $\cPi{s}\cap B^n$ is connected and $(1- 2 \cdot 4^{-s})\gamma-$dense in $B^n$.
\end{itemize}

\breakheading{Essential resonances}. 
\begin{itemize}
\item We have the essential lattices $\cEi{2}, \cdots, \cEi{s+1}$, where for each $1\le j \le s$, $\Sgi{j+1} \in \cEi{j+1}$ is a rank $j+1$ irreducible lattice. For each $\Sgi{j+1}$, there exists at least one, and at most two element $\Lbi{j} \in \cLi{j}$, such that $\Lbi{j} \subset \cLi{j}$. 
\item If essential lattice $\Sgi{s+1}\in \cEi{s+1}$ contains only 
one element $\Sgi{s} \in \cEi{s}$, then there exists $1 \le j \le s$, such that 
\[
	\Sgi{j+1} \subset \cdots \subset \Sgi{s+1},  \quad \Sgi{t} \in \cEi{t}, j+1 \le t \le s+1
\]
is the longest chain of essential lattices, meaning $\Sgi{j+1}$  does not contain any element of $\cEi{j}$. We then have the following inclusion
\[
	\Lbi{1} \subset \cdots \subset \Lbi{j-1} \subset \Sgi{j+1} \subset \cdots \subset \Sgi{s+1}.
\]
We use Proposition~\ref{prop:basis} to obtain the chain of adapted bases (called \emph{ordered} basis:
\[
	\cBi{1} \subset \cdots \subset \cBi{j-1} \subset \cBe{j+1} \subset \cdots \subset \cBe{s+1},
\]
where each $\cBi{t}$ is a basis of $\Lbi{t}\in \cLi{t}$ and each $\cBe{t}$ is 
a basis of $\Sgi{t} \in \cEi{t}$. Recording the increment of rank in the chain, 
the essential resonance $\Sgi{s+1}$ is called of type $(s+1,j)$.
Denote by $\cEi{s+1}_j$ the set of essential resonances with this property.
\end{itemize}

\begin{figure}[h]
  \begin{center}
  \includegraphics[width=12.5cm]{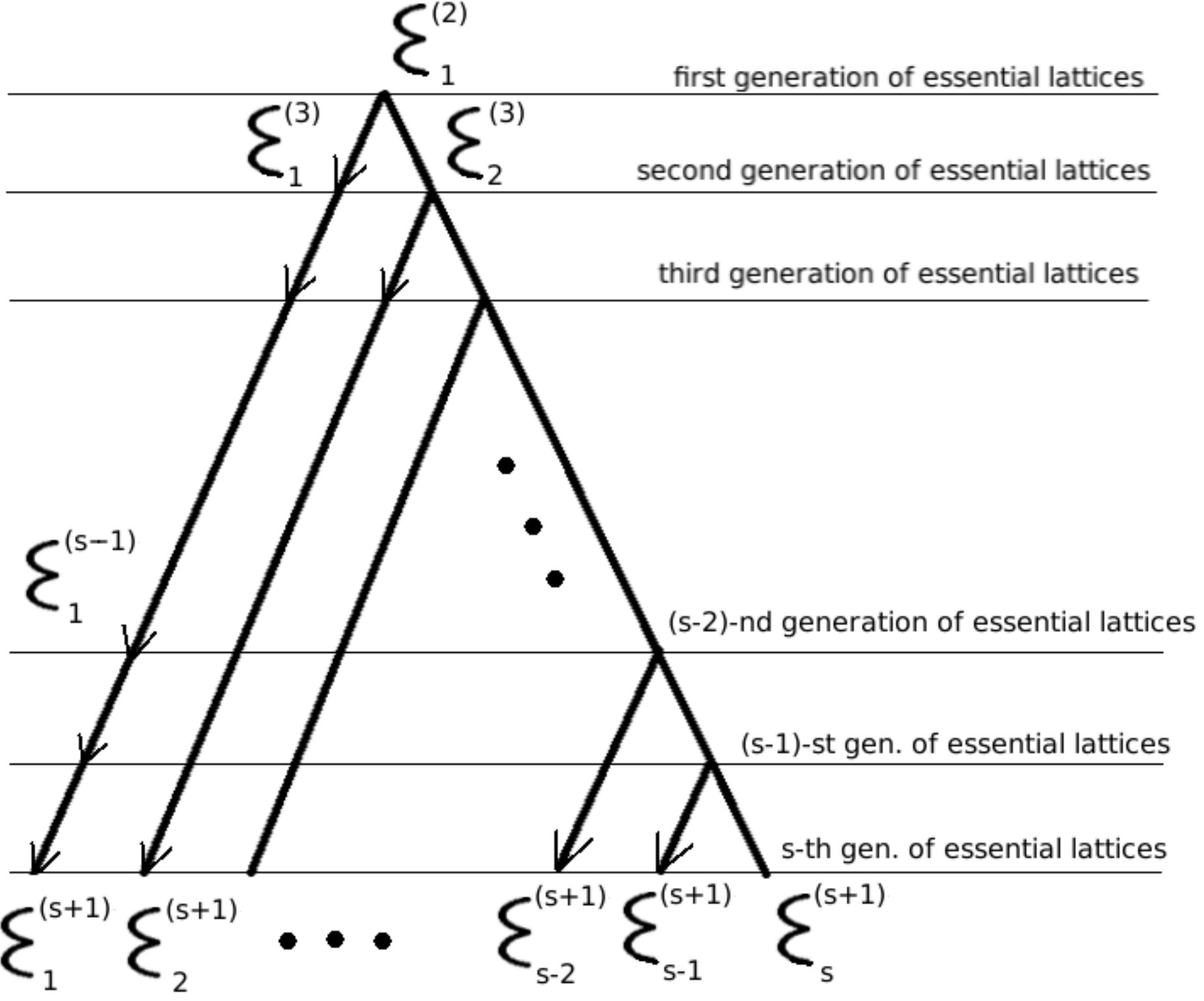}
  \end{center}
  \caption{Essential lattices }
 \label{fig:essential-lattices}
\end{figure}

\breakheading{Strong system and nondegeneracy}.
\begin{itemize}
\item For each $ 1 \le  j \le s$, there exists the nondegeneracy set $\cN_j \subset \cPi{j}$, with the property that each $\cN_j$ is relative open, connected and $(1- 4^{-j}\gamma)-$ dense in $\cPi{j}\cap B^n$ and $(1- 2\cdot 4^{-j})\gamma-$dense in $B^n$. The following inclusion hold
\[
	\cPi{1} \supset \cN_1 \supset \cdots \supset \cPi{s} \supset \cN_s. 
\]
\item There exists a sequence of (nonempty) relative open sets
\[
	\{\|H_1\|_{C^r}=1\} \supset \cU_0 \supset \cU_1 \supset \cU_1' \supset \cdots \supset \cU_s \supset \cU_s'
\]
\item There exists $\lambda_{s} >0$ such that for each $H_1 \in \cU_s$, $\Lbi{s}\in \cLi{s}$ with basis $\cBi{s}$, and  $p \in \overline{\cN_s \cap \Gamma_{\Lbi{s}}}$, the strong system $H_{p, \cBi{s}}$ is nondegenerate 
in the sense of Property A3 and the averaged potential $U_{p,\cBi{s}}$ has 
at most two $\lambda$-nondegenerate minima. Using compactness, let 
\[
	M_s(\cPi{s}, \cN_s)> \sup_{\cBi{s} \in \cKi{s}}|\cBi{s}|
\]
be a uniform constant such that Property A3 applies. 
\item For each $H_1 \in \cU_s'$, $\Sgi{s+1} \in \cEi{s+1}$ with basis $\cBe{s+1}$, each $\Lbi{s} \subset \Sgi{s+1}$ with basis $\cBi{s}$, and $p \in \overline{\cN_s \cap \Gamma_{\Sgi{s+1}}}$, the pair
\[
	\left(   H_{p, \cBe{s+1}}, h(\cBi{s}, \cBe{s+1})\right)
\]
is nondegenerate in the sense of Property A3. Using compactness, let 
\[
	N_s(\cEi{s+1}, \cN_s)>0
\]
be a uniform constant such that Property A3 applies. 
\end{itemize}

\breakheading{Domination Properties}
\begin{itemize}


\item Let $\Sgi{s+1} \in \cEi{s+1}_{ j}$ be an essential resonance of type
$(s+1,j)$, 
then we have the chain
\[
	\Lbi{1} \subset \cdots \subset \Lbi{j-1} \subset \Sgi{j+1} \subset \cdots \subset \Sgi{s+1}.
\]
The following domination property holds:
\[
	M(\Sgi{j+1}|\Lbi{j-1}) > M_{j-1}(\cPi{j-1}, \cN_{j-1}), \qquad \quad \qquad 
\]
\[
	 M(\Sgi{t+1}|\Sgi{t}) > N_{t-1}(\cEi{t}, \cN_{t-1}), \quad j+1 \le t \le s. 
\]
\item As a corollary of the domination properties, for $\Sgi{s+1}$ with the type $(s+1,j)$, 
let 
\[
	\cBi{1} \subset \cdots \subset \cBi{j-1} \subset \cBe{j+1} \subset \cdots \subset \cBe{s+1}
\]
be the chain of basis. Then
\begin{itemize}
\item For each $p \in \overline{\Gamma_{\Lbi{j-1}}\cap \cN_{j-1}}$, the system $H_{p, \cBi{j-1}}$ dominates $H_{p, \cBe{j+1}}$ in the sense of Property A2.
\item Fore each $j+1 \le t \le s$, $p \in \overline{\Gamma_{\Sgi{t}} \cap \cN_{t-1}}$, the system $H_{p, \cBe{t}}$ dominates $H_{p, \cBe{t+1}}$ in the sense of Property A3. 
\end{itemize}
\end{itemize}

We now define the set $\cLi{s+1}$. This is essentially an elaboration of step 2.  We say the rank $s+1$ lattice $\Lbi{s+1}$ is 
\emph{admissible} if the following hold. 
\begin{enumerate}
\item There exists  $\Lbi{s} \in \cLi{s}$ such that  $\Lbi{s} \subset \Lbi{s+1}$. 
\item $\Lbi{s+1}$ cannot be generated by any previous generation essential resonances, namely
\begin{equation}
\label{eq:lat-generation}
	\Lbi{s+1} \not\subset \bigvee \{ \Sgi{s+1} \in \cEi{s+1}\},
\end{equation}
where $\bigvee $ is the smallest irreducible lattice that 
contains all lattices $\Sgi{s+1} \in \cEi{s+1}$. 

\item Item 2  ensures that $\Lbi{s} \subset \Lbi{s+1}$ is unique. Otherwise, suppose we have  $\Lbi{s}_1, \Lbi{s}_2 \subset \Lbi{s+1}  $ with bases $\cBi{s}_1, \cBi{s}_2$, then $\Lbi{s+1} = \Lbi{s}_1 \vee \Lbi{s}_2$,  and
\[
M(\Lbi{s+1}|\Lbi{s}_1) \le \max\{|\cBi{s}_1|, |\cBi{s}_2|\} \le M_{s}(\cPi{s}, \cN_{s}),
\]
hence $\Lbi{s+1} \in \cEi{s+1}$, which is a violation of item 2. 
\item (ghost property) For each $\Lbi{s} \subset \Lbi{s+1}$ and 
$\Lbi{s}\subset \Sgi{s+1}$, we have
\begin{equation}
\label{eq:ess-dom-s}
	M( \Sgi{s+1} \vee \Lbi{s+1}|\Sgi{s+1} ) > N_s(\cEi{s+1}, \cN_s). 
\end{equation}

\end{enumerate}
\begin{lemma}\label{lem:steps-net}
There exists an collection $\cLi{s+1} = \{\Lbi{s+1}\}$ of admissible pairs 
such that 
\[
	\cPi{s+1} = \bigcup_{\Lbi{s+1} \in \cLi{s+1}} \Gamma_{\Lbi{s+1}} \cap \cN_s
\]
is connected, $ 4^{-s-1}\gamma-$dense in $\cN_s$,  and  $(1- 2\cdot 4^{-s-1})\gamma$ dense in $B^n$. 
\end{lemma}

The set $\cY_{s+1}^\lambda(H_1, \cPi{s+1}) \subset \cPi{s+1}$ is defined by the following condition: For any $\Lbi{s+1} \in \cLi{s+1}$ with basis $\cBi{s+1}$, and  $p \in \Gamma_{\cBi{s+1}} \cap \cPi{s+1}$, there exists $0 < \lambda < \lambda'$ such that the averaged potential $U_{p, \cBi{s+1}}$ has at most two $\lambda'-$nondegenerate minima. 
\begin{lemma}
There exists an open set $\ \cU_{s+1} \subset \cU_s'$ and $\lambda_{s+1}>0$, such that for $H_1 \in \cU_{s+1}$, the nondegeneracy set $\cY_{s+1}^{\lambda_{s+1}}(H_1, \cPi{s+1})$ is $4^{-s-1}\gamma-$dense in $\cPi{s+1}$ and connected. 
\end{lemma}

Define 
\[M_{s+1}(\cPi{s+1}, \cY_{s+1}^{\lambda_{s+1}}) > \sup_{\cBi{s+1}\in \cKi{s+1}} |\cBi{s+1}|\]
 be the uniform constant over all $H_1 \in \cU_{s+1}$, $p \in \cY_{s+1}^{\lambda_{s+1}}$, and $\Lbi{s+1} \in \cLi{s+1}$.  
The essential lattice set $\cEi{s+2}$ is defined as the set of all rank $s+2$ irreducible lattices $\Sgi{s+2}$ satisfying the following conditions: there exists $\Lbi{s+1} \in \cLi{s+1}$ such that 
\[
	\Sgi{s+2} \supset \Lbi{s+1}, \text{ and } M(\Lbi{s+1}|\Sgi{s+2}) < M_{s+1}(\cPi{s+1}, \cY_{s+1}^{\lambda_{s+1}}). 
\]
We have the following remarks:
\begin{itemize}
\item Suppose there exists $\Sgi{s+2} \supset \Sgi{s+1}$ with $\Sgi{s+1}\in \cEi{s+1}$, then $\Sgi{s+1}$ is unique. Otherwise, suppose $\Sgi{s+2}$ contains both $\Sgi{s+1}_1, \Sgi{s+1}_2$, then there exists $\Lbi{s+1} \subset \Sgi{s+2} = \Sgi{s+1}_1 \vee \Sgi{s+1}_2$, this is a violation of \eqref{eq:lat-generation}. 
\item In case that $\Sgi{s+2} \supset \Sgi{s+1}$, then for $\Lbi{s+1} \in \cLi{s+1}$ with $\Sgi{s+2} \supset \cLi{s+1}$, we get 
\[
	M(\Sgi{s+1}|\Sgi{s+2}) = M(\Sgi{s+1}|\Sgi{s+1}\vee \Lbi{s+1}) > N_s(\cEi{s+1}, \cN_s)
\]
by \eqref{eq:ess-dom-s}. 
\end{itemize}

Finally,  for each $\Sgi{s+2} \subset \cEi{s+2}$, we define the nondegeneracy set $\cZ_{s+1}(H_1, \Sgi{s+2}, \cY_{s+1}^{\lambda_{s+1}})$ to be the subset that for each $\cBi{s+1}=(k_1, \cdots,  k_{s+1}) \subset \Sgi{s+2}$, the pair
 \[
 	\left( H_{p, \cBe{s+2}}^s, h(\cBi{s+1}|\cBe{s+2})  \right) 
 \]
 is nondegenerate. We then define $\cZ_{s+1}(H_1, \cEi{s+2}, \cY_{s+1}^{\lambda_{s+1}})$ to be the union  over essential resonances $\Sgi{s+2} \in \cEi{s+2}$. 

The following lemma is proven using the type of essential resonances, similar to step 2. 
\begin{lemma}\label{lem:ess-nondeg-s}
 Suppose $s+2 < n$.  There exists an open set $\cU_{s+1}' \subset \cU_{s+1}$  and a relative open set $\tilde{\cZ}_{s+1}$ such that the following hold. 
\begin{enumerate}
\item For each $H_1 \in \cU'_{s+1}$, $\tilde{\cZ}_{s+1}$  is compactly contained in $\cZ_{s+1}(H_1, \cEi{s+2}, \cY_{s+1}^{\lambda_{s+1}})$. 
\item 
The subset
\[
	\cN_{s+1} :=   \cY_{s+1}^{\lambda_{s+1}} \cap \tilde{\cZ}_{s+1}   
\]
is open, connected,  $4^{-s-1}\gamma-$dense in $\cPi{2}$ and $(1 - 4^{-s-1})\gamma-$dense in $B^n$.
\end{enumerate}
Moreover, if $s+2 =n$, 
\[
	\bigcup_{\Sgi{s+2} \in \cEi{s+2}} \Gamma_{\Sgi{s+2}} \cap \cY_{s+1}^{\lambda_{s+1}}
\]
is a collection of isolated points. Then the same two points hold with 
\[
	\tilde{Z}_{s+1} = \bigcup_{\Sgi{s+2} \in \cEi{s+2}} \Gamma_{\Sgi{s+2}} \cap \cY_{s+1}^{\lambda_{s+1}}. 
\]
\end{lemma}
This finishes the construction of the lattices and verification of properties for step $s+1$.

\subsubsection{Concluding the induction} 

The induction ends when $\cU_{n-1}'$,  $\cLi{n-1}$, $\cPi{n-1}$, $\cN_{n-1}$ and $\mathcal E^{(n)}$ are defined. Then $\cPi{n-1} \cap \cN_{n-1}$ is $\gamma-$dense diffusion path in $\cB^n$, and for each $p \in \Gamma_{\Lbi{n-1}}\cap \cN_{n-1}$, $H_1 \in \cU_{n-1}'$, the potential $U_{p, \cBi{n-1}}$ has at most two $\lambda_{n-1}-$nondegenerate minima. 

We then have 
\begin{lemma}
There exists an open and dense set of $\cU''_{n-1} \subset \cU_{n-1}'$ such that  [H1$\lambda_{n-1}$] holds for all $H_1 \in \cU_{n-1}''$ on $\cPi{n-1} \cap \cN_{n-1}$. 
\end{lemma}

Moreover, from Lemma~\ref{lem:ess-nondeg-s} we know that condition [H2] holds on all essential resonances. Therefore, the diffusion path $\cPi{n-1}\cap \cN_{n-1}$ satisfies all the conditions required.

\section{Diffusion mechanism and  AM property}
\label{sec:diffusion-amproperty}

The goal of this section is to give a short review of diffusion mechanisms. Then we focus on diffusion mechanism 
using variational methods and discuss difficulties arising 
in higher dimensions. After that we explain the role of 
dominant systems. 

 In \cite{Arn64} Arnold proposed the following example 
\[
H(q,p,\varphi,I,t)=\dfrac{I^2}{2}+\dfrac{p^2}{2}+\eps
(1-\cos q)(1-\mu (\sin \varphi+\sin t)),
\]
where $q,\varphi$ and $t$ are angles and $p,I\in\R$. 
This example is a perturbation of the product of 
a one-dimensional pendulum and a one-dimensional rotator.
There is a rich literature on 
Arnold example and we do not intend to give extensive list 
of references; we mention 
\cite{AKN, BB02, DLS06, Tre04}, and references therein. 

The important feature of this example is that it has a 3-dimensional NHIC
$\Lambda=\{p=q=0\}$, which is the direct product of $\R$ 
and $2$-dimensional torus $\T^2$. Later having 
a 3-dimensional NHIC means that there is a NHIM diffeomorphic to 
the direct product of $\R$ and $2$-dimensional torus $\T^2$.

\begin{center}
{\bf Whiskered tori and transition chains}
\end{center}

In \cite{Arn64} Arnold noticed that for each $\omega \in\R$ 
there are an invariant $2$-dimensional torus
$$\T^2_\omega=\{p=q=0,I=\omega\}$$ having 
$3$-dimensional stable and instable manifolds
$W^s(\T^2_\omega)$ and $W^u(\T^2_\omega)$ resp. 
Notice that orbits inside $\T^2_\omega$ have a well
defined rotation number equal to $\omega$.  

Call an ordered sequence of tori $\{\T^2_{\omega_i}\}_{i}$
{\it transition chain} if for each $i$ we have 
\[
W^s(\T^2_{\omega_i}) \text{ and }W^u(\T^2_{\omega_{i+1}})
\text{ intersect transversally}.
\]
In \cite{Arn64} proved that for any $a<b$ 
there is a transition chain such that $\omega_0=a$ 
and $\omega_N=b$ for some $N$ and showed that 
this implies existence of orbits asymptotic to 
$\T^2_b$ in the future and to $\T^2_a$ in the past. 

\begin{center}
{\bf Generalized transition chains}
\end{center}

In \cite{Mather91a} Mather proposed a diffusion 
mechanism where invariant tori where replaced 
by Aubry-Mather invariant sets for twist maps. 

For fiber convex superliear time-periodic Hamiltonians 
$H(\theta,p,t), \theta,t\in\T, p\in \R$ for each rotation 
number $\omega$ there is a ``minimal'' invariant set 
$\cA_\omega$ consisting ``minimal'' orbits rotation 
number $\omega$. 

\begin{center}
{\bf The $2$-torus graph property.}
\end{center}
Let $\pi:(\theta,p,t)\to (\theta,t)$ is the natural projection.
Mather proved that each such a set $\cA_\omega$ 
is a Lipschitz graph over $\pi \cA_{\omega}$, i.e. $\pi$ is 
one-to-one on $\cA_{\omega}$ and the inverse 
$\pi^{-1}: \pi \cA_{\omega}\to \cA_{\omega}$
is Lipschitz. 

We say that an invariant set $\cA_\omega$ has 
{\it a $2$-torus graph property}, if there is $C^1$ 
smooth map $\pi:T\T\times \T \to \T^2$ having 
maximal rank in a neighborhood of $\cA_\omega$ 
such that $\pi$ is one-to-one on $\cA_{\omega}$ 
and the inverse $\pi^{-1}: \pi \cA_{\omega}\to 
\cA_{\omega}$ is Lipschitz. 

\begin{itemize}
\item (rational case) if $\omega=p/q$ for some integers 
$p,q$ with $q>0$ the set $\cA_\omega$ contains 
``minimal'' periodic  orbits.
\item (irrational case) if $\omega\not\in \mathbb Q$ 
the set $\cA_\omega$ either  a Lipschitz $2$-torus, i.e. 
$\pi \cA_\omega=\T^2$ or contains a suspension of a Denjoy set. 
\end{itemize}

\begin{center}
{\bf A generalized transition chain.}
\end{center}
Using give a precise meaning of a stable and 
an unstable set of each invariant set $\cA_\omega$. 
These sets are not necessarily manifolds, but still denoted 
$W^s(\cA_{\omega})$ and $W^u(\cA_{\omega})$ resp. 
One can give a precise meaning of transverse intersection 
of these sets using the barrier function. Call 
it {\it a generalized transverse intersection.} 

An ordered sequence of ``minimal'' invariant sets 
$\{\cA_{\omega_i}\}_i$ is called {\it a generalized transition chain}
if 
\begin{itemize}
\item each invariant set $\cA_{\omega_i}$ has 
a $2$-torus graph property;
\item for each $i$ invariant sets $W^s(\cA_{\omega_i})$ 
and $W^u(\cA_{\omega_{i+1}})$ have 
{\it a generalized transverse intersection}. 
\end{itemize}

In \cite{Mather91,Mather93} Mather proposed 
a generalization of constuction from \cite{Mather91a}. 
Inspired by these ideas, in \cite{Ber08, CY04,CY09} 
proved that for a generic perturbation in the Arnold 
example there are generalized transition chains. 
Moreover, there are orbits shadowing this transition chain. 

\begin{center}
{\bf An equivalence of invariant sets in a generalized 
transition chain.}
\end{center}
In \cite{Ber08} replaces generalized transversality 
condition with forcing relation. Then he shows that if 
$\cA_{\omega}$ forces $\cA_{\omega'}$ and vise versa
then this is an equivalence relation. In particular, if 
nearby invariant sets $\cA_{\omega_i}$ and 
$\cA_{\omega_{i+1}}$ are equivalent, then there 
are orbits heteroclinic orbits for any pair of invariant 
sets in a generalized transition chain. 

In \cite{BKZ11} we construct ``short'' $3$-dimensional 
NHICs. Then we show 
that each of such cylinders carries a generalized 
transition chain. Moreover, all invariant sets in such 
a chain are equivalent and, therefore, there are orbits 
connecting any pair of invariant sets in 
this transition chain.

In \cite{KZ13} we construct a ``connected'' collection 
of $3$-dimensional NHICs and show that each cylinder 
carries a generalized transition chain.

\begin{center}
{\bf The $2$-torus graph and AM properties.}
\end{center}

Partial averaving of nearly integrable system 
$H_\eps = H_0+\eps H_1$ near a resonant manifold 
leads to a mechanical system of $d\le n$ degrees of
freedom  
\[
H(\varphi, I) = K(I) - U(\varphi), \quad \varphi\in \T^d, I \in \T^d,
\]
where $K(I)$ is a positive definite quadratic form
and $U$ is a sufficiently smooth function (see (\ref{eq:slow-system})). 
 
In order to find ``minimal'' invariant set
having the $2$-torus graph property
\begin{itemize}
\item we construct $3$-dimensional NHICs;
 
\item we prove that each $3$-dimensional NHIC
contains a family of ``minimal'' invariant set
and each such a set is localized.

\item due to localization we prove that the projection 
along the action component onto the $2$-torus $\T^2$
is one-to-one with a Lipchitz inverse. 
\end{itemize}

In order to construct a $3$-dimensional NHIC for 
$H_\eps$ near a maximal order resonance it suffices
to construct a $2$-dimensional NHIC, diffeomorphic
to the standard cylinder, for the averaged mechanical 
system $H$. 

Due to concervation of energy each $2$-dimensional 
NHIC consists of ``minimal'' hyperbolic periodic orbits. 
This leads to the following problem: 

{\it Construct a family of ``minimal'' invariant sets 
consisting of hyperbolic periodic orbits! }

In the case $d=2$ it suffices to consider minimal sets 
with rational rotation vectors. Indeed, generically 
minimal sets with rational rotation vector is a hyperbolic 
periodic orbit. In the case $d>2$ it is not longer true as 
the well-know Hedlund example shows  (see e.g. \cite{Lev97}). 
 
More exactly, if we consider an integer homology $h$ on $\T^d$ 
and consider infinite minimizers of homology class $h$, i.e. the Aubry 
set $\cA(h)$ (see section \ref{sec:intro-weak-kam} for precise definition). Then 
\begin{itemize}
\item $\cA(h)$ does not have to consist of periodic orbits or does not even have to have 
countably many invariant components (see e.g. \cite{Mather04}). 

\item $\cA(h)$ consisting  of periodic orbits do not 
guarantee they have homology $h$ (see \cite{Lev97}). 

\item In the class of Tonelli Hamiltonians minimization within 
the class of closed loops in some homology class $h$ might lead 
to non-hyperbolic minimal periodic orbits (see \cite{Arn98}).  
\end{itemize}
The AM property guarantee all these properties. 

\begin{center}
{\bf The jump.}
\end{center}

In \cite{KZ13} section 12 we show that for each pair 
of ``crossing'' cylinders there is {\bf a jump} from 
an invariant set from one generalized transition chain 
with another one. The jump, in particular, means that 
these invariant sets are equivalent and, therefore, 
invariant sets from both generalized transition chains 
are equivalent. 

One of the main conclusions of this paper is that 
we construct a diffusion path $\Gm$ and a ``connnected'' 
collection $3$-dimensional NHICs and show each of 
these cylinders carries a collection of invariant 
sets haing $2$-torus property. 

Using the technique from \cite{BKZ11,KZ13}
it should imply that invariant sets in each cylinder
form a generalized transition chain and are equivalent. 

Aside of many technical details we beleive that the only 
important missing part of construction of diffusing
orbits along the path $\Gamma$ is {\bf the jump}.
Construction of a variational problem leading to 
the jump for $3\frac 12$-degree of freedom 
is in section 8 \cite{KZ14}.

\section{Normally hyperbolic invariant manifolds}
\label{sec:nhic-proof}

In this section we state a version of the center manifold theorem and prove Theorem~\ref{thm:nhic-persist}. 
While the central manifold theorem is classical,  we need an version whose center direction is a non-compact 
set equipped with a Riemannian metric. This is done in the first two subsections. In the last subsection, 
we perform a reduction on our system to apply the central manifold theorem. 

\subsection{Normally hyperbolic invariant manifolds via isolation block}

We state an abstract theorem on existence of normally hyperbolic invariant 
manifolds for a smooth map $F$. based on Conley's isolation blocks (see McGehee, \cite{McG73}). 


We introduce a set of notations. 
We have three components $x \in \R^s,\ y\in \R^u,
\ z\in \Omega^c \subset \R^c$ , where $\Omega^c$ is a (possibly unbounded) convex set. 
We assume that $\Omega^c$ admits a $C^1$ complete Riemannian metric $g$. We also consider a Riemannian metric  on the product space $W = \R^s \times \R^u \times \Omega^c$ by taking the tensor product of $g$ and $\Omega^c$, and the standard Euclidean metric on $\R^s, \R^u$. 

Fix some $r>0$ and let  $D^s\subset \R^{s}$ and 
$D^u \subset \R^{u}$ be \emph{closed}
 balls of radius $r$ at the origin in $\R^s$ and $\R^u$ ($r$ is considered fixed and we omit the dependence).  Denote 
$D^{sc} = D^s \times \Omega^c$, $D^{uc} = D^u \times M$ and  $D = D^{sc} \times D^u$. 
 .

Consider a $C^1$ smooth map
\[
F:D=D^s\times D^u \times \Omega^c \to \R^s \times \R^u \times \Omega^c,
\]
we state a set of conditions guaranteeing the set 
\[
W^{sc}(F)=\{Z\in D: F^k(Z)\in D \text{ for all }k>0 \},
\]
called the center-stable manifold, 
is a graph $\{(X,Y) \in D^{sc}\times D^u: 
w^{sc}(X)=Y\}$ for a $C^1$ function $w^{sc}$.

\begin{itemize}
\item[{[C1]}] $\pi_{sc}F(D^{sc}\times D^u)
\subset D^{sc}$. 

\item[{[C2]}] $F$ maps $D^{sc}\times \partial D^u$ into  $D^{sc}
\times \R^u \setminus D^u$ and  is 
a homotopy equivalence.
%
\end{itemize}

The first two conditions guarantee a {\it topological 
isolating block}: $F$ stretches $D^{sc}\times B^u$ along 
the unstable component $D^u$ and is a weak contraction along the center-stable 
component $D^{sc}$. 

Now we state the {\it cone conditions}. For some $\mu>0$
\[
C^u_\mu(Z)=\{v=(v^c,v^s,v^u)\in T_ZD: 
\mu \|v^u\|^2\ge \|v^c\|^2+\|v^s\|^2\}.
\] 
Note that 
\[
	(C^u_\mu(Z))^c = \{v=(v^c,v^s,v^u)\in T_ZD: 
 \mu^{-1}(\|v^c\|^2+\|v^s\|^2) \ge \|v^u\|^2\} =: C^{sc}_{\mu^{-1}}(Z). 
\]
Let us also define
\[
	K_u^\mu(x_1, y_1, z_1) = \{(x_2,y_2,z_2) : \mu \|y_2 - y_1\|^2 \ge \|x_2 - x_1\|^2 + \dist(z_1, z_2)^2 \},
\]
where the distance is induced by the Riemannian metric $g$. 

We assume there is $\mu >1$ and $\nu>1$ with the property that 
for any $Z_1,Z_2\in D$ such that 
$Z_2 \in K_u^\mu(Z_1)$ we have 
\begin{itemize}
\item[{[C3]}] $F(Z_2)\in K_\mu^u(F(Z_1)).$
\item[{[C4]}]  $\|\pi_u(F(Z_2)-F(Z_1))\|
\ge \nu \|\pi_u(Z_2-Z_1)\|.$
\end{itemize}

\begin{proposition}
 \label{prop:lipschitz-manifold} 
(Lipshitz center-stable manifold theorem) Suppose $F$ satisfies 
conditions [C1-C4], then 
$W^{sc}(F)$ is given by the graph of a Lipschitz function
\[
W^{sc}(F) = \{(x, y, z) \in D : w^{sc}(x,z) = y\}.
\]
Moreover, for Lebesgue almost every $Z \in W^{sc}(F)$, we have 
\[
	T_Z W^{sc}(F) \in C^{sc}_{\mu^{-1}}(Z). 
\]
\end{proposition}

In order to obtain the center-unstable manifold, consider the involution $I:(x,y,z) \mapsto (y,x,z)$
and assume $\inv(F) = I \circ F^{-1} \circ I^{-1}$ satisfies the same conditions. 

\begin{theorem}\label{thm:center-c1}
Assume that $F,\inv(F)$ satisfies the conditions [C1-C4], there exists a $C^1$
function $w^c: M \to D$ such that 
\[
	W^c(F):=W^{sc}(F) \cap W^{uc}(F) = \{(x,y,z) \in D: (x,y) = w^c(z)\}. 
\]

\end{theorem}
\begin{proof}
Proposition~\ref{prop:lipschitz-manifold} implies the existence of 
Lipshitz functions $w^{uc}: D^{uc} \to D$ and $w^{sc}: D^{sc} \to D$, 
with
\[
	W^{sc}(F) = \{x = w^{sc}(y,z)\}, \quad W^{uc}(F) = W^{sc}(\inv(F)) = \{y = w^{uc}(x,z)\}. 
\]
Then standard arguments (see Theorem 5.18 in \cite{Shub}) implies 
these functions are $C^1$. The fact that $\mu>1$ and 
\[
	T_Z W^{sc}(F) \in C^{sc}_{\mu^{-1}}(Z),  \quad T_Z W^u(F) \in C^{sc}_{\mu^{-1}}(Z)
\]
implies $W^{sc}(F)$ and $W^{uc}(F)$ intersect transversally, and $W^{sc}(F)\cap W^{uc}(F)$ is a graph
over the center component $M$. 
\end{proof}

\subsection{Existence of Lipschitz invariant manifolds}

We prove Proposition~\ref{prop:lipschitz-manifold}.
Let $\cV$ be the set  $\Gamma \subset D$ satisfying the following conditions: (a) $\pi_{u}\Gamma = D^u$, (b) $Z_2 \in K_\mu^u(Z_1)$ for all $Z_1, Z_2 \in \Gamma$, where $\pi_u$ is the projection to the unstable component. These conditions ensures $\pi_u: \Gamma \to D^u$ is one-to-one and onto, therefore $\Gamma$ is a graph over $D^u$. Moreover, condition (b) further implies that the graph is Lipshitz. In particular, each $\Gamma \in \cV$  is a topological disk. 

\begin{lemma}\label{lem:inv-unstable-disk}
Let $\Gamma \in \cV$, then $F(\Gamma)\cap D\in \cV$. 
\end{lemma}
\begin{proof}
By [C4] for any $Z_1$ and $Z_2$ we have that 
$F(Z_2)$ belongs to the cone $K^u_{F(Z_1)}$ of $F(Z_1)$.
Thus, it suffices to show that $D^{u}\subset \pi_{u}
(F(\Gamma)\cap D)$. The proof is by contradiction. Suppose 
there is $Z_* \in B^{u}$ such that 
$Z_*\not\in \pi_{u}(F(\Gamma))$. 

We have the following commutative diagram
\be 
\beal 
\partial \Gamma \qquad & \quad \xhookrightarrow{i_1} & \Gamma \ \ \qquad \\ 
\downarrow \pi_{u}\circ F & & \downarrow \pi_{u}\circ F \\
\R^u\setminus D^u& \quad \xhookrightarrow{i_2} & \R^u \setminus\{Z_*\}
\enal. 
\ee
From [C2] and using the fact that $B^s$ and $\Omega^c$ are contractible, $\pi_u \circ F|\Gamma$ is a homotopy equivalence. Note that  
$i_2$ is a homotopy equivalences, and  $\pi_u \circ F | \Gamma$ is 
a homeomorphism onto its image. Let $h$ and $g$ be the homotopy inverses of 
$\pi_u \circ F | \partial \Gamma$ and $i_2$, then $h \circ g \circ (\pi_u \circ F) $ defines a homotopy inverse of $i_1$. As a result $\Gamma$ is homotopic to $\partial \Gamma$, this is a contradiction. 
\end{proof}

Proposition~\ref{prop:lipschitz-manifold} follows from the next statement.
\begin{proposition}
The mapping $\pi_{sc}: W^{sc}(F)\to D^{sc}$ is one-to-one and onto, therefore it is the graph of a function $w^{sc}$. Moreover $w^{sc}$ is Lipshitz and 
\[
	T_Z W^{sc}(F) \in (C_\mu^{u}(Z))^c = C_{\mu^{-1}}^{sc}(Z), \quad Z \in W^{sc}(F). 
\] 
\end{proposition}
\begin{proof}
For each $X \in D^{sc}$, we define $\Gamma_X = (\pi_{sc})^{-1} X$, clearly $\Gamma_X \in \cV$. We first show $\Gamma_X \cap W^{sc}(F)$ is nonempty and consists of a single point. Assume first that $\Gamma_X \cap W^{sc}(F)$ is empty. Then by definition of $W^{sc}(F)$, there is $n \in \N$ such that $F^n(\Gamma_X) \cap D = \varnothing$. However, by Lemma~\ref{lem:inv-unstable-disk}, $\bigcap_{i=1}^n F^i(\Gamma_X) \cap D \in \cV$ is always nonempty, a contradiction. We now consider two points $Z_1, Z_2 \in W^{sc}(F)$ with $\pi_u Z_1 = \pi_u Z_2$. Note that $F^k(Z_1), F^k(Z_2) \in D$ for all $k \ge 0$, and $Z_2 \in K_\mu^u(Z_1)$, by [C4] we have 
\[
2\ge \|\pi_u(F^k(Z_1)-F^k(Z_2))\|\ge \nu^k\|\pi_u(Z_1-Z_2)\| 
\]
for all $k$, which implies $Z_1=Z_2$.

The last argument actually shows $Z_2 \notin K_\mu^u(Z_1)$ for all $Z_1, Z_2 \in W^{sc}(F)$. For any $\epsilon>0$, for $Z_1=(X_1, Y_1), Z_2=(X_2, Y_2) \in W^{sc}(F)$ with $\dist(X_1, X_2)$ small, we have  $\|Y_1 - Y_2\| \le \mu^{-\frac12}  \dist(X_1, X_2)$. This implies both the Lipshitz and the cone properties in our proposition.
\end{proof}

\subsection{NHIC for the dominant system}

We prove Theorem~\ref{thm:nhic-persist} in this section. First, an overview of notations.
\begin{enumerate}
 \item The strong Hamiltonian is $H^\st = \cH^\st(p_0, \cB^\st, U^\st)$ defined on $\T^m\times \R^m$,
  and its associated Lagrangian vector field is $X^\st$ (see \eqref{eq:Xst}). We call the 
  We denote the time-$1$-map of $X^\st$ by $G^\st_0$ and we will lift it to the universal cover $\R^m \times \R^m$ without changing its name. 
 \item The vector field $X^\st$ is extended trivially to $(\T^m \times \R^m)\times (\T^{d-m}\times \R^{d-m})$ (see \eqref{eq:XstL}). The time-$1$-map is denoted $G_0$, and we have $G_0 = G_0^\st \times \Id$. We will also lift it to the universal cover with the same name. 
 \item The slow Hamiltonian is $H^s = \cH^s(\cB^\st, \cB^\wk, p_0, U^\st, \cU^\wk)$,  and consider its Lagrangian vector field $X_{Lag}^s$. We apply a coordinate change  $(\varphi^\st, v^\st, \varphi^\wk, v^\wk) = \Phi(\varphi^\st, v^\st, \varphi^\wk, I^\wk)$ as in
   \eqref{eq:half-lag}, and a rescaling $\Phi_\Sigma$ as defined in \eqref{eq:rescaling}. The new vector field is denoted $\tilde{X}^s$ (see \eqref{eq:Xs}, \eqref{eq:rescaling}). We denote its time-$1$-map $G$, which is considered a map on the Euclidean space $\R^m \times \R^m \times \R^{d-m} \times \R^{d-m}$. 
\end{enumerate}
By Theorem~\ref{thm:resc-est}, we have:
\begin{corollary}\label{cor:G-G0}
	Assume that $(\cB^\wk, p_0, U^\st, \cU^\wk) \in \Omega_{\kappa,q}^{m,d}(\cB^\st)$, then for any $\delta_1 >0$, there exists $M>0$ such that for all $(\cB^\wk, p_0, U^\st, \cU^\wk)$ with $\lM(\cB^\wk)>M$, uniformly on $\R^m\times \R^m \times\R^{d-m} \times \R^{d-m}$, we have
	\[
		\|\Pi_{(\varphi^\st, v^\st)}(G - G_0)\|< \delta_1, \quad \|DG - DG_0\| < \delta_1. 
	\]
\end{corollary}

By assumption, the Hamiltonian flow $H^\st$ admits an NHIC $\chi^\st(\T^l \times B_{1+a}^l)$, where $\chi^\st$ is an embedding. Therefore $G_0^\st$ admits an NHIC $\Lambda_a^\st = \Phi^{-1} \circ (\T^l \times B_{1+a}^l)$ with the exponents $\alpha, \beta$. We use local coordinates in a tubular neighborhood to simplify the setting. 
\begin{lemma}\label{lem:F0-coord}
There exists a tubular neighborhood $N(\Lambda^\st_a) \subset \T^m\times \R^m$ of $\Lambda_a^\st$ and a diffeomorphism 
\[
	h^\st: B_1^l \times B_1^l \times (\T^l \times B_{1+a}^l) \to N(\Lambda^\st_a)
\] 
such that:
\begin{enumerate}
 \item $h^\st(0,0,z) = \chi^\st(z)$, in particular $h^\st(\cC_a^\st) := h^\st(\{0\}\times \{0\} \times (\T^l \times B_{1+a}^l)) = \Lambda_a^\st$.  
 \item For the map $F_0^\st := (h^\st) \circ G_0^\st \circ (h^\st)^{-1}$:
 \begin{enumerate}
  \item $\cC_a^\st$ is an NHIC for $F_0^\st$ with the same exponents $\alpha, \beta$. 
  \item The associated stable/unstable bundles take the form
  \[
  E^s = \R^l \times \{0\} \times \{0\}, \quad E^u = \{0\}\times \R^l \times \{0\}. 
  \]
  In particular, $DF_0^\st$ is a block diagonal matrix in the blocks corresponding to the three components. 
  \item Let $g_0$ denote the Euclidean metric. Then there exists a Riemannian metric $g$ on $\T^l \times B^l_{1+a}$ such that the tensor metric $g_0 \otimes g_0 \otimes g$ on $B_1^l \times B_1^l \times (\T^l \times B^l_{1+a})$ is an adapted metric for the NHIC $\cC_a^\st$. 
 \end{enumerate}
\end{enumerate}
\end{lemma}
\begin{proof}
We use the bundles $E^u$, $E^s$, and the parametrization $\chi^\st$ of $\Lambda_a^\st$ to build a coordinate system for the normal bundle to $\Lambda_a^\st$, which is diffeomorphic to the tubular neighborhood. We then pull back the adapted metric of $\Lambda_a^\st$ using this map to $\cC_a^\st$. 
\end{proof}

Denote $\Omega^\wk = \R^{d-m}\times \R^{d-m}$ and consider the trivial extension 
\[
	h: B_1^l \times B_1^l \times ((\T^l \times B_{1+a}^l) \times \Omega^\wk) \to  N(\Lambda_a^\st) \times \Omega^\wk
\]
by $h(x,y, (z^\st, z^\wk)) = (h^\st(x, y, z^\st), z^\wk)$. Define the following maps
\begin{equation}\label{eq:def-F0F}
  	F_0 = h^{-1} \circ G_0 \circ h = (F_0^\st, \Id) , \quad F = h^{-1} \circ G \circ h. 
\end{equation}

Finally, to apply Theorem~\ref{thm:center-c1}, we denote $\Omega_a^\st = \R^l \times B_{1+a}^l$ which is the universal cover of $\T^l \times B_{1+a}^l$. We lift the maps $F_0, F$ to the covering space without changing their names, namely
\[
	F, F_0: B_1^l \times B_1^l \times (\Omega_{1+a}^\st\times \Omega^\wk) \circlearrowleft. 
\]
$\Omega_{1+a}^\st\times \Omega^\wk$ is our center component and is denoted $\Omega$. While the maps are defined on unbounded regions, we keep in mind that $F_0 = (F_0^\st, \Id)$ where $F_0^\st$ is defined on a compact set $B_1^l \times B_1^l \times (\T^l \times B_{1+a}^l)$.

We still need one reduction to apply Theorem~\ref{thm:center-c1}. Recall that $\Omega_a^\st = \R^l \times B_{1+a}^l$. Write $F_0 = (F_0^x, F_0^y, F_0^z)$, define
\begin{equation}\label{eq:def-L}
	L(x,y,z) = (D_xF_0^x(0,0,z)\cdot x ,  D_yF_0^y(0,0,y) \cdot y, F_0^z(0,0,z))
\end{equation}
this is the linearized map at $(0,0,z)$ (we used $F_0(0,0,z) = (0,0, F_0^z)$, and  $DF_0$ is block diagonal from Lemma~\ref{lem:F0-coord}). since $F_0= (F_0^\st, \Id)$ and $F_0^\st$ is defined over a compact set, we obtain as $r \to 0$,
\begin{equation}\label{eq:L-F0}
  	 	\|L - F_0\| = o(r), \quad \|DL - DF_0\| = o(1) \text{ on } B_r^l \times B_r^l \times (\Omega_a^\st \times \Omega^\wk).
\end{equation}

Moreover, since $F_0$ preserves $\{0\}\times \{0\} \times \partial \Omega$, we get 
\begin{equation}\label{eq:fix-boundary}
  	L (B_1^l \times B_1^l \times \partial \Omega )\subset  \R^l \times \R^l \times \partial \Omega. 
\end{equation}
Namely, the linearized map $L$ preserves the boundary of the center component. Finally, we modify the map $F$ so that it also fixes the center boundary. 
 Let $\rho$ be a standard mollifier satisfying
\[
	\begin{cases}
	\rho(x,y, (z^\st, z^\wk)) =  \rho(z^\st) = 0 & z^\st \in \Omega_0^\st \\
	\rho(x, y, (z^\st, z^\wk)) = \rho(z^\st) =  1 & z^\st \in \Omega_a^\st \setminus \Omega_{a/2}^\st. 
	\end{cases}
\]
Let
\begin{equation}\label{eq:tildeF}
  	\tilde{F} = F(1-\rho) + L \rho ,
\end{equation}
we have:
\begin{lemma}\label{lem:F-iso}
For any $\mu>1$, $\epsilon>0$ and $r_0 >0$, there exists $\delta_1 >0$ and $0<r<r_0$ such that if $G$ and $G_0$ satisfies
	\[
		\|\Pi_{(\varphi^\st, v^\st)}(G - G_0)\|< \delta_1, \quad \|DG - DG_0\| < \delta_1,
	\]
  the map $\tilde{F}$  defined by \eqref{eq:def-F0F}, \eqref{eq:def-L} and \eqref{eq:tildeF} satisfies conditions [C1]-[C4] with the parameters $\mu$ and $\nu = \alpha^{-1} - \epsilon$ on $B_r^l \times B_r^l \times \Omega$. The same hold for the map $\inv(\tilde{F})$. 
\end{lemma}
\begin{proof}
First of all, from Lemma~\ref{lem:F0-coord}, $DF_0(0,0,z) = \diag\{ D_xF_0^x, D_y F_0^y, D_zF_0^z\}$ with 
\begin{equation}\label{eq:F0-hyper}
  	\|D_xF_0^x\|, \|(D_y F_0^y)^{-1}\|^{-1} \le \alpha, \quad \|D_zF_0^z\|, \|(D_zF_0^z)^{-1}\|  \ge \beta. 
\end{equation}
Recall that $F_0 = (F_0^\st, \Id)$ where $F_0^\st$ is defined over a compact set. Therefore for sufficiently small $r>0$, we have 
\[
	\|\Pi_x F_0(x,y,z)\| \le (\alpha + \epsilon) \|x\|, \quad \|\Pi_y DF_0(x,y,z)\| \ge (\alpha + \epsilon)^{-1} \|y\|,
\]
hence
\[
	\Pi_xF_0(B_r^l \times B_r^l \times \Omega) \subset B_{(\alpha +\epsilon) r}^l, \quad \|\Pi_yF_0(B_r^l \times \partial B_r^l \times \Omega)\|  \ge (\alpha + \epsilon)^{-1} r .
\]
Since $\|\tilde{F} - F_0\| \le \|(1-\rho)(F-F_0)\| + \|\rho(L-F_0)\|$,  $\|\Pi_{(x,y)}(F- F_0)\| \le C \|\Pi_{(\varphi^\st, v^\st)}(G - G_0)\| \le C \delta_1$, and $\|L - F_0\| =o(r)$, by choosing $\delta_1,r$ small enough, we get 
\[
	\Pi_x\tilde{F}(B_r^l \times B_r^l \times \Omega) \subset B_r^l, \quad \|\Pi_y \tilde{F} (B_r^l \times \partial B_r^l \times \Omega)\|  > r .
\]
The first half of the above formula combined with \eqref{eq:fix-boundary} gives [C1], and the second half gives [C2]. 

We now prove the cone conditions [C3] and [C4]. We first show the map $\tilde{F}$ is well approximated by the linearized map $DF_0(0,0,z)$. 
Given any $\epsilon>0$, we use Corollary~\ref{cor:G-G0} to choose $\delta_1$ so small such that 
\[
\|D(F- F_0)\| + \|\Pi_{z^\st}(F-F_0)\| \|d\rho\| \le C\|D(G- G_0)\| + \|\Pi_{\varphi^\st, v^\st}(G-G_0)\| \|d\rho\| < \epsilon/2	
\]
 By \eqref{eq:L-F0}, we can choose $r_1$ such that for $0< r < r_1$, $\|D(F_0-L)\|+\|F_0-L\|\|d\rho\| < \epsilon/2$ on $B_r^l \times B_r^l \times (\Omega_a^\st \times \Omega^\wk)$. Then from $\tilde{F} = F + (L-F)\rho$, and the fact that $\rho$ depends only on $z^\st$ gives 
\begin{multline*}
	\|D\tilde{F}(x,y,z) - DL(x,y,z)\|  \le \|DF - DL\|  + \|\Pi_{z^\st}(F-L)\|\|d\rho\|  \\
	\le \|D(F-F_0)\| + \|D(F_0 - L)\| + (\|\Pi_{z^\st}(F-F_0)\| + \|\Pi_{z^\st}(F_0- L)\|)\|d \rho\| < \epsilon.
\end{multline*}

Consider $(x_1, y_1, z_1), (x_2, y_2, z_2) \in B_r^l \times B_r^l \times \Omega$, denote $(\Delta x, \Delta y, \Delta z) = (x_2, y_2, z_2) - (x_1,y_1,z_1)$ and $d = \|\Delta x\| + \|\Delta y\| + \dist(z_1,z_2)$. For $d$ small enough
\begin{align*}
	& \|\tilde{F}(x_2, y_2, z_2) - \tilde{F}(x_1, y_1, z_1) - DF_0(0,0,z)(\Delta x, \Delta y, \Delta z) \| \\
	&= \|L(x_2, y_2, z_2) - L(x_1, y_1, z_1)  + (\tilde{F}-L)(x_2,y_2, z_2) - (\tilde{F}-L)(x_1,y_1, z_1)
	\\ &-  DF_0(0,0,z_1)(\Delta x, \Delta y, \Delta z) \| \\
	& \le   \|D(\tilde{F}-L)(x_1, y_1, z_1)\|d \le \epsilon d.
\end{align*}

To prove [C3], we first show the linear map preserves the unstable cone. 
For any $\mu>1$ and   $(v_x, v_y, v_z) \in T_{(x_1, y_1, z_1)}B_r^l \times B_r^l \times \Omega$ with $\mu \|v^y\|^2 \ge \|v^z\|^2 + \|v^x\|^2$, let $(v_x', v_y', v_z') = DF_0(0,0,z) (v_x, v_y, v_z)$, we have
\[
	\mu \|v_y'\|^2 \ge \mu \alpha^{-1} \|v_y\|^2 \ge \alpha^{-1} (\|v_x\|^2 + \|v_z\|^2) \ge \alpha^{-1}(\|v_x'\|^2 + \beta\|v_z'\|^2) \ge \frac{\beta}{\alpha}(\|v_x'\|^2 + \|v_z'\|^2).
\]
In other words, for any $\mu>1$, we have $ DF_0(0,0,z) C_\mu^u  \subset C_{\alpha\mu/\beta}^u$.

Coming to the non-linear map $\tilde{F}$, for $(x_1, y_1, z_1), (x_2, y_2, z_2) \in B_r^l \times B_r^l \times \Omega$, let $(x_i', y_i', z_i') = \tilde{F}(x_i, y_i, z_i)$,  and $(\Delta x, \Delta y, \Delta z)$ , $(\Delta x', \Delta y', \Delta z') $ be the corresponding difference. If $(x_2, y_2, z_2) \in K_\mu^u(x_1, y_1, z_1)$, then  $\mu\|\Delta y\|^2 \ge \|\Delta x\|^2 + \dist(z_1,z_2)^2$. In particular $\dist(z_1,z_2)^2 \le \mu \|\Delta y\|^2 \le \mu r^2$. When $r$ is small enough $\|(\Delta x', \Delta y', \Delta z') - DF_0(0,0,z_1)(\Delta x, \Delta y, \Delta z) \| \le \epsilon d $.
Furthermore assume $r$ is so small that $(1-\epsilon) \dist(z,z') \le \|\Delta z\|_{(0,0,z_1)} \le (1+\epsilon) \dist(z,z')$, where $\|\Delta z\|_{(0,0,z_1)}$ is measured using the local Riemannian metric. We drop the subscript from now on. Using the linear calculation, there exists a uniform constant $C>1$ such that  
\begin{multline*}
\mu \|\Delta y'\|^2 \ge \frac{\beta}{\alpha}(\|\Delta x'\|^2 + \|\Delta z'\|^2) - C \epsilon^2 d^2 \ge \frac{\beta}{\alpha}(\|\Delta x'\|^2 + \|\Delta z'\|^2) - C (1+\mu)\epsilon^2 \|\Delta y\|^2 \\
\ge  (1-\epsilon)\frac{\beta}{\alpha}(\|\Delta x'\|^2 + \dist(z_1', z_2')^2) - C^2(1+\mu)\epsilon^2 \|\Delta y'\|^2.
\end{multline*}
noting that  $\|D\tilde{F}^{-1}\|, \|D\tilde{F}\|$ are uniformly bounded. When $\epsilon$ is small enough we get $\mu \|\Delta y\|^2 \ge \|\Delta x\|^2 + \dist(z_1', z_2')^2$. [C3] is proven. 

[C4] follows directly from $\|(\Delta x', \Delta y', \Delta z') - DF_0(0,0,z_1)(\Delta x, \Delta y, \Delta z) \| \le \epsilon d $  and \eqref{eq:F0-hyper}. The proof for $\inv(\tilde{F})$ is identical and is omitted. 
\end{proof}


\begin{proof}[Proof of Theorem~\ref{thm:nhic-persist}]
For any $\delta>0$, we choose $0 < r < \delta/C$ and $\mu>1$, where $C$ is a constant specified later. Apply  Lemma \ref{lem:F-iso}, there exists $M>0$ such that whenever $\lM(\cB^\wk)>M$, the map $\tilde{F}$ associated to $H^s(\cB^\st, \cB^\wk, p_0, U^\st, \cU^\wk)$ satisfies [C1]-[C4] on $B_r^l \times B_r^l \times (\Omega_a^\st \times \Omega^\wk)$. As a result, we obtain a function $w^c: \Omega_a^\st \times \Omega^\wk \to B_r^{m-l} \times B_r^{m-l}$ such that 
\[
	W^c = \graph(w^c) = \{(x, y, (z^\st, z^\wk)): (x,y) = w^c(z^\st, z^\wk)\}
\]
is invariant under $\tilde{F}$, and is the maximally invariant set on $B_r^{m-l} \times B_r^{m-l} \times \Omega_a^\st \times \Omega^\wk$. Since $\tilde{F} = F$ on whenever $z^\st \in \Omega_0^\st$, any $F$ invariant set with $z^\st \in \Omega_0^\st$ is also $\tilde{F}$ invariant and hence is contained in $W^c$. We now consider the map
\[
	\zeta: (z^\st, z^\wk) \mapsto h(w^c(z^\st, z^\wk), z^\st, z^\wk),
\]
then $\graph(\zeta)$ is an $F-$invariant set. Finally we invert the coordinate changes \eqref{eq:half-lag} and \eqref{eq:rescaling} to obtain the desired embedding $\eta^s = \Phi \circ \Phi_\Sigma \circ(\zeta, \Id)$. 

 Moreover, we have $\|w^c\|_{C^0} \le r$, using the fact that $h^\st(0, 0, z) = \chi^\st(z)$, and that $\Phi$ does not change the strong component,  there is $C>0$ such that $\|\zeta - \Phi^{-1} \circ \chi^\st\|_{C^0}\le C'r$. Finally, since the rescaling $\Phi_\Sigma$ do not change the strong component, there exists $C>0$ depending only on $\Phi$ such that 
 \[
 \|\Pi_{(\varphi^\st, I^\st)}\eta^s - \chi^\st\| = \|\Phi ( \zeta - \Phi^{-1} \circ \chi^\st)\| \le Cr <\delta. 
 \]
 We choose the open set $V = \Phi\circ h^\st (B_r^{m-l} \times B_r^{m-l} \times \Omega_0^\st)$, then any invariant set in $V \times \Omega^\wk$ must be contained in $\eta^s(\Omega^\st_0 \times \Omega^\wk)$. This concludes the proof. 
\end{proof}

\subsection*{Acknowledgments} The first author acknowledges NSF 
for partial support grant DMS-5237860. The authors would 
like to thank John Mather, Marcel Guardia, and Abed 
Bounemoura for useful conversations. 

\printbibliography

\end{document}

%% file: ess-resonance-3df.pdf_tex
\begingroup%
  \makeatletter%
  \providecommand\color[2][]{%
    \errmessage{(Inkscape) Color is used for the text in Inkscape, but the package 'color.sty' is not loaded}%
    \renewcommand\color[2][]{}%
  }%
  \providecommand\transparent[1]{%
    \errmessage{(Inkscape) Transparency is used (non-zero) for the text in Inkscape, but the package 'transparent.sty' is not loaded}%
    \renewcommand\transparent[1]{}%
  }%
  \providecommand\rotatebox[2]{#2}%
  \ifx\svgwidth\undefined%
    \setlength{\unitlength}{174.35266113bp}%
    \ifx\svgscale\undefined%
      \relax%
    \else%
      \setlength{\unitlength}{\unitlength * \real{\svgscale}}%
    \fi%
  \else%
    \setlength{\unitlength}{\svgwidth}%
  \fi%
  \global\let\svgwidth\undefined%
  \global\let\svgscale\undefined%
  \makeatother%
  \begin{picture}(1,0.77142499)%
    \put(0,0){\includegraphics[width=\unitlength]{ess-resonance-3df.pdf}}%
    \put(-0.0029036,0.44294173){\color[rgb]{0,0,0}\makebox(0,0)[lb]{\smash{$\Gamma$}}}%
    \put(0.64035715,0.37509093){\color[rgb]{0,0,0}\makebox(0,0)[lb]{\smash{$\Gamma_{k'}$}}}%
  \end{picture}%
\endgroup%

%% file: first-nondeg.pdf_tex
\begingroup%
  \makeatletter%
  \providecommand\color[2][]{%
    \errmessage{(Inkscape) Color is used for the text in Inkscape, but the package 'color.sty' is not loaded}%
    \renewcommand\color[2][]{}%
  }%
  \providecommand\transparent[1]{%
    \errmessage{(Inkscape) Transparency is used (non-zero) for the text in Inkscape, but the package 'transparent.sty' is not loaded}%
    \renewcommand\transparent[1]{}%
  }%
  \providecommand\rotatebox[2]{#2}%
  \ifx\svgwidth\undefined%
    \setlength{\unitlength}{288.8bp}%
    \ifx\svgscale\undefined%
      \relax%
    \else%
      \setlength{\unitlength}{\unitlength * \real{\svgscale}}%
    \fi%
  \else%
    \setlength{\unitlength}{\svgwidth}%
  \fi%
  \global\let\svgwidth\undefined%
  \global\let\svgscale\undefined%
  \makeatother%
  \begin{picture}(1,0.82591759)%
    \put(0,0){\includegraphics[width=\unitlength]{first-nondeg.pdf}}%
    \put(0.64105994,0.17187373){\color[rgb]{0,0,0}\makebox(0,0)[lb]{\smash{$\Gamma_{\Lbi{1}}$}}}%
    \put(0.40069741,0.45156839){\color[rgb]{0,0,0}\makebox(0,0)[lb]{\smash{$\cY_1^{\lambda}$}}}%
  \end{picture}%
\endgroup%

%% file: second-nondeg.pdf_tex
\begingroup%
  \makeatletter%
  \providecommand\color[2][]{%
    \errmessage{(Inkscape) Color is used for the text in Inkscape, but the package 'color.sty' is not loaded}%
    \renewcommand\color[2][]{}%
  }%
  \providecommand\transparent[1]{%
    \errmessage{(Inkscape) Transparency is used (non-zero) for the text in Inkscape, but the package 'transparent.sty' is not loaded}%
    \renewcommand\transparent[1]{}%
  }%
  \providecommand\rotatebox[2]{#2}%
  \ifx\svgwidth\undefined%
    \setlength{\unitlength}{601.825bp}%
    \ifx\svgscale\undefined%
      \relax%
    \else%
      \setlength{\unitlength}{\unitlength * \real{\svgscale}}%
    \fi%
  \else%
    \setlength{\unitlength}{\svgwidth}%
  \fi%
  \global\let\svgwidth\undefined%
  \global\let\svgscale\undefined%
  \makeatother%
  \begin{picture}(1,0.40150109)%
    \put(0,0){\includegraphics[width=\unitlength]{second-nondeg.pdf}}%
    \put(0.30764435,0.08246836){\color[rgb]{0,0,0}\makebox(0,0)[lb]{\smash{$\Gamma_{\Lbi{1}}$}}}%
    \put(0.16363953,0.23835709){\color[rgb]{0,0,0}\makebox(0,0)[lb]{\smash{$\cY_1^{\lambda}$}}}%
    \put(0.82773962,0.08246836){\color[rgb]{0,0,0}\makebox(0,0)[lb]{\smash{$\Gamma_{\Lbi{1}}$}}}%
    \put(0.68373484,0.23835709){\color[rgb]{0,0,0}\makebox(0,0)[lb]{\smash{$\cN_1$}}}%
    \put(0.37894779,0.39005154){\color[rgb]{0,0,0}\makebox(0,0)[lb]{\smash{$\Gamma_{\Sgi{2}}$}}}%
  \end{picture}%
\endgroup%

%% file: 3-resonance-hier.pdf_tex
\begingroup%
  \makeatletter%
  \providecommand\color[2][]{%
    \errmessage{(Inkscape) Color is used for the text in Inkscape, but the package 'color.sty' is not loaded}%
    \renewcommand\color[2][]{}%
  }%
  \providecommand\transparent[1]{%
    \errmessage{(Inkscape) Transparency is used (non-zero) for the text in Inkscape, but the package 'transparent.sty' is not loaded}%
    \renewcommand\transparent[1]{}%
  }%
  \providecommand\rotatebox[2]{#2}%
  \ifx\svgwidth\undefined%
    \setlength{\unitlength}{288.825bp}%
    \ifx\svgscale\undefined%
      \relax%
    \else%
      \setlength{\unitlength}{\unitlength * \real{\svgscale}}%
    \fi%
  \else%
    \setlength{\unitlength}{\svgwidth}%
  \fi%
  \global\let\svgwidth\undefined%
  \global\let\svgscale\undefined%
  \makeatother%
  \begin{picture}(1,0.8258461)%
    \put(0,0){\includegraphics[width=\unitlength]{3-resonance-hier.pdf}}%
    \put(0.64105711,0.17185184){\color[rgb]{0,0,0}\makebox(0,0)[lb]{\smash{$\Gamma_{\Lbi{1}}$}}}%
    \put(0.62736182,0.71324592){\color[rgb]{0,0,0}\makebox(0,0)[lb]{\smash{$\cPi{2}$}}}%
  \end{picture}%
\endgroup%